\documentclass[11pt]{amsart}

\usepackage[utf8]{inputenc}
\usepackage[T1]{fontenc}   

\usepackage{amssymb}
\usepackage{amsmath}
\usepackage{amsthm}

\usepackage{graphicx}

\usepackage[all]{xy}
\usepackage{stmaryrd}
\usepackage{amscd}
\usepackage{bbm}
\usepackage{enumitem}
\usepackage{fncylab}
\usepackage{xcolor}
\usepackage{hyperref}
\usepackage{leftidx}
\usepackage{calc}
\usepackage[width=\textwidth]{caption}
\usepackage[normalem]{ulem}
\usepackage{mathtools}

\definecolor{boubelcolor}{rgb}{.65,0.05,0}

\DeclareUnicodeCharacter{2260}{\colorlet{memoireboubel}{.}\color{boubelcolor}} % Le caract\`ere ≠
\DeclareUnicodeCharacter{00B1}{\color{memoireboubel}{}} % Le caract\`ere ±  

\DeclareUnicodeCharacter{00A3}{\colorlet{memoirejuillet}{.}\color{blue}} % Le caract\`ere £
\DeclareUnicodeCharacter{00A7}{\color{memoirejuillet}{}} % Le caract\`ere §

\theoremstyle{plain}
\newtheorem{themBJ}{Theorem}

\newtheorem{them}{Theorem}[section]
\newtheorem{pro}[them]{Proposition}
\newtheorem{pronotation}[them]{Proposition/Notation}
\newtheorem{prodefinition}[them]{Proposition/Definition}
\newtheorem{lemnotation}[them]{Lemma/Notation}
\newtheorem{cor}[them]{Corollary}
\newtheorem{lem}[them]{Lemma}

\theoremstyle{definition}
\newtheorem{vocabulary}[them]{Vocabulary}

\newtheorem{defi}[them]{Definition}
\newtheorem{convention}[them]{Convention}
\newtheorem{notation}[them]{Notation} %?
\newtheorem{remnotation}[them]{Remark/Notation} %?

\newtheorem{remdefi}[them]{Remark/Definition}
\newtheorem{reminder-notation}[them]{Reminder/Notation}
\newtheorem{problem}[them]{Problem}

\theoremstyle{remark}
\newtheorem{rem}[them]{Remark}
\newtheorem{ex}[them]{Example}    
\newtheorem{reminder}[them]{Reminder}  

\newcommand{\dd}{\mathrm{d}}

\newcommand{\id}{\mathrm{Id}}
\newcommand{\join}{\mathrm{Joint}}
\newcommand{\como}{\mathfrak{Q}}
\newcommand{\mq}{\mathfrak{MQ}}
\newcommand{\lev}{\mathrm{Lev}}
\newcommand{\ML}{\mathrm{LK}}
\newcommand{\IT}{\mathrm{IK}}
\newcommand{\disp}{\mathfrak{Disp}}

\newcommand{\vol}{\operatorname{Vol}}
\newcommand{\parti}{\operatorname{Part}}
\newcommand{\idk}{\operatorname{id}}

\newcommand{\espace}{E}
\newcommand{\subd}{R} 
\newcommand{\Tt}{{\mathcal T}}
\newcommand{\XX}{{\mathcal X}}
\newcommand{\EE}{{\mathcal E}}
\newcommand{\A}{{\mathcal A}}
\newcommand{\Var}{{\mathcal L}}
\newcommand{\Id}{\operatorname{Id}}
\newcommand{\one}{\mathbbm{1}}
\newcommand{\mb}{\mathversion{bold}}
\newcommand{\mn}{\mathversion{normal}}

\newcommand{\trans}{\leftidx{^{\rm t}}}
\renewcommand{\leq}{\leqslant} 
\renewcommand{\geq}{\geqslant} 
\newcommand{\eqdef}{\coloneqq}
\newcommand{\eqdefup}{\coloneqq} 
\newcommand{\espcdot}{\mbox{$\,\cdot\,$}}

\newcommand\longoverline{\bgroup\markoverwith
{{\rule[1.5ex]{2pt}{0.4pt}}}\ULon}

\newcounter{point}
\newcounter{souspoint}[point]
\renewcommand{\thepoint}{\alph{point}}
\renewcommand{\thesouspoint}{\roman{souspoint}}
\labelformat{point}{(#1)}          % cette commande fait appel à \usepackage{fncylab}
\labelformat{souspoint}{(#1)}  % cette commande fait appel à \usepackage{fncylab}
\newcommand{\point}{\refstepcounter{point}{\bf(\thepoint)}}
\newcommand{\souspoint}{\refstepcounter{souspoint}{\bf(\thesouspoint)}}
\newcommand{\maz}{\setcounter{point}{0}\setcounter{souspoint}{0}}

\newcommand{\eps}{\varepsilon}

\newcommand{\Lg}{\lambda|_{[0,1]}}
\newcommand{\la}{\lambda}
\renewcommand{\P}{\mathbb{P}}
\newcommand{\p}{\mathcal{P}}
\newcommand{\m}{\mathcal{M}}
\newcommand{\ma}{\mathrm{Marg}}
\newcommand{\C}{\mathcal{C}}

\newcommand{\R}{\mathbb{R}}
\newcommand{\Q}{\mathbb{Q}}
\newcommand{\N}{\mathbb{N}}
\newcommand{\Z}{\mathbb{Z}}

\newcommand{\AND}{\quad\text{and}\quad}

\newcommand{\E}{\mathbb{E}}
\newcommand{\law}{\operatorname{Law}}

\newcommand{\sto}{\mathrm{sto}}
\newcommand{\stosup}{\operatorname{stosup}}
\newcommand{\stoinf}{\operatorname{stoinf}}
\newcommand{\lcsup}{\operatorname{losup}}
\newcommand{\lcinf}{\operatorname{loinf}}

\newcommand{\leqc}{\preceq_{C}}
\newcommand{\leqcs}{\preceq_{C,\sto}} 
\newcommand{\leqs}{\preceq_{\text{\rm sto}}}
\newcommand{\geqs}{\succeq_{\text{\rm sto}}}
\newcommand{\e}{\mathrm{e}}

\newcommand{\pr}{\operatorname{proj}}

\newcommand{\pdec}{\mathcal{P}^{\scriptscriptstyle\searrow}}
\newcommand{\mdec}{\mathcal{M}^{\scriptscriptstyle\searrow}}

\newcommand{\op}{\mathopen}
\newcommand{\clo}{\mathclose}
\newcommand{\leqlc}{\preceq_{\rm lo}}
\newcommand{\geqlc}{\succeq_{\rm lo}}
\newcommand{\rsemibracket}{\mathclose{\rfloor\hspace*{-.9ex}\lceil}}

\newcommand{\semi}{\mathclose{\rfloor\hspace*{-.8ex}\lceil}}

\newlength{\arraycolsepsauvegardegenerale}
\newlength{\arraycolsepsauvegardetemporaire}

\newlength{\tabcolsepsauvegardegenerale}
\newlength{\tabcolsepsauvegardetemporaire}

\newenvironment{nnarray}%
{\setlength{\arraycolsepsauvegardetemporaire}{\arraycolsep}\setlength{\arraycolsep}{0cm}\begin{array}}
{\end{array}\setlength{\arraycolsep}{\arraycolsepsauvegardetemporaire}}

{\setlength{\tabcolsepsauvegardetemporaire}{\tabcolsep}\setlength{\tabcolsep}{0cm}\begin{tabular}}
{\end{tabular}\setlength{\tabcolsep}{\tabcolsepsauvegardetemporaire}}

\let\cal\mathcal

\title[The Markov-quantile process]{The Markov-quantile process attached to a family of marginals}
\author{Charles Boubel and Nicolas Juillet}

\address{Institut de Recherche Math\'ematique Avanc\'ee, UMR 7501, Universit\'e de Strasbourg et CNRS, 7 rue Ren\'e Descartes, 67\,000 Strasbourg, France}
\email{charles.boubel@unistra.fr, nicolas.juillet@unistra.fr}

\begin{document}

\begin{abstract}
Let $\mu=(\mu_t)_{t\in \R}$ be any 1-parameter family of probability measures on $\R$. Its quantile process $(G_t)_{t\in\R}:\op]0,1\clo[\to\R^\R$, given by $G_t(\alpha)=\inf\{x\in\R:\mu_t(\op]-\infty,x])\geqslant\alpha\}$, is not Markov in general. We modify it to build the Markov process we call ``Markov-quantile''.

We first describe the discrete analogue: if $(\mu_n)_{n\in \Z}$ is a family of probability measures on $\R$, a Markov process $Y=(Y_n)_{n\in\Z}$ such that $\law(Y_n)=\mu_n$ is given by the data of its couplings from $n$ to $n+1$, i.e.\ $\law((Y_n,Y_{n+1}))$, and the process $Y$ is the inhomogeneous Markov chain having those couplings as transitions. Therefore, there is a canonical Markov process with marginals $\mu_n$ and as similar as possible to the quantile process: the chain whose transitions are the quantile couplings.

We show that an analogous process exists for a continuous parameter $t$: there is a \emph{unique} Markov process $X$ with the measures $\mu_t$ as marginals, and being a limit for the finite dimensional topology of quantile processes where the past is made independent of the future at finitely many times (many non-Markovian limits exist in general). The striking fact is that the construction requires no regularity for the family $\mu$. We rely on order arguments, which seems to be completely new for the purpose.

We also prove new results the Markov-quantile process yields in two contemporary frameworks:\smallskip 

-- In case $\mu$ is increasing for the stochastic order, $X$ has increasing trajectories. This is an analogue of a result of Kellerer dealing with the convex order, peacocks and martingales. Modifiying Kellerer's proof, we also prove simultaneously his result and ours in this case.\smallskip

-- If $\mu$ is absolutely continuous in Wasserstein space $\p_2(\R)$ then $X$ is solution of a Benamou--Brenier transport problem with marginals $\mu_t$. It provides a Markov probabilistic representation of the continuity equation, unique in a certain sense.\medskip

\end{abstract}
\maketitle

\section{Introduction}\label{sec:intro}

We prove four main theorems, stated and labelled as \ref{them:a}, \ref{them:b}, C and D in this introduction and proved in the order C, A, B, D in the paper ; see also their interdependence in Figure \ref{fig:un} p.\ \pageref{fig:un}.  
Theorem \ref{them:a} answers Problem \ref{problem} p.\ \pageref{problem} below, and is a general theoretical result in Probability Theory; it builds a certain stochastic process with given marginals: the {\em Markov-quantile} process. Theorem \ref{them:b} gives a convergence result to it. Theorems C and D present applications of the Markov-quantile process to two other contexts (Martingales and a theorem of Kellerer \cite{Ke72,Ke73}, and Optimal Transport), giving by the way Problem \ref{problem} additional motivations, see \S\ref{subsec:intro_kellerer}-\ref{subsec:intro_transport} and Figure \ref{fig:un}. We prove also Theorems \ref{them:sous_n} and \ref{them:commun_d_et_1}, linked with Theorem C and D respectively. Being a bit more technical they are not stated in this introduction.

In this introduction we first give the very necessary notions to state Theorems \ref{them:a}-\ref{them:b} as quickly as possible in \S\ref{subsec:resultats}, then state Theorems C and D in \S\ref{subsec:intro_kellerer}-\ref{subsec:intro_transport}. We slow the flow in \S\ref{subsec:intro_quantilemarkovien} to give a qualitative insight into Problem \ref{problem} and its difficulties, that also shows why it is a natural problem in itself. We give in \S\ref{subsec:organisation} the structure of the article and in \S\ref{subsec:notation} an index of our notation.

This paper treats of Measure, Probability, and Transport Theories. To be understood by a large panel of readers, we give the definitions of the more specific notions of each of these fields, or {\em Reminders} about them if needed.

\subsection{Our results and their motivation}\label{subsec:resultats}

Take $(\espace_\tau)_{\tau\in\Tt}$ a (finite or not) family of measurable spaces; $\prod_{\tau\in\Tt}\espace_\tau$ is endowed with its cylindrical $\sigma$-algebra, generated by the preimages of those of the factors by the projections.

\begin{reminder}\label{reminder:process}
A \emph{process} 
is a family $X=(X_\tau)_{\tau\in\Tt}$ of measurable maps $X_\tau:\Omega\to \espace_\tau$, called random variables, defined on the same probability space $(\Omega,\P)$. In this article, contrarily to what may be considered usual, no measurability condition is required on $\Omega\times \Tt$. For every $\Tt'\subset \Tt$, $(X_\tau)_{\tau\in\Tt'}$ defines a map $F_{\Tt'}$ from $\Omega$ to $\prod_{\tau\in\Tt'}\espace_\tau$. The \emph{law} of $(X_\tau)_{\tau\in\Tt'}$ is the pushed-forward probability measure $(F_{\Tt'})_\#\P$ on $\prod_{\tau\in\Tt'}\espace_\tau$, which is also called the {\em marginal law} of the measure $(F_{\Tt})_\#\P$ on $\prod_{\tau\in\Tt'}\espace_\tau$. \end{reminder}

Now let $\mu_\tau$ be a probability measure on $E_\tau$, for each $\tau$.

\begin{notation}\label{notation:marg} For all measurable space $E$, \mb$\m(E)$\mn\ and \mb$\p(E)$\mn\ are the spaces of measures and probability measures on $E$. If $\Tt'\subset\Tt$, \mb$\pr^{\Tt'}$\mn\ is the projection $\prod_{\tau\in\Tt}E_\tau\to\prod_{\tau\in\Tt'}E_\tau$; in case $\Tt=\{\tau_1,\ldots,\tau_m\}$ is finite, $\pr^{\tau_1,\ldots,\tau_m}$ means $\pr^{\{\tau_1,\ldots,\tau_m\}}$. When $P\in\p\bigl(\prod_{\tau\in\Tt}\espace_\tau\bigr)$ and $s<t$, $P^s$ stands  for $(\pr^{s})_\#P$ and $P^{s,t}$ for $(\pr^{s,t})_\#P$, and \mb$\ma((\mu_\tau)_{\tau\in\Tt})$\mn\ denotes $\{P\in\p\bigl(\prod_{\tau\in\Tt}\espace_\tau\bigr):\ \forall\tau\in\Tt,(\pr^\tau)_\#P=\mu_\tau\}$. When not otherwise specified, what we call the marginals of $P$ are its marginals $P^s$ on a single factor.

\end{notation}

\begin{reminder}\label{remind:processus_canonique}\maz \point\label{item:processus_canonique} If $P\in\ma((\mu_\tau)_{\tau\in\Tt})$, setting $\Omega\eqdefup(\prod_{\tau\in\Tt}E_\tau,P)$ and $X=(X_\tau)_{\tau\in \Tt}\eqdef(\pr^\tau)_{\tau\in \Tt}$ we get a process called the \emph{canonical process}, of law $P$. For this reason, by an abuse of language ---e.g.,  in the title of this article---, we sometimes call \emph{process} a probability measure on a product space. For the same reason we may also see $\ma((\mu_\tau)_{\tau\in\Tt})$ as the set of the processes $(X_\tau)_{\tau\in\Tt}$ such that $\law(X_\tau)=\mu_\tau$ for all $\tau$.

\point\ If $\sharp{\mathcal T}=2$, i.e.\ if $\mu\in\p(\espace)$ and $\nu\in\p(\espace')$, a measure $P\in\ma(\mu,\nu)$ is called a {\em transport (plan)} from $\mu$ to $\nu$, or a {\em coupling} between $\mu$ and $\nu$.
\end{reminder}

Here is our problem. It is stated for any family of measures, without any assumption of regularity in the parameter $t$.

\begin{problem}\label{problem}\maz If $\mu=(\mu_t)_{t\in\R}$ is a one-parameter family of probability measures on $\R$, we want to build a measure $\mq\in\ma(\mu)$ that at once:
\begin{itemize}
\item[\point\label{q:a}] is Markov,
\item[\point\label{q:b}] resembles as much as possible the quantile measure $\como\in\ma(\mu)$.\medskip
\end{itemize}
\end{problem}

Let us explain those two points. Take $P\in\ma((\mu_t)_{t\in\R})$ and $(X_{t})_{t\in\R}$ a process of law $P$; point \ref{q:a} means:
\begin{equation}\label{eq:def_markovien}
\forall s\in\R,\forall t>s, \law(X_t|\,(X_{u})_{u\leqslant s})=\law(X_t|\,X_s),
\end{equation}
where $\law(X_t|\,(X_u)_u)$ is the law of $X_t$ conditionally to the $\sigma$-algebra generated by the $X_u$. See also Definition \ref{defi:markov2}, where the Markov property is  introduced only through notions defined in this article.

For point \ref{q:b}, 
$\como\in\ma((\mu_t)_{t\in\R})$ may be defined by an explicit construction (see Reminder \ref{remind:quantile} below) or implicitly, as the unique measure such that for all $x\in\R$ and all $t>s$, if $X$ is a process of law $\como$:
\begin{equation}\label{eq:quantile_as_minimum}\law(X_{t}|\,X_{s}\leqslant x)=\min\{\law(Y_{t}|\,Y_{s}\leqslant x): \law(Y)\in\ma((\mu_t)_{t\in\R})\},
\end{equation}
where this minimum is with respect to the stochastic order:\medskip

\begin{reminder}\label{remind:sto} The \emph{stochastic order} on $\p(\R)$ is defined by: $\mu\leqs\nu$ if, for all $x\in\R$, $\mu(\op]-\infty,x])\geqslant\nu(\op]-\infty,x])$.
\end{reminder}

The problem is that $\como$ is not Markov in general, see Remark \ref{rem:como_markovien}\ref{p1:rem:como_markovien} a bit below. In view of its definition through \eqref{eq:quantile_as_minimum}, we seek some Markov process $\mq$, if it exists, such that:\smallskip

-- Processes of law $\mq$ satisfy a version of (\ref{eq:quantile_as_minimum}) {\em among Markov processes},\smallskip

-- like for $\como\in \ma((\mu_t)_{t\in \R})$, if $(\mu_t)_{t\in \R}$ is increasing for $\leqs$, some processes $(X_t)_{t\in\R}$ of law $\mq\in \ma((\mu_t)_{t\in \R})$ are increasing, i.e.\ the functions $t\to X_t(\omega)$ are,\smallskip

-- like for $\como$, the couplings $(\pr^{s,t})_\#\mq$ of $\mq$ have an  \emph{increasing kernel} (or briefly, $\mq$ has increasing kernels), as  follows (see Definition \ref{prodef:transitions} for alternative definitions not resorting to conditional laws):

\begin{defi}\label{defi:increasing_kernel_intro}Take $P\in\ma((\mu_t)_t)$ and for $s<t$, set $P^{s,t}=(\pr^{s,t})_\#P$. We call {\em kernel} of $P^{s,t}$ the data of the conditional measures $\law(X_t|X_{s}=x)$ where $X$ is a process of law $P$; just below we denote it by $(P^{s,t}_x)_{x\in \R}$.

We say that $P^{s,t}$ has {\em increasing kernel} if, for all process $(X_s,X_t)$ of law $P^{s,t}$, $x\leqslant y$ $\Rightarrow$ $P^{s,t}_x\leqs P^{s,t}_y$, and that $P$ has increasing kernels if every $P^{s,t}$ has.
\end{defi}

Be careful that the following convention is now used throughout. Our answer to Problem \ref{problem} is Theorem \ref{them:a} below.

\begin{convention}\label{conv:croissant}When we introduce finite sets $\{r_1,\ldots,r_m\}$ or $m$-tuples $(r_k)^m_{k=1}$ of real numbers, we mean implicitly that $r_1<\ldots<r_m$.
\end{convention}

\begin{themBJ}\label{them:a}\maz
Let $(\mu_t)_{t\in\R}$ be a family of probability measures on $\R$.\smallskip

\point\label{p1:them:a} There exists a unique measure $\mq\in \ma((\mu_t)_{t\in \R})$ such that:
\begin{enumerate}
\setcounter{enumi}{0}
\item[\souspoint\label{item:markov}] $\mq$ is Markov,
\item[\souspoint\label{item:transitions}]  $\mq$ has increasing kernels,
\item[\souspoint\label{item:minimal}] 
$\mq$ has minimal couplings (alias transports) among the measures satisfying {\rm\ref{item:markov}} and {\rm\ref{item:transitions}}, in the sense that it satisfies {\rm (\ref{eq:quantile_as_minimum})} where the minimum is taken among processes $(Y_t)_t$ satisfying \ref{item:markov} and \ref{item:transitions}.
\end{enumerate}

\noindent It is also the unique process satisfying {\ref{item:markov}} above and:
 \begin{enumerate}
\setcounter{enumi}{3}
\item[\souspoint\label{item:limite_de_compose}] each $\mq^{s,t}$ is a limit of products 
of quantile couplings $(\como_{[R^{s,t}_n]}^{s,t})_{n\in \N}$ where for $R=\{r_1,\ldots,r_m\}\subset[s,t]$, $\como_{[R]}^{s,t}$ 
is the product
$$\como^{s,r_1}. \como^{r_1,r_2}.\cdots. \como^{r_{m-1},r_m}. \como^{r_{m},t}$$
of the quantile couplings $\como^{r_i,r_j}\in\ma(\mu_{r_i},\mu_{r_j})$.
\end{enumerate}

\noindent Moreover:\smallskip

\point\label{item:croissant} If $(\mu_t)_{t\in \R}$ is increasing for the stochastic order, i.e. $s\leqslant t\Rightarrow\mu_s\leqs \mu_t$,  there exists a process $X=(X_{t})_{t\in\R}:\Omega\rightarrow\R^\R$ of law $\mq$ with increasing trajectories, i.e.\ such that $t\mapsto X_t(\omega)$ is an increasing function, for all $\omega\in\Omega$. 
\end{themBJ}
One may also see $\como^{s,r_1}. \como^{r_1,r_2}.\cdots. \como^{r_{m},t}$ as $(\pr^{s,t})_\#\como_{[R]}$, where $\como_{[R]}$ is introduced in Proposition \ref{pro:markovinise_fini1} below.\medskip

Informally, we may interpret Theorem \ref{them:a}\ref{p1:them:a} as the following answer to Problem \ref{problem}: $\mq$ is the Markov process whose ``infinitesimal transitions'' are those of the quantile process. Besides, an immediate consequence of Theorem \ref{them:a}\ref{p1:them:a} is the important point \ref{p2:rem:como_markovien} of Remark \ref{rem:como_markovien}.

\begin{rem}\label{rem:como_markovien}\maz\point\label{p1:rem:como_markovien} In general, the quantile measure $\como\in\ma((\mu_t)_{t\in \R})$ is not Markov. Take, e.g.,  $\mu_t=\frac12(\delta_0+\delta_1)$ for $t\neq0$ and $\mu_0=\delta_0$, then $\como=\law((X_t)_{t\in\R})$ with $(X_t)_{t\in\R}=0$ or $(X_t)_{t\in\R}=\one_{\R^\ast}$, both with probability $\frac12$. Hence for all $t>0$, $\law(X_t|X_0)=\mu_t$. Now $\law(X_t|X_0=0,\,X_{-1}=i)=\delta_i$ for $i\in\{0,1\}$, so that \eqref{eq:def_markovien} is false for $u=0$. As proved in \cite[Proposition 3]{Ju_seminaire}, $\como$ is Markov except if a phenomenon of this type happens, see details in Example \ref{ex:quantile}.
In particular, $\como$ is Markov when the measures $\mu_t$ have no atoms (see a direct proof in Remark \ref{rem:demo_nonatomique}).

\point\label{p2:rem:como_markovien} When $\como\in\ma((\mu_t)_{t\in \R})$ is Markov, $\mq=\como$. Indeed, then, $\como^{s,t}=\como^{s,r_1}.\cdots.\como^{r_n,t}$ for every $s\leq r_1<\cdots<r_m\leq t$. Hence according to \ref{item:limite_de_compose}, $\forall s, \forall t>s, \mq^{s,t}=\como^{s,t}$. Since both processes are Markov they coincide in law (see Corollary \ref{cor:consist}).
\end{rem}

\begin{reminder}[The quantile process]\label{remind:quantile}The quantiles of a measure $\mu\in\p(\R)$ generalize the notion of the median, which is the quantile of level $\frac{1}{2}$. 
The {\em quantile of $\mu$ of level $\alpha$} is the smallest real number $x_\mu(\alpha)$ such that $\mu(\op]-\infty,\linebreak[1]x_\mu(\alpha)])\geq \alpha$ and $\mu([x_\mu(\alpha),+\infty\clo[)\geq 1-\alpha$. 
The quantile process $X=(X_\tau)_{\tau\in \Tt}$, defined on $\Omega=[0,1]$ with the Lebesgue measure, is given by $X_t(\alpha)=x_{\mu_t}(\alpha)$, and we denote $\law(X)$ by $\como\in \ma((\mu_t)_{t\in \Tt})$.
\end{reminder}

\begin{rem}[Justification of the name ``Markov-quantile'']While Properties \ref{item:markov} and \ref{item:transitions} of Theorem \ref{them:a}\ref{p1:them:a} are satisfied by the product measure (law of the independent process) $\bigotimes_{t\in \R} \mu_t$, the quantile process $\como$ satisfies \ref{item:transitions} and \ref{item:minimal} in the sense that it satisfies \ref{item:transitions} and its couplings $\como^{s,t}$ are minimal among those of the measures satisfying \ref{item:transitions}. In fact, Theorem  \ref{them:a} is constructive and builds $\mq$ as a modification of $\como$; therefore we call this measure $\mq$ the {\bf ``Markov-quantile'' measure} attached to $(\mu_t)_{t\in\R}$.
\end{rem}

In fact, a deeper convergence statement holds than that resulting from point \ref{item:limite_de_compose} above. Indeed we introduce the following notion of a measure in $\ma((\mu_t)_{t\in\R})$ ``turned into a Markov law at a finite set of instants'', denoted in a way that is consistent with the notation of Theorem \ref{them:a}\ref{item:limite_de_compose}.

\begin{pronotation}\label{pro:markovinise_fini1} If $P\in\ma((\mu_t)_{t\in\R})$ and $R\subset \R$ is finite, there is a unique measure in $\ma((\mu_t)_{t\in\R})$, denoted by $P_{[R]}$, such that:\smallskip

-- $P_{[R]}$ is the law of a family of variables $(X_t)_{t\in\R}$ that is ``Markov at the instants of\/ $R$'' i.e.\ (\ref{eq:def_markovien}) holds with ``$\,\forall s\in R$'' instead of ``$\,\forall s\in\R$'',\smallskip

-- for the closure $I$ of each connected component of $\R\setminus R$, $(\pr^I)_\# P_{[R]}=(\pr^I)_\# P$.
\end{pronotation}

This proposition follows from the way we define $P_{[R]}$ in Definition \ref{defi:quantile_discretement_markovien}, using the catenation of transport plans given by  Definition \ref{defi:catenation}. We show:

\begin{themBJ}\label{them:b}There is an increasing sequence $(R_n)_{n\in\N}$ of finite subsets of $\R$ such that $\como_{[R_n]}\in\ma((\mu_t)_{t\in\R})$ converges weakly to $\mq$.
\end{themBJ}

\begin{reminder}\label{reminder:weak_convergence} A sequence $(P_n)_n$ of (probability) measures on some measurable topological space $E$ {\em converges weakly} to $P$ if, for all bounded continuous function $f$, $\int f\dd P_n\to\int f\dd P$. For $E=\R^\R$ with the weak topology, this convergence amounts to the weak convergence of all finite marginals.
\end{reminder}

\begin{rem} Of course, not every sequence $(R_n)_{n\in \N}$ is admissible for Theorem \ref{them:b}. In fact, we prove a more precise version of Theorem \ref{them:b}, see Theorem \ref{them:deuxieme_sec}. It introduces notably the notion of {\em essential atomic times} of $(\mu_t)_t$ that turn out to be the times contained in $\cup_n R_n$ for all sequence $(R_n)_{n\in \N}$ admissible for Theorem \ref{them:b}; see also Remark \ref{rem:choix} in this introduction.
\end{rem}
\noindent{\bf Our problem: a classical type of question.} The problem of defining a measure or a process $P$ with given marginals and additional properties is a general problem that includes Problem \ref{problem} and has already been explored several times in pure and applied Probability Theory as well as in Analysis or Dynamics. Without claiming exhaustiveness on this rich topic we review some research streams and provide references.

A result related to Theorem \ref{them:a}\ref{item:croissant} is proved by Kamae and Krengel in \cite{KaKr}. The measures $(\mu_t)_{t\in \R}$ are in $\p(E)$ where $E$ is a partially ordered Polish space. Assuming the measures in stochastic order, in a suitable sense, the authors prove that there exists an increasing process in $\ma(\mu)$. Other orders can be considered together with expected properties on the processes. For $E=\R^d$, Chapter 8 of \cite{SS} proposes plenty of orders. In Stochastic Analysis and Mathematical Finance, the topic of peacocks and their associated martingales is closely related to our problem. ``Peacock'' stands for PCOC: Processus Croissant Pour l'Ordre Convexe (French), that is, increasing process for the convex order. One aims at defining a martingale in $\ma(\mu)$, where $\mu=(\mu_t)_{t\in \R}$ are the marginals of some peacock, using various techniques. Most of the time the peacock is part of a specific class, so the purpose is more specific than the work of Kellerer presented in Subsection \ref{subsec:intro_kellerer}. In most of this literature the martingales may or not be Markov \cite{Gy,MaYo02,Ju_ejp,HTT,HiRo13,Hob16,pages}. The papers by Lowther \cite{Low,Low2} on limits of diffusion processes for the finite dimensional convergence permitted some authors to refocus on the Markov setting (see, e.g.,  \cite{BHS,HiRoYo14,Ju_seminaire}), rediscovering Kellerer's work by the way. An important example in the topic are the fake Brownian motions, that are processes sharing some of the properties characterizing the standard Brownian motion: they are continuous Markov martingales with marginals $\mu_t=\mathcal{N}(0,t)$. See, e.g.,  \cite{MaYo02,HaKl,Albin,Ol,Hob_fake,HTT} for examples of constructions.

In Section \ref{sec:cont_eq} we will present a more analytic field related to Optimal Transport Theory: Benamou--Brenier's study of the incompressible Euler equation \cite{BeBr}, transport representation for solutions of PDEs after Jordan, Kinderlehrer and Otto \cite{JKO,Ott}. This transport formalism was thoroughly studied by Ambrosio, Gigli and Savar\'e  \cite{AGS}, and continued among others (see \cite{GH,ST,AT}) by Lisini \cite{Lis} in metric spaces.

\subsection{Relations to Kellerer's Theorem}\label{subsec:intro_kellerer}

If you forget about $\mq$ itself, Theorem \ref{them:a}\ref{item:croissant} gives the following existence property.

\begin{cor}\label{coro:c} If $(\mu_t)_{t\in \R}\in\p(\R)^\R$ is increasing for $\leqs$, there exists a Markov process $X=(X_{t})_{t\in\R}:\Omega\rightarrow\R^\R$ such that $\law(X)\in\ma((\mu_t)_t)$ and that the trajectories $t\mapsto X_t(\omega)$ are increasing. 
\end{cor}

This extends to the case of the stochastic order a famous theorem of Kellerer on martingales and submartingales with given marginals \cite{Ke72,Ke73}. Our Theorem \ref{them:c} recovers, with a different proof, Kellerer's result, as well as (simultaneously) Corollary \ref{coro:c} on increasing processes. The proof of Theorem \ref{them:c} is also completely independent from that of Theorem \ref{them:a}. Moreover the method used to show it leads to an existence statement for certain Markov processes, Theorem \ref{them:sous_n} p.\ \pageref{them:sous_n}, omitted in this introduction. To state Theorem \ref{them:c} we need to recall two definitions.

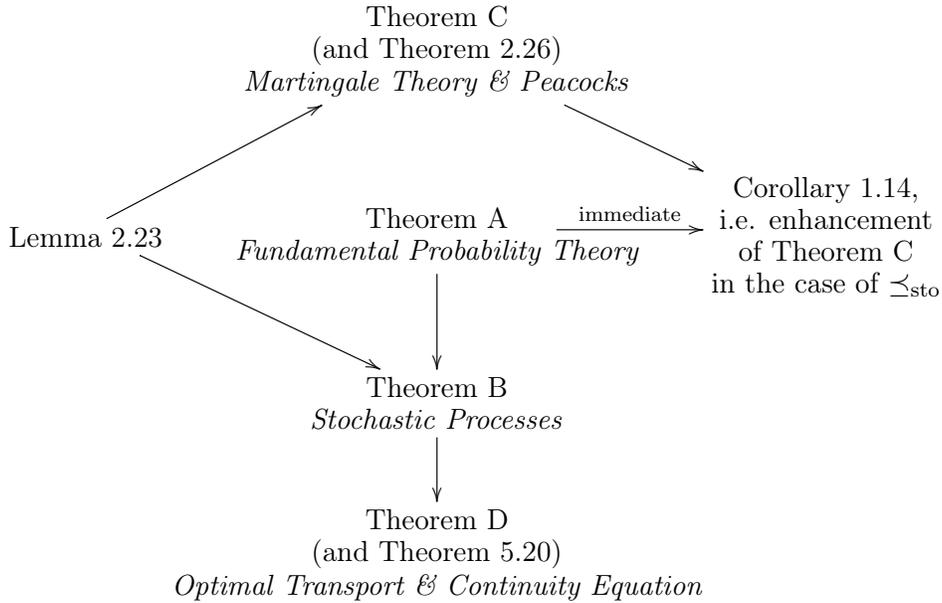
\begin{figure}[ht]
\xymatrix{
&\txt{Theorem C\\
(and Theorem \ref{them:sous_n})\\\emph{Martingale Theory \& Peacocks}}\ar[dr]\\
\txt{Lemma \ref{lem:closed}}\ar[ur]\ar[dr]&\txt{Theorem A\\\hspace*{-1.2cm}{\em Fundamental Probability Theory}\hspace*{-1.2cm}}\ar[d]\ar[r]^-{\txt{\scriptsize immediate}}&\txt{Corollary \ref{coro:c},\\  i.e. enhancement\\of Theorem \ref{them:c}\\ in the case of $\leqs$}\\
&\txt{Theorem B\\ \emph{Stochastic Processes}}\ar[d]&\\
&\txt{Theorem D\\
(and Theorem \ref{them:commun_d_et_1})\\
\hspace*{-3cm}{\em Optimal Transport \& Continuity Equation}\hspace*{-3cm}
}}
\caption{Main theorems: interdependance and field of application.}\label{fig:un}
\end{figure}

\begin{defi}\label{defi:leqc}Two measures $\mu$ and $\nu$ on $\R$, with finite first moments, are said to be in convex order $\leqc$, respectively in convex stochastic order $\leqcs$, if for every convex, respectively convex increasing function $\varphi$:
\begin{align}\label{eq:order}
\int \varphi\, \dd \mu\leq \int \varphi\, \dd \nu.
\end{align}
\end{defi}

Notice that $\mu\leqs\nu$ if and only if \eqref{eq:order} holds for (bounded) increasing functions $\varphi$. Now we define a martingale. We do it in the Markov framework (all we need), where this is a bit simpler. We add in \ref{item:increasing_coupling} a terminology of our own.

\begin{defi}\label{defi:martingale}\maz\point\label{item:martingale}\ A measure $P$ on $(\R^d)^2$ is a \emph{martingale coupling} if for every non-negative continuous bounded function $f:\R^d\to \R$:
\begin{equation}\label{eq:def_martingale}\iint f(x)(y-x) \dd P(x,y)=0.
\end{equation}
For $\Tt\subset\R$, a Markov measure $P$ on $(\R^d)^{\Tt}$ is a Markov martingale if for every $\{s,t\}\subset\Tt$ with $s<t$, the coupling $(\pr^{s,t})_\#P$ is.

\point\label{item:sousmartingale}\ When $d=1$, {\em submartingale couplings} and {\em Markov submartingales} are defined alike, the integral in \eqref{eq:def_martingale} being non-negative instead of null.

\point\label{item:increasing_coupling} A measure $P$ on $\R^2$ is called an {\em increasing coupling} if $P(\{(x,y)\in \R^2:\,x\leq y\})=1$, i.e.\ $P=\law((X_i)_{i\in\{1,2\}})$ where $X_1\leq X_2$. For $\Tt\subset\R$, we say that a measure $P$ on $\R^\Tt$ has increasing couplings if all $(\pr^{s,t})_\#P$ are so.
\end{defi}

\begin{rem}\label{rem:martingale}\maz\point\label{item:traj_croiss} If a measure $P$ on $\R$ is the law of a process with increasing trajectories, as in Theorem \ref{them:a}\ref{item:croissant}, $P$ is in case \ref{item:increasing_coupling} of Definition \ref{defi:martingale}. Actually a classical type of reasoning shows the converse, see Lemma \ref{coro} below. So the result in Corollary \ref{coro:c} amounts to give the existence of a Markov measure $P\in\ma((\mu_t)_t)$ with increasing couplings.

\point\label{item:partie_directe} Take $\Tt\subset\R$. If there exists $P\in\ma((\mu_t)_{t\in\Tt})$ as defined in case \ref{item:martingale}, \ref{item:sousmartingale}, or \ref{item:increasing_coupling} of Definition \ref{defi:martingale}, then it is immediate that $(\mu_t)_t$ is respectively increasing for $\leqc$, $\leqcs$ or $\leqs$. When $\#\Tt=2$, by Strassen's theory \cite{St65}, the converse is true. More generally, it is also true when $\Tt=\N$; one can deduce it by a quite simple induction based on the Markov catenation (Definition \ref{defi:catenation}).
\end{rem}

Kellerer shows Remark \ref{rem:martingale}\ref{item:partie_directe} for $\leqc$ and $\leqcs$ in the much more delicate case $\Tt=\R$.

\begin{them}[Kellerer, {\cite[Theorem 3]{Ke72},\cite{Ke73}}]\label{them:kellerer_intro} If $(\mu_t)_{t\in \R}$ is a family of\linebreak[4] probability measures on\/ $\R$, increasing for $\leqc$ (or $\leqcs$), there is a Markov measure $P\in\ma((\mu_t)_t)$ which is a (sub)martingale.
\end{them}

Kellerer's Theorem remains unproved for vectorial measures, see Open question \ref{ss:Kel_dimsup}. This a major motivation to search new methods to construct or to establish the existence of certain Markov processes, as it is done to prove Theorems \ref{them:a} and \ref{them:c}. Following Kellerer's line of proof, but replacing one of its key lemmas by another one (see details below), we prove its following generalization.
\begin{themBJ}\label{them:c}If $(\mu_t)_{t\in \R}$ is a family of probability measures on\/ $\R$, increasing for $\leqc$, $\leqcs$ or $\leqs$, there is a Markov measure $P\in\ma((\mu_t)_t)$ which respectively is a martingale, is a submartingale or has increasing couplings.
\end{themBJ}

Using Remark \ref{rem:martingale}\ref{item:traj_croiss}, one sees that this theorem proves Corollary \ref{coro:c}, that follows from Theorem \ref{them:a}, by another way.

Kellerer's proof uses a continuity result for certain kernels, recalled in Lemma \ref{lem:lipschtz_implique_cinq}. We replace it by Lemma \ref{lem:closed}, a continuity result for the increasing kernels of Definition \ref{defi:increasing_kernel_intro}. A form of Lemma \ref{lem:lipschtz_implique_cinq} appears in every proof of Kellerer's theorem we know \cite{Ke72,Ke73,Low,HiRoYo14,BHS}, so that in this respect our proof, resting on another type of kernels, is new. About Lemma \ref{lem:closed}, the following comment, that also appear on Figure \ref{fig:un}, are in order.\smallskip

-- It is a significant result of this article; \S\ref{sec:auxiliaire} is devoted to its proof.\smallskip

-- It is not used in the proof of Theorem \ref{them:a}, so that the proofs of our Theorems \ref{them:a} and \ref{them:c} are really independent. They bring separately Corollary \ref{coro:c}, an enhancement Theorem \ref{them:c} in the (new with respect to Kellerer's work) case of $\leqs$.\smallskip

-- It plays a prominent role in Theorem \ref{them:b}, see p.\ \pageref{appel_lemme_closed}, and, as a consequence of it, in all the results of \S\ref{sec:cont_eq} concerning the representation of absolutely continuous curves of order two in Wasserstein space.

\begin{rem}The Doob--Meyer decomposition theorem of some submartingales in a sum of an increasing process and a martingale is another reason our generalization of Theorem \ref{them:kellerer_intro} to $\leqs$ is a natural work.
\end{rem}

We mention also that Kellerer seems to have never considered the question of the extension to $\leqs$ in his papers. However in \cite{Ke86} with application in \cite{Ke87}, he explored the related question of the existence of increasing couplings $P\in \ma(\mu,\nu)$ that are as independent as possible, in a suitable sense.

Finally, here is the lemma announced in Remark \ref{rem:martingale}\ref{item:traj_croiss}, yielding Corollary \ref{coro:c} from Theorem \ref{them:c}. It is proven p.\ \pageref{subsec:proof_coro} in \S\ref{subsec:proof_coro}.

\begin{lem}\label{coro}If $P$ is a probability measure on $\R^\R$ with increasing couplings (see Definition \ref{defi:martingale}\ref{item:increasing_coupling}), there exists an increasing process $X=(X_t)_{t\in\R}:\Omega\to\R$, i.e.\ the functions $t\mapsto X_t(\omega)$ are increasing, such that $P=\law(X)$.  
\end{lem}

\subsection{Application to the continuity equation and its treatment in Optimal Transport Theory}\label{subsec:intro_transport}
Our Markov-quantile process provides in \S\ref{sec:cont_eq} a uniqueness statement in the context of the transport (i.e.\ continuity) equation. We introduce it briefly here and state Theorem \ref{them:d}; in the introduction of \S\ref{sec:cont_eq} we give more detail on  the context of Optimal Transport, we place and motivate our work within it, and introduce Theorem \ref{them:commun_d_et_1}, absent of the general introduction of the article ---this theorem shows that the Markov-quantile process is the limit of processes built by a construction by interpolation, classical in Transport, and opens the question of a generalization of this construction to greater dimensions.

We stress that the Markov property was not involved up to now in the a priori rather analytic context of the transport equation. As we have seen in \S\ref{subsec:intro_kellerer}, considering this property comes from Kellerer's Theorem that is nowadays mostly represented in Martingale Optimal Transport (and Peacocks). This provides new results in Optimal Transport, where it is a novelty. The links between both fields was up to now only the other way around.\medskip

In \S\ref{sec:cont_eq}, we consider only the —nevertheless still rich--- set of {\em continuous} curves $(\mu_t)_{t\in\R}:\R\to\p(\R)$, $\p(\R)$ being endowed with a topology as follows.

\begin{notation}\label{not:mac} For every metric space $(\mathcal{X},d)$ we denote by $\mathcal{C}([0,1],\mathcal{X})$ the space of continuous curves from $[0,1]$ to $\mathcal{X}$ ---simply by $\mathcal{C}$ when $\mathcal{X}=\R$. We denote by $\ma_\mathcal{C}(\mu)$ the set $\{\Gamma\in \p(\mathcal{C}):\,\Gamma^t=\mu_t\ \text{for every }t\in [0,1]\}$ and by $\mathcal{P}_2(\R^d)$ the 2-Wasserstein space $\{\mu\in\mathcal{P}(\R^d):\int \|x\|^2\,\dd \mu(x)<\infty\}$.
\end{notation}

\begin{rem}The set $\p(\mathcal{C}([0,1],\mathcal{X}))$ is considered with its weak topology, which is finer that the trace topology given by $\p(\mathcal{C})\subset\p(\R^{[0,1]})$.  
\end{rem}

We introduce the energy of a curve $\gamma:[0,1]\mapsto \mathcal{X}$ in a metric space $(\mathcal{X},d)$:
$$\EE:\gamma\in \mathcal{C}([0,1],\mathcal{X})\mapsto\sup_{R}\sum_{k=0}^{m}\frac{d(\gamma(r_k),\gamma(r_{k+1}))^2}{r_{k+1}-r_k}\in [0,+\infty],$$
where $R=\{r_1,\ldots,r_m\}\subset\op]0,1\clo[$ and $(r_0,r_{m+1})=(0,1)$, as well as  the 2-Wasserstein distance $W_2$ on $\p_2(\R^d)$, based on the distance of $\R^d$ itself. Then an inequality involving energies for curves in $\R^d$ and $\p(\R^d)$ is proven in Remark \ref{rem:about_action}: 
for all $\Gamma\in\p(\mathcal{C}([0,1],\R^d))$, if $(\pr^t)_\#\Gamma\in\p_2(\R^d)$ for all $t\in [0,1]$:
\begin{equation}\label{eq:inegalite_energie}\int \EE(\gamma)\,\dd\Gamma(\gamma)\geq \EE((\Gamma^t)_{t\in [0,1]}),\ \ \text{where }\Gamma^t\!\eqdef\!(\pr^t)_\#\Gamma.
\end{equation}

In Theorem \ref{them:d} below, the existence of $\Gamma$ such that  \eqref{eq:inegalite_energie} is an equality is well-known. The theorem rather proves the existence and uniqueness of one {\em Markov} such measure.

\begin{themBJ}\maz\label{them:d}
Let $\mu=(\mu_t)_{t\in[0,1]}$ be a curve of finite energy $\EE(\mu)$ in $\p_2(\R)$.\smallskip

\point\label{p1:them:d} {\em (Existence of a Markov representation)} There exists a {\em Markov} probability measure $\Gamma$ in $\ma_\mathcal{C}(\mu)$ such that \eqref{eq:inegalite_energie} is an equality, i.e.:
$$\int \EE(\gamma)\,\dd \Gamma(\gamma)=\EE(\mu),$$
and such that there exists a nested sequence $(R_n)_{n\in \N}$ of finite subsets of\/ $\op]0,1\clo[$ such that $\como_{[R_n]}$ converges to $\Gamma$ in $\p(\mathcal{C})$.\smallskip

\point\ {\em(Uniqueness)} If\/ $\Gamma$ is as in \ref{p1:them:d} then its law is $\mq$.
\end{themBJ}
Note that Theorem \ref{them:d} relies on Theorems \ref{them:a}-\ref{them:b} and Lemma \ref{lem:closed}, as represented on Figure \ref{fig:un}.

\subsection{A first insight in the Markov-quantile process}\label{subsec:intro_quantilemarkovien}

We build the Markov-quantile process $\mq((\mu_t)_{t\R})$ answering Problem \ref{problem} in elementary examples of increasing difficulty, where what it shall be is clear. This makes the generalization of this construction, i.e.\ solving Problem \ref{problem}, a natural goal. Trying to achieve it by a naive strategy will then reveal its difficulties.

\begin{notation}\label{not:leb}
In \S\ref{subsec:intro_quantilemarkovien} we denote the Lebesgue measure on $[0,1]$ by $\la$, on $\R$ by $\dd x$, $\la_\R$ or also $\la$ when there is no ambiguity, on $\R^d$ by $\la_{\R^d}$. 
\end{notation}

\begin{ex}\label{ex1}Define $\mu=(\mu_t)_{t\in\R}$ by:
\begin{align*}
\mu_t=
\begin{cases}
\delta_0&\text{if }t=0\\
\frac{1}{|t|}\one_{[0,|t|]}\dd x&\text{if }t\in \R\setminus \{0\}.
\end{cases}
\end{align*}
The quantile trajectory  $(X_t(\alpha))_{t\in\R}$ associated with the level $\alpha\in [0,1]$ on the probability space $\Omega=([0,1],\la)$ is $t\in \R\mapsto \alpha|t|$. The process $X$ is not Markov because at time $t=0$, with information on $(X_t)_{t<0}$, we better know $(X_t)_{t>0}$ ---actually here, we determine it completely. A modification makes $X$ Markov. Namely, consider the catenation $X^{[0]}$ at time $0$ of $t\in \R^-\mapsto\alpha(-t)$ and $t\in \R^+\mapsto\beta t$, where $(\alpha,\beta)$ is uniformly picked in  $[0,1]^2$, in place of $\alpha=\beta$ picked in $[0,1]$. Then $X^{[0]}$ is the only possible answer to Problem \ref{problem}: it is Markov (hence $(X_t)_{t<0}$ and $(X_t)_{t>0}$ are independent) and equal to the quantile process where the latter is Markov. Moreover it satisfies the properties of Theorem \ref{them:a}, so it is the Markov-quantile process.
\end{ex}

\begin{defi}\label{defi:atomic}A measure without atom is said to be diffuse. We say that $t$ is an {\em atomic time} of a family $(\mu_t)_{t\in\R}$ of measures if $\mu_t$ is not diffuse.
\end{defi}

\begin{ex}\label{ex2}Now take $\mu_0$ any non-diffuse measure and, for $t\neq 0$, $\mu_t$ any diffuse measure, for instance $\mu_0\ast \theta_t$, where $\theta_t$ stands for $\mu_t$ of Example \ref{ex1}, or $\mu_t=\la$. The quantile process is $(X_t)_t=(\min\{x\in \R:\,\mu_t(\op]-\infty,x])\geq \alpha\})_t$; restricted to $\R^-$ or $\R^+$ it is Markov. Pick $\beta$ uniformly on $[0,1]$. On $([0,1]^2,\la\otimes \la)$, we consider the modified process $X^{[0]}$ defined by:
\begin{align*}
X_t^{[0]}(\alpha,\beta)=
\begin{cases}
X_t(\alpha)&\text{if }t\leq 0\\
X_t(\alpha^-+\beta(\alpha^+-\alpha^-))&\text{it }t>0,
\end{cases}
\end{align*}
where $\alpha^+=\alpha^-=\alpha$ if $\mu_0(X_t(\alpha))=0$, and otherwise $\op]\alpha^-,\alpha^+\clo[$ is the maximal open interval such that $X_0(\alpha')=X_0(\alpha)$ for all $\alpha'$ in it. Let us introduce $I_\alpha=\{\alpha\}$ in the first case and $I_\alpha=\op]\alpha^-,\alpha^+\clo[$ in the second one. Conditionally on the fact that $X_0^{[0]}$ equals some atom $x_0=X_0(\tilde\alpha)$ of $\mu_0$, the values of $X^{[0]}_t$ for $t>0$ are made independent of those of $X^{[0]}_t$ for $t<0$. Indeed, the vector $(\alpha,\alpha^- +\beta(\alpha^+-\alpha^-))$ is uniformly distributed on $I_{\tilde\alpha}\times I_{\tilde\alpha}$: this generalizes Example \ref{ex1}; $X^{[0]}$ is Markov and has law\/ $\mq\in \ma(\mu)$.
\end{ex}

\begin{ex}When the set $R$ of atomic times of $(\mu_t)_{t\in \R}$ is finite, one may repeat at each $r\in R$ the independence operation described above at time $0$, to produce a Markov process. Moreover, one may check that it does not matter in which order, because these operations commute. The resulting process is indeed our Markov-quantile process. With the notation of the paper, its law is $\como_{[R]}$. Similarly it is easy to imagine the Markov-quantile process when the atomic times form a locally finite set, like, e.g.,  $\mathbb{Z}$.
\end{ex}

\begin{rem} More examples of Markov-quantile processes are given  in Section \ref{sec:exemples}.
\end{rem}

The situation becomes  complicated when the set $\A$ of atomic times is not locally finite or even worse, uncountable. Consider the following a priori reasonable approach. Let $(R_n)_{n\in \N}$ be a nested family of finite sets such that $R_\infty=\cup_n R_n$ is dense in $\A$. We consider the sequence $(\como_{[R_n]})_n$ \label{page:tentative_limite} and hope for a limit  $Q$. Then we encounter three problems:\smallskip
 
-- By compactness of $\ma((\mu_t)_t)$, $(\como_{[R_n]})_n$ has an accumulation point (see similar reasonings in \cite{HiRo13,Ju_ejp,KaKr,Lis,Vi2}), but it has no reason to be unique.\smallskip
 
-- If $\A$ is uncountable, this limit needs not to satisfy the Markov property \eqref{eq:def_markovien} at times $s\in\A\setminus R_\infty$, at which we did not perform the modification of Examples \ref{ex1}-\ref{ex2}. A continuity assumption on $t\mapsto \mu_t$ could let hope to yield it (this was used, with other goals, for measures in stochastic or convex order, see, e.g.,   \cite{BHS,HiRoYo14}) but we do not make such an assumption. Also in the ``space of quantile levels'', the irregularity may be maximal: the set $\{(t,\alpha)\in\R\times[0,1]:\,X_t(\alpha)$ is an atom of $\mu_t\}$ needs not to be measurable.\smallskip
 
-- Anyway, limits of Markov processes are in general not Markov so here, property \eqref{eq:def_markovien} is not ensured even at $s\in R_\infty$. To our knowledge, before the present paper this type of problem has principally one solution, based on Lipschitz kernels (see Lemma \ref{lem:lipschtz_implique_cinq}), first discovered by Kellerer \cite{Ke72,Low,HiRoYo14,BeJu16, BHS}. However, see \cite[Lemma 5.3]{Naga} for a different statement.\smallskip
 
\noindent Another consequence of this non stability of the Markov property is that it is also not possible to consider the sequence of quantile processes for mollified curves $\mu^{(n)}=(\mu_t\ast \theta_n)_t$, relying on the fact that all the measures $\mu_t^{(n)}$ are diffuse, so that each $\como\in\ma((\mu^{(n)}_t)_{t\in\R})$ is Markov.

\begin{rem}\label{rem:choix}\maz\point\  In fact, the convergence in our Theorem \ref{them:a}\ref{item:limite_de_compose}, and hence in its enhancement Theorem \ref{them:b}, rests on the {\em order} $\leqlc$ introduced in Definition \ref{defi:geqlc}: for all $s$ and $t>s$, the choice of the times $r_i$ follow from that of a sequence in some set of measures, tending to the supremum of this set for $\leqlc$, see Lemma \ref{lem:existenceT}, in particular its point \ref{p3:def_L_R}.

\point\ As the examples above suggest, for all $s$ and $t>s$, the products appearing in point \ref{item:limite_de_compose} of Theorem \ref{them:a} need only to use couplings $\como_{r_i,r_{i+1}}$ where the times $r_i$ are atomic. Adding non-atomic times has no effect. Similarly, in Theorem \ref{them:b}, that provides nested finite sets $R_n$ such that $\como_{[R_n]}\to\mq$, $R=\cup_nR_n$ may avoid all non-atomic times.

Now it appears moreover, but only as a consequence of Theorem \ref{them:a} once it is proven, that all the atomic times of $(\mu_t)_t$ do not play the same role:\smallskip
 
-- Some are ``essential'' (see Definition \ref{defi:essentiel}). All of them that lie in $\op]s,t\clo[$ must eventually appear among the $r_i$ in Theorem \ref{them:a}\ref{item:limite_de_compose}, and all of them must belong to $R$ in Theorem \ref{them:b}. This is possible as they turn out to be at most countable (Proposition \ref{pro:denombrable}).\smallskip
 
 -- One may choose the $r_i$ in Theorem \ref{them:a}, or $R$ in Theorem \ref{them:b}, so that they avoid any fixed finite set of the other atomic times.\smallskip
 
 \noindent Therefore, the intersection of the sets $R$ satisfying the convergence property of Theorem \ref{them:b} is the set of the essential atomic times.
 
The existence, for any sequence $(\mu_t)_t$, of the set of its essential atomic times, at most countable  even if the set $\{(t,\alpha)\in\R\times[0,1]:\,X_t(\alpha)$ is an atom of $\mu_t\}$ is not  measurable, is in itself a significant result of this article. Perhaps does this notion admit generalizations when the set of parameters or the measurable space, both equal to $\R$ in this work, are more general spaces.
\end{rem}

\begin{rem}It is also very important to notice that as soon as the set $\{(t,\alpha)\in\R\times[0,1]:\,X_t(\alpha)$ is an atom of $\mu_t\}$ is regular enough (see the examples of \S\ref{sec:exemples} for clearly stated instances of this), the Markov-Quantile process may be explicitly computed, as it is done in the important Example \ref{ex:integers}. More generally, certain properties of this set and of $\mq$ are linked, see the whole of \S\ref{sec:exemples}. Through its various examples, that section gives also an intuition of how $\mq$ behaves.
\end{rem}

\subsection{Organization of the paper}\label{subsec:organisation}
In \S\ref{sec:catenation+kellerer+boujui} we introduce in \S\ref{subsec:kernels_markov} kernels and transport plans, their composition and catenation, and the Markov property expressed in this language. We give in \S\ref{subsec:proof_kelsto} the structure of Kellerer's work \cite{Ke72,Ke73}, explain why it motivates our reasoning towards  Theorem \ref{them:c}, and prove the latter. However, we postpone the introduction of one auxiliary notion, and the proof of Lemma \ref{lem:transitions_croissantes_limite} and of the essential Lemma \ref{lem:closed} to \S\ref{sec:auxiliaire}. In \S\ref{subsec:markovinification} we state and prove the ``Markovinification'' Theorem \ref{them:sous_n}.

In \S\ref{sec:auxiliaire} we introduce the auxiliary notions and results leading to the proofs of Lemmas \ref{coro}, \ref{lem:closed} and \ref{lem:transitions_croissantes_limite}, then also used in \S\ref{sec:construction}, namely:
\begin{itemize}\maz
\item[\point] the ``lower orthant'' and stochastic orders, related suprema and the notion of increasing kernel in \S\ref{subsec:stochastic_order},
\item[\point] the quantile transport and the notion of minimal coupling in \S\ref{ss:minimality},
\item[\point] two distances, $\rho$ and $\tilde \rho$, inducing the weak topology on spaces of transport plans $\ma(\mu,\nu)$ in \S\ref{subsec:rho}.
\end{itemize}
Lemma \ref{coro} is proved in \S\ref{subsec:proof_coro} and Lemmas \ref{lem:closed} and \ref{lem:transitions_croissantes_limite} are  in \S\ref{subsec:proof_lemmas}.

By the way, \S\ref{subsec:kernels_markov}, \S\ref{subsec:stochastic_order}, and \S\ref{ss:minimality} also give all the background to understand in detail the three properties \ref{item:markov}--\ref{item:minimal} characterizing $\mq$ in Theorem \ref{them:a}.

In \S\ref{sec:construction} we prove Theorems \ref{them:a} and \ref{them:b}: in \S\ref{subsec:little_t} we explain how the situation may be pushed forward to the space $[0,1]$ of ``levels of quantiles'', in \S\ref{subsec:preuve_premier} we prove Theorem \ref{them:a}, i.e.\ build the Markov-quantile process, and in \S\ref{subsec:preuve_deuxieme} we state and prove Theorem \ref{them:deuxieme_sec} which is a more precise and complete version of Theorem  \ref{them:b}. To do this we introduce the essential atomic times of $(\mu_t)_t$.

In \S\ref{sec:cont_eq} we state and prove Theorem D, restated as Theorem \ref{them:action}, and Theorem \ref{them:commun_d_et_1}. See details in the introduction \S\ref{subsec:intro5}.

In \S\ref{sec:exemples} we exhibit the Markov-quantile process in a series of examples, state three last remarks about Theorem \ref{them:a}, and give open questions.\medskip

\noindent{\em Note.} When we introduce various tools, sometimes classical, we do it in a way and with remarks adapted to our context. The reader already knowing them may read quickly, taking notice of our few specific remarks, which are useful in the rest of the article.

\subsection{Notation}\label{subsec:notation}\maz\point\ We gather here the notation we use widely, indicating where each item is introduced, so that the reader can find it quickly if needed.
\mb

We introduce $\m(E)$, $\p(E)$, $\pr^{\Tt'}$, $P^t$ (similarly $\Gamma^t$), $P^{s,t}$ and the set $\ma((\mu_\tau)_{\tau\in\Tt})$ in Notation \ref{notation:marg}, the quantile process $\como$ in Reminder \ref{remind:quantile} and, together with the quantile coupling $\como(\mu,\nu)$, in Definitions \ref{defi:quantile1} and \ref{defi:quantile_process} and Notation \ref{not:multi}, the stochastic order $\leqs$ in Reminder \ref{remind:sto}, $\mq$ in Definition \ref{pro:comp_conv}\ref{bbb}, if \mn$P$ is some process, $R\subset\R$ and $\#R<\infty$\mb, $P_{[R]}$ in Notation \ref{pro:markovinise_fini1} and Definition \ref{defi:quantile_discretement_markovien}, $\leqc$ and $\leqcs$ in Definition \ref{defi:leqc}, $\mathcal{C}$ and $\ma_\mathcal{C}(\mu)$ in Notation \ref{not:mac} and  and \ref{nota:courbes_continues}, $\la$ in Notation \ref{not:leb},  the composition $k.k'$ of kernels in \S\ref{ss:compos}, the kernels $k_P$ in Notation \ref{notation:k_P} and $\idk_E$ in Notation \ref{not:ident}, $\join(\mu,k)$ in Notation \ref{notation_join}, the transport plans $\id_{2,\mu}$ and $\id_{n,\mu}$ in Notation \ref{notation:transport_id}, if \mn$P$\mb\ is some proces, $\trans P$ in Definition \ref{defi:transpose}, $P\circ P'$ in Definition \ref{defi:catenation}, $\mathcal{N}^\ML_{s,t}$ in Definition \ref{defi:lipschitz} and $\mathcal{N}^\IT_{s,t}$ in Notation \ref{nota:Nik}, $x\leq y$ when \mn $(x,y)\in(\R^d)^2$\mb\ in Notation \ref{notation:intervallesRd}, $F_\mu$ or $F[\mu]$, if \mn$\mu$\mb\ is some measure, in Definition \ref{def:cumul}, $\mu\leqlc\nu$ in Definition \ref{defi:geqlc}, $\lcsup_\tau P_\tau$ in Definition \ref{def:lcsup}, $\m(\mu)$, $\p(\mu)$, $\mdec(\mu)$ and $\pdec(\mu)$ in Notation \ref{notation:mdec}, $G_\mu$ in Definition \ref{defi:quantile2}, $\rsemibracket$ in Notation \ref{notation:intervallesRd}, the distance $\rho$ in Notation \ref{pro:rho}, $\tilde \rho$ in \S\ref{ss:another}, the kernels $q_r,\, k_r$ and $\trans k_r$ in Notation \ref{notation:kernels}, $A_{r,x}$, $A_{r}$ and $\ell_r$ in Notation \ref{nota:Asx}, $L_R$, for \mn$R\subset\R$\mb, in Notation \ref{lem:existenceT} and $\ell_R$ in Notation \ref{nota:kernel_tS}, $\lev$ in Definition \ref{defi:lev}, $\parti([a,b])$ and $|R|$ in Definition \ref{defi:parti}, $\mathcal{L}_a^b(\gamma)$, $\mathcal{AC}$ and $\mathcal{AC}_2$ in Notation \ref{nota:length}, $\EE_a^b(\gamma)$ and $\EE_a^b(\gamma,R)$ in Definition \ref{defi:energie}, $W_2(\mu,\nu)$ in Reminder \ref{reminder:wasserstein}, $\A(\Gamma)$ in Defintion \ref{defi:action}, $u^\Gamma_t$ in Definition \ref{defi:barycentre} and $\disp_{[R]}$ in Definition \ref{defi:disp}
\mn.\smallskip

\point\ In this article, if $E'\subset E$, $\mbox{\mb$\one_{E'}$\mn}:E\rightarrow\{0,1\}$ is the indicator function of $E'$ and, if $\mu\in\m(E)$, \mb$\mu\lfloor_{E'}$\mn\ stands for $\one_{E'}\mu$. The Dirac measure at $x\in E$ is denoted by \mb$\delta_x$\mn\ and \mb$\la$\mn\ stands for the Lebesgue measure on $[0,1]$, $\R$ or $\R^d$. Most of the time we deal with $\la|_{[0,1]}$ so, when there is no ambiguity, we simply write $\la$ in this case. If $f:E\rightarrow F$ is measurable and $\mu\in\m(E)$, ${\text{\mb$f_\#\mu$\mn}}\in\m(F)$ is defined by $f_\#\mu(B)=\mu(f^{-1}(B))$. Product measures are denoted by \mb$\mu\otimes \nu$\mn. If $f$ and $g$ are functions, \mb$f\otimes g$\mn\ stands for $(x,y)\mapsto(f(x),g(y))$.\smallskip

\point\ Recall also Convention \ref{conv:croissant}: introducing $\{r_1,\ldots,r_m\}$ or $m$-tuples $(r_k)^m_{k=1}$ of real numbers, we mean implicitly that $r_1<\ldots<r_m$.

\begin{vocabulary}\maz\point\ Similarly, we gather  our common vocabulary:

\begin{itemize}
\item[--] \emph{process} and \emph{marginal (law)}: see Reminder \ref{reminder:process},
\item[--] \emph{canonical process, coupling} and \emph{transport (plan)}: see Reminder \ref{remind:processus_canonique},
\item[--] \emph{martingale coupling, submartingale coupling and increasing coupling}: see Definition \ref{defi:martingale},
\item[--] \emph{atomic} and \emph{essential atomic times}: see Definitions \ref{defi:atomic} and \ref{defi:essentiel},
\item[--] \emph{increasing}, resp.\ \emph{Lipschitz kernel}: see Definitions \ref{defi:increasing_kernel_intro} and \ref{prodef:transitions}, resp.\ Definition \ref{defi:lipschitz},
\item[--] \emph{quantile coupling} and \emph{quantile measure}: see Definitions \ref{defi:quantile1} and \ref{defi:quantile_process},
\item[--] \emph{atomic levels}: see Notation \ref{nota:Asx},
\item[--] a process ``\emph{$M$ made Markov at the points of $R$}'': see Definition \ref{defi:quantile_discretement_markovien},
\item[--] a \emph{partition} of an interval: see Definition \ref{defi:parti},
\item[--] an \emph{absolutely continuous} curve, its \emph{energy}: see Reminder \ref{nota:length} and Definition \ref{defi:energie}.
\end{itemize}
\end{vocabulary} 

\point\ Throughout this paper ``increasing'' and ``decreasing'' mean ``nondecreasing'' and ``nonincreasing''. Indeed we deal often with partial orders, for which the two latter terms are unclear: the contrary of ($\forall s<t$, $\mu_s\leq \mu_t$ and $\mu_s\neq \mu_t$) is ($\exists s<t: \mu_s\geq \mu_t$ or ({$\mu_s$ and $\mu_t$ are incomparable)).

\subsubsection*{Aknowledgements} The authors wish to thank Martin Huesmann, Christian L\'eonard, Emmanuel Opshtein and Xiaolu Tan for bibliographic or editorial suggestions as well as Michel \'Emery and Erwan Hillion for discussions on examples related to this work.

\section{An extension of a theorem of Kellerer}\label{sec:catenation+kellerer+boujui}

\subsection{The Markov property, composition and catenation of kernels and transport plans}\label{subsec:kernels_markov}

Everywhere $E$, $E'$, $E''$ {\em etc.}\ are topological spaces (or sometimes Polish spaces) and $\mathcal{B}(E)$, $\mathcal{B}(E')$ and $\mathcal{B}(E'')$ their Borel $\sigma$-algebras. 

\begin{defi}
A probability kernel, or kernel $k$ from $E$ to $E'$ is a map $k:E\times \mathcal{B}(E')\to [0,1]$ such that $k(x,\cdot)$ is a probability measure on $E'$ for every $x$ in $E$ and $k(\cdot,B)$ is a measurable map for every $B\in {\mathcal B}(E')$.
\end{defi}

Probability kernels are usually interpreted as transition matrices, see Remark \ref{rem:matrix}: after one step a particle at $x$ in $E$ arrives at a random position in $E'$, distributed with respect to $k(x,\cdot)$. We often have that interpretation in mind.

\begin{remnotation}\label{notation:k_P}Every transport plan $P\in\p(E\times E')$ can be disintegrated with respect to its first marginal $P^1\eqdefup(\pr^1)_\#P$ and a kernel that we denote by $k_P$, defined from $E$ to $E'$, so that:
$$\iint f(x,y)\,\dd P(x,y)=\int\left(\int f(x,y)\,k_P(x,\dd y)\right)\,\dd P^1(x)$$
for every bounded continuous function $f$. Observe that $x\mapsto k(x,\cdot)$ is $P^1$-almost surely uniquely determined.
\end{remnotation}

\subsubsection{Composition and action of kernels.}\label{ss:compos} Kernels $k$ from $E$ to $E'$ and $k'$ from $E'$ to $E''$ can be composed as follows:
$$(k.k')(x,A)=\int_{E'} k'(y,A) k(x,\dd y).$$
Similarly, acting on the right, kernels from $E$ to $E'$ \emph{transport}, or \emph{send} positive measures $\theta$ on $E$ on positive measures on $E'$. Acting on the left, they send (adequately integrable) functions $f:E'\rightarrow \R$, on functions $E\rightarrow \R$:
$$(\theta.k)(A)=\int_{E'}k(y,A)\theta(\dd y)\ \ \text{and:}\ \ (k.f)(x)=\int_{E'} f(y) k(x,\dd y).$$
\noindent Associativity holds, e.g.,  $(k.k').k''=k.(k'.k'')$, and $\theta.(k.f)=(\theta.k).f$ where the action of measures on functions is the obvious one. This is consistent with the following remark.

\begin{rem}\label{rem:matrix}
 We recall the usual interpretation of the composition as matrix multiplication. If $E=\{x_i\}_{i=1}^{n}$, $E'=\{y_j\}_{j=1}^{n'}$, $E''=\{z_k\}_{k=1}^{n''}$ are finite, a measure $\theta\in\m(E)$ is a row vector $(\theta(\{x_i\}))_{i=1}^{n}=(\theta_i)_{i=1}^{n}$, a kernel $k$ from $E$ to $E'$ is a matrix $k=((k_{i,j})_{i=1}^{n})_{j=1}^{n'}$ where $(k_{i,j})_{j=1}^{n'}$ is the measure $k(x_i,\,\cdot\,)\in\m(E')$, viewed as a vector, and a function $f$ from $E''$ to $\R$ is a (column) vector $f=(f(z_j))_{j=1}^{n''}$. Then, taking $k'=((k'_{j,k})_{j=1}^{n'})_{k=1}^{n''}$ a kernel from $E'$ to $E''$, $\theta.k$, $k.k'$ and $k'.f$  introduced above have the same sense as products of matrices.
\end{rem}

\begin{notation}\label{not:ident} We denote by $\idk_E$ the identity kernel (that acts trivially) $(x,B)\mapsto\delta_x(B)=\one_B(x)$.
\end{notation}

\begin{notation}\label{notation_join} With $\mu\in\m(E)$ and $k$ a kernel from $E$ to $E'$ is naturally associated the law $\join(\mu,k)\in\m(\espace\times \espace')$, having $\mu$ as first marginal and the family $(k(x_0,\cdot))_{x_0\in E}$ as laws (on $E'$) conditioned by $x_0\in E$:
$$\forall B,B'\in{\mathcal B}(E)\times{\mathcal B}(E'), \join(\mu,k)(B\times B')=\int_Bk(x,B')\dd\mu(x).$$
In particular $P=\join(P^1,k_P)$.
\end{notation}

\subsubsection{Composition of transport plans.}\label{subsubsec:composition} If $P\in\ma(\mu,\mu')$ and $Q\in\ma(\mu',\linebreak[1]\mu'')$, we can compose them in a similar way as we compose kernels, getting the product:
$$P.Q\eqdefup\join(\mu,k_P.k_Q)\in\ma(\mu,\mu''),\ \text{so that: }k_{P.Q}=k_P.k_Q.$$

\begin{notation}\label{notation:transport_id}
We denote $((\Id_E)_{i=1}^n)_\#\mu\in\ma((\mu)_{i=1}^n)$ by $\Id_{n,\mu}$ or simply $\id_n$ when there is no ambiguity. With $n=2$, $\Id_{2,\mu}=\join(\mu,\idk_E)$. It is moreover the identity transport: $\id_{2,\mu}.P=P=P.\id_{2,\mu'}$.
\end{notation}

\subsubsection{Action of transport plans on measures and functions.}\label{subsubsec:action_des_transports} If $\mu\in\m(E)$ and $\mu'\in\m(E')$, transport plans $P\in\ma(\mu,\mu')$ have an action similar to that of kernels from $E$ to $E'$, on ($\mu'$-almost surely defined) classes of functions $f$ and on measures $\theta\in\m(E;\mu)$ absolutely continuous with respect to $\mu$. For instance, the latter are transported in $\m(E',\mu')$, as follows:
$$\text{if $\theta\ll \mu$ has density $g$ and $B'\in{\mathcal B}(E')$, }(\theta.P)(B')=\int_{B'}\int_E g(x)\dd P(x,y).$$
If $k$ is a kernel from $E$ to $E'$, $\theta.k=\theta.\join(\mu,k)$. Conversely $\theta.P=\theta.k_P$.\medskip

\begin{defi}\label{defi:transpose} Take $P\in\ma(\mu_1,\ldots,\mu_k)$, or $P\in\ma((\mu_t)_{t\in\R})$. We define its transpose $\trans P\in\ma(\mu_k,\ldots,\mu_1)$, resp.\  $\trans P\in\ma((\mu_{-t})_{t\in\R})$ by: $ \trans P(B_1\times \ldots\times B_k)=P(B_k\times\ldots\times B_1)$, resp.\ $\trans P(\prod_iB_{t_i})=P(\prod_iB_{-t_i})$. 
\end{defi}

The notation $\trans P$ is a reference to a transposed matrix in Remark \ref{rem:matrix}. For $P=\ma(\mu,\nu)$ we will often consider the bilinear map 
$$B:(\theta,f)\mapsto (\theta P)f=\theta (Pf)=\iint g(x)f(y)\,P(\dd x,\dd y)$$
where $g$ is the density of $\theta$ with respect to $\mu$. The case $\theta=\mu\lfloor_{\op]-\infty,x]}$, $f=\one_{\op]-\infty,y]}$ with $B(\theta,\nu)=P(\op]-\infty,x]\times\op]-\infty,y])$ is of special interest, see \S \ref{subsec:stochastic_order}.\medskip

\subsubsection{Catenation of transport plans.}\label{subsec:caten} (See, e.g.,  \cite{Ke72} p.\ 111, \cite{Vi2} p.\ 23.)
\begin{defi}\label{defi:catenation}
If $\mu_i\in\p(E_i)$ for $i\in\{1,2,3\}$, if $P_{1,2}\in\ma(\mu_1,\mu_2)$ and $P_{2,3}\in\ma(\mu_2,\mu_3)$, their catenation $P_{1,2}\circ P_{2,3}$ is the unique $R\in\p(\R^3)$ such that for every $ (B_1,B_2,B_3)\in{\cal B}(E_1)\times{\cal B}(E_2)\times{\cal B}(E_3)$:
\begin{align}\label{eq:catenation}
R(B_1\times B_2\times B_3)&=\int_{x\in B_1}\int_{y\in B_2}\int_{z\in B_3} \dd\mu_1(x)k_{1,2}(x,\dd y)k_{2,3}(y,\dd z).
\end{align}
In particular, $R\in\ma((\mu_1,\linebreak[1]\mu_2,\linebreak[1]\mu_3)$, $(\pr^{1,2})_\#R=P_{1,2}$, and $(\pr^{2,3})_\#R=P_{2,3}$.
\end{defi}

\begin{rem} Let $k_{i,j}$ be a disintegration kernel for $P_{i,j}$ and $P_{2,1}\eqdefup\trans P_{1,2}$. The right side of \eqref{eq:catenation} also reads: $\iint_{B_1\times B_2} \int_{B_3}k(y,\dd z)\dd P_{1,2}(x,y)$, hence also:
\begin{equation}
\int_{y\in B_2} \iint_{x\in B_1,\,z\in B_3}k_{2,1}(y,\dd x)k_{2,3}(y,\dd z)\mu_2(\dd y).\label{eq:catenation3}
\end{equation}
Catenation is ``reversed when time is reversed'' in the sense that $\trans P_{2,3}\circ{}\trans P_{1,2}=\trans(P_{1,2}\circ P_{2,3})$; this is immediate after (\ref{eq:catenation3}). Formula (\ref{eq:catenation}) gives immediately that $\circ$ is associative, leading to its following generalization: 
$$P_{1,2}\circ\ldots\circ P_{n-1,n}\Bigl({\prod_{i=1}^n}B_i\Bigr)=\int_{(x_i)_i\in \prod_iB_i}
\hspace{-4ex}\dd\mu_1(x_1)k_{1,2}(x_1,\dd x_2)\ldots k_{n-1,n}(x_{n-1},\dd x_n).$$
\end{rem}

\begin{rem}\maz\point\ In \S\ref{subsubsec:composition}, one can also define $P.Q$ as $\pr^{E,E''}_\#(P\circ Q)$.

\point\ The composition and catenation of transport plans find an easy interpretation in terms of random variables. If $(X_1,X_2,X_3)$ is a random vector with $\law(X_1,X_2)=P$ and $\law(X_2,X_3)=Q$ such that $(X_i)_{i\in\{1,2,3\}}$ is a Markov process ($X_1$ and $X_3$ are independent conditionally on $X_2$), then $P.Q$ is the law of $(X_1,X_3)$ and $P\circ Q$ the law of $(X_1,X_2,X_3)$, see \eqref{eq:catenation3}.
\end{rem}

\subsubsection{The Markov Property.} 

We introduce the Markov property here in an alternative, equivalent way to the usual one.

\begin{remdefi}\label{defi:markov2}As recalled in \eqref{eq:def_markovien}, a process $(X_t)_{t\in\R}$ is said to be {\em Markov} if: $\forall s\in\R,\forall t>s, \law(X_t|\,(X_{u})_{u\leqslant s})=\law(X_t|\,X_{s})$. Denoting $\law((X_t)_t)\in\p(\R^\R)$ by $P$, this is equivalent to the fact that for all finite subset $S=\{s_1,\ldots,s_d\}$ of $\R$, $(\pr^S)_\#P$ is the catenation $P^{s_1,s_2}\circ\cdots\circ P^{s_{d-1},s_d}$, where $P^{s_i,s_{i+1}}$ denotes $(\pr^{\{s_i,s_{i+1}\}})_\#P$. More generally, we say that any measure $P\in\p(\R^\R)$ is Markov if it satisfies this property.

We extend these definition in the obvious way to processes indexed on subsets $R\subset \R$ and measures on $\p(\R^R)$.
\end{remdefi}

We recall the Kolmogorov--Daniell theorem and its usual corollary on Markov measures.

\begin{pro}[Kolmogorov--Daniell theorem]\label{pro:inductive_limit}
Let $(\mu_S)_S$ be a family of probability measures on some Polish space $E$, indexed by the finite subsets $S$ of $\R$. If $(\pr^{S'})_\#\mu_S=\mu_{S'}$ for every $S'\subset S$, there exists a unique $P\in \p(E^\R)$ with $(\pr^S)_\#P=\mu_S$ for every $S$.
\end{pro}

One of the most usual applications of Proposition \ref{pro:inductive_limit} is for measures $\mu_S$ of type $\mu_{s_1,s_2}\circ\cdots\circ \mu_{s_{d-1},s_d}$ where $S=\{s_1,\ldots,s_d\}$.

\begin{cor}\label{cor:consist}
Let $(\mu_{s,t})_{s<t}$ be a family of transport plans in $\p(E\times E)$ such that:
$$\mu_{s,u}=\mu_{s,t}.\mu_{t,u}$$
for every $s<t<u$. Then there exists a unique Markov measure $P\in \p(E^\R)$ with $P^{s,t}=\mu_{s,t}$ for every $s<t$.
\end{cor}

\begin{defi}\label{defi:consist}It is usual to call \emph{consistent family} every family $(\mu_S)_S$ or $(\mu_{s,t})_{s<t}$ as in Proposition \ref{pro:inductive_limit} and Corollary \ref{cor:consist}.
\end{defi}

\subsection{Kellerer's work. Our motivation and proof of Theorem \ref{them:c}}\label{subsec:proof_kelsto}

In \cite{Ke72} and \cite{Ke73}, Kellerer proves the three results that we reproduce as Theorem \ref{them:existence}, Lemma \ref{lem:lipschtz_implique_cinq} and finally Theorem \ref{them:kellerer_final}, which is a more precise version of Theorem \ref{them:kellerer_intro} given in the introduction. He also introduces Definition \ref{defi:lipschitz}. As we will see Theorem \ref{them:existence} extends Corollary \ref{cor:consist}: take $\mathcal{N}_{s,t}=\{\mu_{s,t}\}$.

Our goal here is to prove Theorem \ref{them:c}. To put forward quickly both the background and our reasoning we postpone all the intermediate proofs, as well as the introduction of the technical tools they require  to the next section.

Kellerer first proves the following statement ---we give the sketch of proof p.\ \pageref{proof:them:existence}. It seems a bit stronger than in \cite{Ke72} but is what he actually shows.

\begin{them}[{\cite[Theorem 1]{Ke72}}]\label{them:existence}
Let $(\mu_t)_{t\in \R}$ be a family of probability measures on some Polish space $E$, and for every $s< t$ let $\mathcal{N}_{s,t}\in\p(E^2)$ be a set of transport plans.
Assume that:
\begin{itemize}
\item[{\bf (1)}] for every $s,t$, $\mathcal{N}_{s,t}$ is not empty,
\item[{\bf (2)}] for every $s,t$, $\mathcal{N}_{s,t}\subset\ma(\mu_s,\mu_t)$,
\item[{\bf (3)}] for every $s,t$, $\mathcal{N}_{s,t}$ is closed for the weak topology,
\item[{\bf (4)}] for $r<s<t$ and any $(P,P')\in\mathcal{N}_{r,s}\times \mathcal{N}_{s,t}$, $P.P'\in\mathcal{N}_{r,t}$,
\item[{\bf (5)}] for every $d$ and $t_1<\cdots<t_d$, if the sequences $(Q^n_{t_i,t_{i+1}})_n\in \mathcal{N}_{t_i,t_{i+1}}$ converge weakly to $Q_{t_i,t_{i+1}}$, then the sequence $(Q^n_{t_1,t_2}\circ\cdots \circ Q^n_{t_{d-1},t_d})_{n}$ tends weakly to $Q_{t_1,t_2}\circ\cdots \circ Q_{t_{d-1},t_d}$.
\end{itemize}

Then, there exists a Markov measure $P\in\ma((\mu_t)_t)$ with $(\pr^{s,t})_\#P \in \mathcal{N}_{s,t}$ for every $s<t$.
\end{them}

\begin{defi}[{\cite[Definition 3]{Ke72}}]\label{defi:lipschitz} Let $(E,\mu)$ and $(E',\mu')$ be two measure metric spaces and $P$ be in $\ma(\mu,\mu')$. Then $P$ {\bf has Lipschitz kernel} if for every 1-Lipschitz map $h:E'\to [0,1]$, $P.h:E\to [0,1]$ is also 1-Lipschitz, i.e., more exactly, there is a 1-Lipschitz $\tilde h:E\to [0,1]$ such that $\tilde h=k_P.h$, $\mu$-almost surely (see Notation \ref{notation:k_P} for $k_P$). For $(\mu_t)_{t\in\R}\in\p(\R)^\R$, we denote $\{P\in\ma(\mu_s,\mu_t):\,P\ \text{has Lipschitz kernel}\}$ by $\mathcal{N}^\ML_{s,t}$.
\end{defi}

Remark \ref{rem:comments_ml} gives some comments, Remark \ref{rem:otherNs} is used in the following.

\begin{rem}\label{rem:comments_ml}{\bf (a)} The terminology ``Lipschitz property'' was introduced in \cite[Definition 4.1]{Low} for a Markov process with Lipschitz transition kernels. It is renamed as ``Lipschitz-Markov property'' by Hirsch, Roynette and Yor in \cite{HiRoYo14}. The fact that this property is stable for finite dimensional convergence of processes is crucial in those papers and in \cite{Low2} and appears as an avatar of Kellerer's Lemma \ref{lem:lipschtz_implique_cinq} stating that the catenation operator $\circ$ is continuous for the corresponding class of kernels. These kernels are called Lipschitz in \cite{Ju_seminaire,BeJu16} and the present paper, and Lipschitz-Markov in \cite{BHS}.

{\bf(b)} You may compare Definition \ref{defi:lipschitz} with that of transport plans with increasing kernel in Definition \ref{prodef:transitions} \ref{pt:condexp}.

{\bf(c)} \cite[p.\ 115]{Ke72} In case the topology of $E$ and $E'$ is discrete, hence induced, e.g.,  by the distance $d(x,y)=1-\delta_{x,y}$, every $\tilde h$ is 1-Lipschitz; hence any $P$ has Lipschitz kernel.
\end{rem}

\begin{rem}\label{rem:otherNs}\maz
\point\label{item:rem:otherNsA} If some family $(\mathcal{N}_{s,t})_{s,t}$ satisfies the properties of Theorem \ref{them:existence}, a family of subsets $(\mathcal{N}'_{s,t})_{s,t}$ with $\mathcal{N}'_{s,t}\subset \mathcal{N}_{s,t}$ satisfies them as soon as it satisfies (1), (3) and (4), (2) and (5) being automatically true.

\point\label{item:rem:otherNsB} For $(\mu_t)_t$ any family of measures on $\R$, it is easy to check that $(\mathcal{N}^\ML_{s,t})_{s<t}$ satisfies (1)--(4) in Theorem \ref{them:existence} (for (3), see \cite[Satz 13]{Ke72}).
\end{rem}

\begin{lem}[Continuity of $\circ$ when the kernels are Lipschitz]\label{lem:lipschtz_implique_cinq}\cite[Sätze 14 and 15]{Ke72} 
If $(E_t,\mu_t)$ are complete and separable measure metric spaces, $(\mathcal{N}^\ML_{s,t})_{s<t}$ satisfies (5) in Theorem \ref{them:existence}.
\end{lem}

\begin{rem}In fact \cite[Sätze 14 and 15]{Ke72} proves (5) for sequences $(Q_{t_i,t_{i+1}}^n)_n$ of Markov-Lipschitz transports, without the assumption that the $Q_{t_i,t_{i+1}}^n$ have the same marginals for all $n$, though this stronger result is not used further in \cite{Ke72}. In our Lemma \ref{lem:closed} this assumption is crucial.
\end{rem}

Finally Kellerer proves this more precise version of Theorem \ref{them:kellerer_intro}.

\begin{them}[{\cite[Theorem 3]{Ke72},\cite{Ke73}}]\label{them:kellerer_final} If $(\mu_t)_t$ is an increasing family of measures on $\R$, for $\leqc$ (or $\leqcs$), there is a Markov measure $P\in\ma((\mu_t)_t)$ such that $P$ is a (sub)martingale and the couplings $P^{s,t}$ have Lipschitz kernel.
\end{them}

\begin{rem}\label{rem:plan_kellerer}To prove Theorem \ref{them:kellerer_final}, by both points of Remark \ref{rem:otherNs}
and \ref{item:rem:otherNsB}, 
and Lemma \ref{lem:lipschtz_implique_cinq}, Kellerer has only to show that $(\mathcal{N}_{s,t}'^\ML)_{s<t}\eqdefup(\{P\in\mathcal{N}_{s,t}^\ML :\,P$ is a (sub)mar\-tin\-gale transport$\})_{s<t}$ ---see Definition \ref{defi:martingale}--- satisfies Assumptions (1), (3) and (4) of Theorem \ref{them:existence} and to apply this theorem. Checking (3) and (4) being easy, we see that the two important facts enabling to use Theorem \ref{them:existence} and thereby getting Theorem \ref{them:kellerer_final} are:\smallskip

{\bf(i)} Lemma \ref{lem:lipschtz_implique_cinq},\smallskip

{\bf(ii)} the proof of Property (1), i.e.\ the non-emptiness of the $\mathcal{N}_{s,t}'^\ML$.
\end{rem}

Replacing point (i) by an alternative version (i'), consisting of Lemma \ref{lem:closed} below, and proving a version of (ii) adapted to this change, we prove Theorem \ref{them:c}. Namely we prove that increasing kernels, introduced in Definition \ref{defi:increasing_kernel_intro} (see also Definition \ref{prodef:transitions} for more details), satisfy Lemma \ref{lem:closed}, a counterpart of Lemma \ref{lem:lipschtz_implique_cinq}, as well as Property (3), i.e.\ the little Lemma \ref{lem:transitions_croissantes_limite}. They are proven respectively on pp.\ \pageref{proof:lem:closed} and \pageref{proof:lem:transitions_croissantes_limite}. Then we prove Theorem \ref{them:c}.

\begin{lem}[Continuity of $\circ$ when the kernels are increasing]\label{lem:closed}Take $
(\mu_1,\linebreak[4]\ldots,\mu_n)\in\p(\R)^n$ and for all $i\in\llbracket1,n-1\rrbracket$ a closed set $\mathcal{I}_{i,i+1}\subset\ma(\mu_i,\mu_{i+1})$ of transport plans with increasing kernel. The sets $\mathcal{I}_{i,i+1}$ satisfy Property (5) stated for the sets $\mathcal{N}_{t_i,t_{i+1}}$ in Theorem \ref{them:existence}.
\end{lem}

\begin{lem}\label{lem:transitions_croissantes_limite} Take $\mu$ and $\mu'$ in $\p(\R)$. The space of transport plans with increasing kernel in $\ma(\mu,\mu')$ is closed for the weak topology.
\end{lem}

\begin{notation}\label{nota:Nik}If $(\mu_t)_t\in\p(\R)^\R$ is given, we denote $\{P\in\ma(\mu_s,\mu_t):\,P$ has increasing kernel$\}$ by $\mathcal{N}^\IT_{s,t}$.
\end{notation}

\begin{proof}[Proof of Theorem \ref{them:c}] Take $(\mu_t)_t\in\p(\R)^\R$, increasing for $\leqc$ (case (a)), $\leqcs$ (case (b)) or for $\leqs$ (case (c)) to prove the corresponding cases of Theorem \ref{them:c}. In the sketch of proof of Theorem \ref{them:kellerer_final} given in Remark \ref{rem:plan_kellerer}, replace $\mathcal{N}^\ML_{s,t}$ by $\mathcal{N}^\IT_{s,t}$ and introduce, similarly as defined in Remark \ref{rem:plan_kellerer} for case (a) and (b), the spaces $\mathcal{N}_{s,t}'^\IT$ ---equal to 
$\{P\in\mathcal{N}^\IT_{s,t}:\,P(\{(x,y)\in\R^2:\,x\leqslant y\})=1\}$ in case (c).

Properties (2)--(5) of Theorem \ref{them:existence} are satified by $(\mathcal{N}^\IT_{s,t})_{s<t}$: (2) by definition, (3) by Lemma \ref{lem:transitions_croissantes_limite}, (4) is straightforward and (5) by Lemma \ref{lem:closed}. By Remark \ref{rem:otherNs} (a) we are left with showing (1), (3) and (4) for $(\mathcal{N}_{s,t}'^\IT)_{s<t}$. Plainly, the conditions defining the $\mathcal{N}_{s,t}'^\IT$ as subspaces are closed and stable by composition, (3) and (4) follow. For (1), in our three cases:\smallskip

{\bf(a)} By \cite[Subsection 3.1]{BHS}, if $P\in\ma(\mu_s,\mu_t)$ is a martingale transport plan, then $P\in\mathcal{N}_{s,t}^\ML\Leftrightarrow P\in\mathcal{N}_{s,t}^\IT$. Therefore $\mathcal{N}_{s,t}'^\IT\neq\varnothing$ if and only if $\mathcal{N}_{s,t}'^\ML\neq\varnothing$, which Kellerer proved.\smallskip

{\bf(b)} The element of $\mathcal{N}_{s,t}'^\ML$ Kellerer built in \cite[Definition 7 and Theorem 2]{Ke72} is in $\mathcal{N}_{s,t}'^\IT$ when $\mu_s\leqc\mu_t$.\smallskip

{\bf(c)} As explained in Remark \ref{rem:quantile_up}, $\como(\mu_s,\mu_t)\in \mathcal{N}_{s,t}'^\IT$.
\end{proof}

For the completeness of this exposition, we also provide the following.

\begin{proof}[Sketch of proof of Theorem \ref{them:existence}, after Kellerer]\label{proof:them:existence} Take $S=\{s_1,\ldots,s_d\}$ any finite subset of $\R$. We introduce $\mathcal{N}_{s_1,s_2}\circ\cdots\circ\mathcal{N}_{s_{d-1},s_d}\eqdef\{N_1\circ\ldots\circ N_{d-1}:\,\forall i, N_i\in\mathcal{N}_{t_i,t_{i+1}}\}$ and $\mathcal{L}_S\eqdefup(\pr^S)_\#^{-1}(\mathcal{N}_{s_1,s_2}\circ\cdots\circ\mathcal{N}_{s_{d-1},s_d})$. Since $\pr^S:\ma((\mu_t)_t)\rightarrow\ma((\mu_s)_{s\in S})$ is onto (since for any $\eta\in\p(E^S)$, $(\pr^S)_\#(\eta\otimes(\otimes_{s\in \R\setminus S}\mu_s))=\eta$), (1) and (2) imply that $\mathcal{L}_S\neq\varnothing$; by Properties (3) and (5), $\mathcal{N}_{s_1,s_2}\circ\cdots\circ\mathcal{N}_{s_{d-1},s_d}$ is weakly closed in $\p(E^S)$ hence so is $\mathcal{L}_S$.

By \cite[Subsection 1.2]{Ke72} or \cite[Section 2]{BHS}, $\ma((\mu_t)_{t\in \R})$ is weakly compact, so that $\mathcal{L}_\R\eqdef\cap_{S\subset \R\ \text{finite}}\mathcal{L}_S\neq\varnothing$.  Indeed, else, $\ma((\mu_t)_{t\in \R})$ would be covered by the union of open sets $\cup_S(\R\setminus\mathcal{L}_S)$, so that $\mathcal{L}_{S_1}\cap\cdots \cap \mathcal{L}_{S_N}=\varnothing$ for some $N$-tuple $(S_i)_{i\leq N}$ of finite sets. But by (4), $S\subset S'\Rightarrow \mathcal{L}_{S'}\subset \mathcal{L}_{S}$, hence $\mathcal{L}_{S_1}\cap\cdots \cap \mathcal{L}_{S_N}\supset\mathcal{L}_{S_1\cup\cdots\cup S_N}\neq\varnothing$, a contradiction.

Finally take some $P\in\mathcal{L}_\R$. For every finite $S=\{s_1,\ldots,s_d\}$, $(\pr^S)_\#P\in\mathcal{N}_{s_1,s_2}\circ\cdots\circ\mathcal{N}_{s_{d-1},s_d}$ hence $P$ is a Markov measure.
\end{proof}

\subsection{Relation to the Markov-quantile process}\label{subsec:markovinification}

What precedes provides also, through the application of Theorem \ref{them:existence}, the following existence theorem for Markov processes being limits of products of a given process. When applied to the quantile measure $\como\in\ma((\mu_t)_t)$ introduced in \S\ref{ss:minimality}, it provides the existence part of Theorem \ref{them:a}, see below.

\begin{them}[Markovinification]\label{them:sous_n}\maz
Take $P\in\p(\R^\R)$. If for each $s$ and $t>s$, $P^{s,t}$ has increasing kernel,  there exists a Markov measure $P'$ in $\ma((\mu_t)_t)$ such that each $P^{\prime s,t}$ is a limit of products $P^{s,r_1}.\cdots.P^{r_m,t}$ with $\{r_1,\ldots,r_m\}\subset\op]s,t\clo[$. One may take $P'$ such that for each $(s,t)$, the limit is obtained with a sequence $(\{r^n_1,\ldots,r^n_{m(n)}\})_n$ such that $\max_{k=0}^{m(n)} |r_{k+1}-r_k|\to_{n\to\infty}0$, where $(r_0,\linebreak[1]r_{m+1})$ stands for $(s,t)$.
\end{them}

\begin{proof} If $t>s$, setting $P_{[R]}^{s,t}\eqdef P^{s,r_1}.\cdots.P^{r_m,t}$ for all $R=\{r_1,\ldots,r_m\}\subset\op]s,t\clo[$, we introduce $\mathcal{N}_{s,t}^{\rm (P)}=\cap_{\sigma>0}\overline{\{P_{[R]}^{s,t}:\,\max_{k=0}^m |r_{k+1}-r_k|\leqslant\sigma\}}$, where $(r_0,\linebreak[1]r_{m+1})$ stands for $(s,t)$. It is included in $\mathcal{N}_{s,t}^\IT$ by Lemma  \ref{lem:transitions_croissantes_limite}, and it satisfies Assumptions (1), (3) and (4) of Theorem \ref{them:existence}. Indeed, for (1), $\mathcal{N}_{s,t}^{\rm (P)}\neq\varnothing$ as an intersection of nested non-empty compact (closed in the compact space $\ma(\mu_s,\mu_t)$\ ) sets; (3) is true by definition, and (4) by Proposition \ref{pro:continuity}. Thus by Remark \ref{rem:otherNs}(a), $\mathcal{N}_{s,t}^{\rm (P)}$ satisfies all the assumptions of Theorem \ref{them:existence}. We are done. (Notice that the alternative definition $\mathcal{N}_{s,t}^{\rm (P)}=\overline{\{P_{[R]}^{s,t}\}}$ would have given the same result, except its last sentence.)
\end{proof}

Note that if $P$ is Markov the spaces $\mathcal{N}_{s,t}^{\rm (P)}$ and $\overline{\{P_{[R]}^{s,t}\}}$ are both reduced to $\{P^{s,t}\}$, so that the Markov measure obtained from any of them is $P$ itself. This conservation property also holds locally on intervals $I\subset \R$ if $(P_t)_{t\in I}$ is Markov. Notice also that Theorem \ref{them:sous_n}  does not require $(\mu_t)_{t\in\R}$ to be increasing for $\leqs $.

Theorem \ref{them:sous_n} links \S\ref{sec:catenation+kellerer+boujui} with the Markov-quantile process $\mq$ built in \S\ref{sec:construction}. Indeed, taking $P=\como$, 
 $\como^{s,t}$ is in $\mathcal{N}_{s,t}^\IT$ for all $s<t$ by Remark \ref{rem:quantile_transitions_croissantes}, so Theorem \ref{them:sous_n} gives the existence of a Markov process with 2-marginals in $\mathcal{N}_{s,t}^{\rm (\como)}$ (here equal to $\overline{\{\como_{[R]}^{s,t}\}}$). We prove in \S\ref{sec:construction}, by completely different means, that:\smallskip

--  this process is unique,\smallskip

-- it may be built using the order $\leqs$ (see also Remark \ref{rem:choix}), instead of being obtained by a non-constructive compactness argument.\smallskip

\noindent This is Theorem \ref{them:a}. See also Open question \ref{ss:marko}.

\section{Three auxiliary notions, and postponed proofs of three lemmas}\label{sec:auxiliaire}

The next section introduces the notions needed to prove the results of \S\ref{sec:construction} below. They are also necessary for the proofs of three lemmas that were therefore postponed: Lemma \ref{coro} on versions of increasing processes, the important Lemma \ref{lem:closed} on the continuity of $\circ$ when the kernels are increasing, and Lemma \ref{lem:transitions_croissantes_limite}.

\subsection{Lower orthant and stochastic orders, related suprema, and increasing kernels}\label{subsec:stochastic_order}

\begin{notation}\label{notation:intervallesRd}
\point\ Let us denote $(x_i)_{i=1}^d$ and $(y_i)_{i=1}^d$ in $\R^d$ by $x$ and $y$. We endow $\R^d$ with the natural partial order defined by:
$$x\leqslant y\ \text{if: }\forall i,x_i\leqslant y_i.$$

We also set $[x,y]\eqdef\{z\in\R^d:x\leqslant z\leqslant y\}=\prod_i[x_i,y_i]$ and similarly $]x,y]$ {\em etc.} In particular $\op]-\infty,x]=\op]-\infty,x_1]\times\cdots\times\op]-\infty,x_d]$. 

\point\ Several times appear statements where some intervals have to be considered closed or open at some of their bounds, either arbitrarily or depending on possible cases. To alleviate the writing, we introduce the symbol ``$\rsemibracket$'' and place it at those bounds.
\end{notation}

\begin{defi}\label{def:cumul}
If $\mu\in\m(\R^d)$, its cumulative distribution function $F_\mu$, that we also denote by $F[\mu]$ to avoid multiple subscripts, is defined, using Notation \ref{notation:intervallesRd}, by:
$$F_\mu:x\in\R^d\mapsto\mu(\op]-\infty,x]).$$
\end{defi}

\begin{reminder}\maz\label{remind:F}Recall for instance from \cite[Theorem 3.25]{Kal} that such functions $F$ are characterized by the fact that:\smallskip
\point\label{item:Fonctions_partielles} for the natural partial order of $\R^d$ (see Notation \ref{notation:intervallesRd}), $F$ is increasing and upper semi-continuous in the sense that for all $x\in\R^d$:
\begin{equation}\label{eq:usc}
\forall\varepsilon>0,\exists\eta>0,\forall y\geqslant x, \|y-x\|_\infty\leqslant\eta\Rightarrow |F(y)-F(x)|\leqslant \varepsilon,\smallskip
\end{equation}

\point\label{item:Flimite} $\lim_{\min_i(x_i)\rightarrow-\infty} F(x)=0$ and $\lim_{\min_i(x_i)\rightarrow+\infty} F(x)=1$,\smallskip

\point\label{item:Fdifferences_croisees} for every $h=(h_1,\ldots,h_d)\in[0,+\infty\clo[^d$ and $x
\in \R^d$, the quantity $\sum_{\eps} \sigma(\eps) F(x+\eps h)$, which is the measure of the rectangle $]x,x+h]\subset \R^d$, is non-negative. Here $\eps=(\eps_1,\ldots,\eps_d)$ ranges over $\{0,1\}^d$, $\sigma$ is $1$ if $\sum \eps_i$ is even, $-1$ otherwise, and $\eps h$ means $(\eps_1h_1,\ldots,\eps_dh_d)$.
\end{reminder}

\begin{defi}\label{defi:geqlc} If $d\in\N^\ast$ and $m\in\op]0,+\infty\clo[$, following \cite[Section 6.G]{SS}, we define the \emph{lower orthant order} on $\{\mu\in\m(\R^d):\,\mu(\R^d)=m\}$ by: $\mu\leqlc\nu$ if $F_\mu\geqslant F_\nu$. 
\end{defi}

\begin{defi}\label{def:lcsup}We call lower orthant supremum of a family $(P_\tau)_{\tau\in\Tt}$ of measures of same mass $m$ on $\R^d$, the smallest upper bound of $\{P_\tau\} _\tau$ for $\leqlc$, if it exists, i.e.\ a measure $P$ of mass $m$ such that:
\begin{itemize}
\item[--] for every $\tau$, $P_\tau\leqlc P$,
\item[--] $P\leqlc Q$ as soon as $P_\tau\leqlc Q$ for every $\tau$.
\end{itemize}
By definition, if it exists it is unique. We denote it by $\lcsup_\tau P_\tau$. Similarly we define $\lcinf_\tau P_\tau$.

\end{defi}

\begin{rem}\label{rem:lcsup}\maz
\point\label{p:lcsupa}\ In Reminder \ref{remind:F}, \ref{item:Fonctions_partielles} and the first limit of \ref{item:Flimite} pass to the infimum of functions that are both monotone and upper semi-continuous. To see it notice that for such functions upper semi-continuity (\ref{eq:usc}) reads:
\begin{equation*}
\forall\varepsilon>0,\exists\eta>0,\forall y\geqslant x, \|y-x\|_\infty\leqslant\eta\Rightarrow F(x)\leqslant F(y)\leqslant F(x)+\varepsilon.
\end{equation*}
 If moreover $(P_\tau)_\tau$ has an upper bound $P$, then the second limit of \ref{item:Flimite} holds. Indeed, for all $\tau$, $F[P_\tau]\geqslant F[P]$, so that $\inf_\tau F[P_\tau]\geqslant F[P]$, and besides $\lim_{\min_i(x_i)\rightarrow+\infty}F[P](x)=1$. Thus, if $(P_\tau)_\tau$ is bounded from above, $\inf_\tau F[P_\tau]$ satisfies \ref{item:Fonctions_partielles}--\ref{item:Flimite} of Reminder \ref{remind:F}, so  is a cumulative distribution function if and only if it satisfies \ref{item:Fdifferences_croisees}.\\
\point\label{p:lcsupb}\ If \ref{item:Fdifferences_croisees} of Reminder \ref{remind:F} is satisfied by the functions $F[P_\tau]$, then $(P_\tau)_\tau$ has a lower orthant supremum, and $F[\lcsup_\tau P_\tau]=\inf_\tau F[P_\tau]$.
\end{rem}

\begin{remnotation}\label{nota:stosup} If $d=1$ the order $\leqlc$ is usually called stochastic order and denoted by $\leqs$; we will then call ``stochastic supremum'' the lower orthant supremum of Definition \ref{def:lcsup} and denote it by $\stosup$.
\end{remnotation}

\begin{lem}[Existence criteria for $\lcsup$]\maz\label{lem:sto_order}
\point\label{item:suite_croissante_majoree} If, for $\leqlc$, a sequence\linebreak[4] $(P_n)_{n\in \N}\in(\p(\R^d))^\N$ is bounded from above, and increasing, i.e.\ $n\leqslant m$ $\Rightarrow$ $P_{n}\leqlc P_{m}$, then $\lcsup_nP_n$ exists and  $(P_n)_{n}$ converges weakly to it.

\point\label{item:treillis} If a family $(P_\tau)_{\tau\in\Tt}$ is bounded from above for $\leqlc$ and if for every $\tau,\tau'\in \Tt$ there exists $\sigma\in \Tt$ such that $P_\sigma\geqlc P_\tau$ and $P_\sigma\geqlc P_{\tau'}$, then $\lcsup_\tau P_\tau$ exists and there is an increasing sequence $(P_{\tau_n})_n$ 
that converges weakly to it.

The results extend in an obvious way to measures of mass $m>0$ in $\m(\R).$
\end{lem}
\begin{proof} {\bf\ref{item:suite_croissante_majoree}}
Set $F\eqdef\inf_nF[P_n]$. After Remark \ref{rem:lcsup}, showing that $F$ satisfies \ref{item:Fdifferences_croisees} of Reminder \ref{remind:F} ensures the existence of $\lcsup_nP_n$. Consider $M\eqdef\sum_{\eps} \sigma(\eps) F(x+\eps h)$ as in Reminder \ref{remind:F}\ref{item:Fdifferences_croisees}. Since $F[P_n]$ is decreasing, $F$ is its simple limit, so $M$ is the simple limit of $\sum_{\eps} \sigma(\eps) F[P_n](x+\eps h)$, hence $M\geqslant0$. We are done. The weak convergence is given by the pointwise convergence of the cumulative distribution functions, see Reminder \ref{remind:billingsley}.\smallskip

{\bf\ref{item:treillis}} This is a diagonal construction. Let $\mathcal{C}=\cup_{k\in\N}\{x_k\}$ be a countable dense set in $\R^d$. Set $F\eqdef\inf_{\tau\in\Tt}F_\tau$. Then for every $(k,n)\in\N^{\ast2}$ we find $\tau_{k,n}\in \Tt$ such that $F[P_{\tau_{k,n}}](x_k)\leqslant F(x_k)+\frac1n$. For each $n$, by a finite induction using the assumption of \ref{item:treillis} on the $P_{\tau_{k,n}}$, we find $\sigma_{n}\in \Tt$ such that: $\forall k\leqslant n,\, F[P_{\sigma_{n}}](x_k)\leqslant F(x_k)+\frac1n$ and $(P_{\sigma_n})_n$ is increasing. Hence:
\begin{equation}\label{eq:convergence_sur_partie_dense}
\forall x\in \mathcal{C},\, F[P_{\sigma_{n}}](x)\rightarrow F(x).
\end{equation}
By (a), $P=\lcsup_n P_{\sigma_n}$ exists and
\begin{equation}\label{eq:conv_sur_Rd}
\forall x\in \R^d,\, F[P_{\sigma_{n}}](x)\rightarrow F[P](x).
\end{equation}
Let us prove that (\ref{eq:convergence_sur_partie_dense}) holds for any $x\in\R^d$, so that $F=F[P]$. Assume by contradiction that, for some $x$, $F(x)<\inf_nF[P_{\sigma_n}](x)$, i.e., by definition of $F$, that for some $\tau\in\Tt$, $F(x)<F[P_\tau](x)<\inf_nF[P_{\sigma_n}](x)=F[P](x)$. Now $F[P_\tau]$ is upper semi-continuous so in a neighbourhood ${\cal U}$ of $x$ in $[x,+\infty\clo[$, $F[P_\tau]<F[P]$. But on the dense set ${\cal C}$,  by \eqref{eq:convergence_sur_partie_dense} and \eqref{eq:conv_sur_Rd}, $F[P]=F$, hence on ${\cal U}\cap{\cal C}\neq\varnothing$, $F[P_\tau]<F$, a contradiction.

Finally $F=F[P]$, so $F$ is  a cumulative distribution function, so by Remark \ref{rem:lcsup}\ref{p:lcsupb}, $P=\lcsup_{\tau\in\Tt}P_{\tau}$. Moreover $P_{\sigma_n}\rightarrow P$. 
\end{proof}

\begin{rem}\label{rem:stosup_existe} {\bf(a)} (Case $d=1$) In this case, Reminder \ref{remind:F}\ref{item:Fdifferences_croisees} is automatically true. Hence, in Lemma \ref{lem:sto_order}, \ref{item:suite_croissante_majoree} is true for any bounded $(P_n)_n$, increasing or not, hence \ref{item:treillis} shows that any ${\cal S}\subset\p(I)$ bounded from above has a stochastic supremum (which has though not to be the weak limit of a sequence of elements of ${\cal S}$, consider $\frac12(\delta_{1}+\delta_{2})=\stosup\bigl\{\frac12(\delta_{0}+\delta_{2}),\delta_{1}\bigr\}$). Symmetrically, a family bounded from below has a stochastic infimum.

{\bf(b)} Point (a) is false for $d>1$. Consider, e.g.,  ${\cal S}=\{P_1,P_2\}=\bigl\{\frac12(\delta_{(1,0)}+\delta_{(0,1)}),\frac12(\delta_{(0,0)}+\delta_{(2,2)})\bigr\}\subset\p(\R^2)$, then $\inf\{F[P_1],F[P_2]\}$ does not satisfy \ref{item:Fdifferences_croisees} of Reminder \ref{remind:F} and ${\cal S}$ has no lower orthant supremum: both $P\eqdef\frac12(\delta_{(1,1)}+\delta_{(2,2)})$ and $P'\eqdef\frac12(\delta_{(0,2)}+\delta_{(2,0)})$ are upper bounds for it but no upper bound $P''$ satisfies $P''\leqlc P$ and $P''\leqlc P'$ (observe $F[P]$ and $F[P']$).
\end{rem}

\begin{rem}In the following we use several times the Lebesgue differentiation theorem for Borel measures; a reference is , e.g.,  \cite[\S2.8--2.9]{Fe}.
\end{rem}

\begin{prodefinition}\label{prodef:transitions}\maz
Take $\mu$ and $\nu$ in $\p(\R)$. We say that a transport plan $P\in\ma(\mu,\nu)$ has {\em increasing kernel} if one (and then any) of the following statements holds:

\point\label{item:transitions_croissantes_application_croissante} Initial definition: if $\theta,\theta'\ll \mu$ and $\theta$ and $\theta'$ have he same mass, then $\theta\leqs \theta'$ implies $\theta .P\leqs \theta'. P$.

\point\label{pt:condexp} For every increasing $h:\R\to[0,1]$, $P.h$ is $\mu$-almost surely increasing, i.e., more exactly, there is an increasing $\tilde{h}:\R\to [0,1]$ such that, for every bounded continuous function $g$:
\[\int g(x)h(y) \dd P(x,y)= \int g(x)\tilde{h}(x)\dd \mu(x).\]

\point\ There exists a kernel $k$ in the $\mu$-equivalence class of $k_P$ such that $x\mapsto k(x,\cdot)$ is increasing from $(\R,\leq)$ to $(\p(\R),\leqs)$.

\point\ There exists a random vector $(X,Y)$ with $\law(X,Y)=P$ such that $x\in\R\mapsto\law(Y|\,X=x)$ is increasing from $(\R,\leq)$ to $(\p(\R),\leqs)$ (in the sense of the $\mu$-equivalence classes of {\small increasing} functions: it is increasing on a set of full measure $F\subset \R$).

\end{prodefinition}

\begin{rem}Be cautious that having increasing kernel is distinct from being an increasing coupling, a notion defined in Definition \ref{defi:martingale}\ref{item:increasing_coupling}.
\end{rem}

\begin{proof}[Proof of the equivalence in Proposition \ref{prodef:transitions}]
Statements (c) and (d) are essentially a change of notation (To get (c) from (d), notice that for $y$ in the $\mu$-null set $\R\setminus E$ of (d), $k(y,\cdot)$ can be defined as $\stosup_{x\in E,\,x<y} k(x,\cdot)$) and (c) $\Rightarrow$ (b) follows from the definition of $P.h$. Let us show (b) $\Rightarrow$ (a) $\Rightarrow$ (c).

{\bf\mathversion{bold}(b) $\Rightarrow$ (a)\mathversion{normal}}: if $\theta\leqs \theta'$, take $h:\R\rightarrow[0,1]$ increasing, then $(\theta .P).h= \theta .\tilde h\leq \theta' .\tilde h$ since $\theta\leqs \theta'$ and $\tilde h$ is increasing by (b). Then $\theta'.\tilde h=(\theta'.P). h$ yields (a). 

{\bf\mathversion{bold}(a) $\Rightarrow$ (c)\mathversion{normal}}: Suppose (a) and set $I_q\eqdef\op]-\infty,q]$ for all $q$. We will build $R\subset\R$, with $\mu(R)=1$, on which  $x\leqslant x'$ $\Rightarrow$ $k(x,I_q)\geqslant k(x',I_q)$ for all $q\in\Q$, hence all $q\in\R$, ensuring (c). By definition of the kernel $k_P$, $k_P(\,.\,,I_q)$ is the density with respect to $\mu$ of the measure $B\mapsto P(B\times I_q)$, thus by the Lebesgue differentiation theorem, setting $r(x,\eps)\eqdef\frac{P([x-\eps,x+\eps]\times I_q)}{\mu([x-\eps,x+\eps])}$, the function $x\mapsto\lim_{\eps\rightarrow0}r(x,\eps)$ is $\mu$-almost everywhere defined and equal to it. Now $r(x,\eps)=\bigl(\bigl(\frac1{\mu([x-\eps,x+\eps])}\mu_{\lfloor[x-\eps,x+\eps]}\bigr).P\bigr)(I_q)$, thus by (a), $x\leqslant x'$ $\Rightarrow$ $r(x,\eps)\geqslant r(x',\eps)$. Hence for all $q\in\Q$ there is some $R_q\subset \R$ with $\mu(R_q)=1$ on which $x\leqslant x'$ $\Rightarrow$ $k(x,I_q)\geqslant k(x',I_q)$. Finally set $R\eqdef\cap_{q\in\Q} R_q$.
\end{proof}

\begin{notation}\label{notation:mdec}
If $\mu\in\m(\R)$ we denote by $\m(\mu)$ and $\p(\mu)$ the sets of positive measures, respectively probability measures, absolutely continuous with respect to $\mu$, and $\mdec(\mu)$ and $\pdec(\mu)$, or $\mdec$ and $\pdec$ if there is no ambiguity, their subsets of measures with bounded and decreasing density. 
\end{notation}

\begin{rem}\label{rem:transitions_croissantes_composee} A direct consequence of the first point of Proposition \ref{prodef:transitions} is that, if $P\in\ma(\mu_1,\mu_2)$ and $Q\in\ma(\mu_2,\mu_3)$ have increasing kernel, so has their product $P.Q$.
\end{rem}

\begin{rem}[$\mdec(\mu)$ is closed for the weak topology]\label{rem:mdec} A measure $\theta$ belongs to $\mdec(\mu)$ if and only if its density is bounded and, for all $(a,b,c,d)\in\R^4$:
\begin{equation}\label{eq:reformulation_mdec}
\bigl(a<b<c<d\ \text{and}\ \mu([a,b]).\mu([c,d])\neq0\bigr)\Rightarrow \frac{\theta([a,b])}{\mu([a,b])}\geqslant \frac{\theta([c,d])}{\mu([c,d])}.
\end{equation}
``Only if'' is clear. For the ``if'' part, by the Lebesgue differentiation theorem, $x\mapsto\lim_{\eps\rightarrow0}\frac{\theta([x-\varepsilon,x+\varepsilon])}{\mu([x-\varepsilon,x+\varepsilon])}$ provides a representative of the density. Now if a sequence $(\theta_n)_n$ satisfies (\ref{eq:reformulation_mdec}) and weakly tends (see Reminder \ref{remind:billingsley}) to $\theta\in\m(\mu)$, $\theta$ satisfies it also (if $a$, $b$, $c$ or $d$ is an atom of $\theta$, re-obtain (\ref{eq:reformulation_mdec}) by limit of larger intervals).
\end{rem}

\begin{rem}\label{rem:stable} ($P$ has increasing kernel if and only if $\trans\!P$ maps $\mdec(\nu)$ to $\mdec(\mu)$).
In Proposition \ref{prodef:transitions}, (b) is equivalent to the same statement with decreasing functions; in turn, transposing, this means that, for all decreasing $h:\R\to[0,1]$, $(h\nu).\trans\!P$, which is equal to $(P.h).\mu$, has decreasing density.
\end{rem}

\subsection{Quantile measures and minimal couplings}\label{ss:minimality}
We define the quantile coupling (Definition \ref{defi:quantile1}) and quantile process law (Definition \ref{defi:quantile_process}) through a minimality property that is crucial in our paper. We also state the direct and more classical Definition \ref{defi:quantile2} of the quantile measure. Further characterizations of these coupling and process are given throughout the paper, in particular in \S\ref{sec:cont_eq} where the approach is optimal transport. The reader can refer to \cite{RR1, RR2, Vi1} for more background. See also the papers \cite{Ju_seminaire, Pass}.

\begin{reminder}Take $\mu$ and $\nu$ in $\m(\R)$. Every $P$ of $\ma(\mu,\nu)$ satisfies the Hoeffding-Fr\'echet bound:
\begin{align}\label{eq:HF}
\forall(x,y)\in\R^2,F_P(x,y)\leq \min(F_\mu(x),F_\nu(y)).
\end{align}
\end{reminder}

\begin{prodefinition}\label{defi:quantile1}There is a unique $P$ such that \eqref{eq:HF} is an equality. We call it the Fr\'echet-Hoeffding, comonotonic or {\em quantile coupling} and denote it by $\como(\mu,\nu)$. Said briefly: $\como(\mu,\nu)=\lcinf \ma(\mu,\nu)$.
\end{prodefinition}

The proof follows from Definition \ref{defi:quantile2}, see below.

\begin{rem}[Minimality in the language of transport plans]\label{rem:q_is_mini}

{\bf(a)} The quantile process $\como$ may also be defined by the fact that its transitions are \emph{minimal} among those of all transport plans of $\ma(\mu,\nu)$: for any $P\in\ma(\mu,\nu)$ and fixed $x$,  $F[\mu\lfloor_{\op]-\infty,x]}.k_P](y)= F_P(x,y)$ for every $y\in \R$, hence, after Definition \ref{defi:quantile1} and the characterization of stochastic order of Remark \ref{rem:stosup_existe}:
\begin{equation}\label{eq:minimalite_des_transitions}\forall P\in\ma(\mu,\nu),\,\mu\lfloor_{\op]-\infty,x]}.\como(\mu,\nu)\leqs\mu\lfloor_{\op]-\infty,x]}.k_P
\end{equation}
(in fact, $\mu\lfloor_{\op]-\infty,x]}.\como(\mu,\nu)$ is of the type $\nu\lfloor_{\op]-\infty,y[}+a\delta_y$), i.e.\ the quantile coupling maps the measures $\mu\lfloor_{\op]-\infty,x]}$ to the stochastically smallest possible measures, \eqref{eq:minimalite_des_transitions} being also true for any $\theta\in \mdec(\mu)$ in place of $\mu\lfloor_{\op]-\infty,x]}$.

Notice that a property similar to \eqref{eq:minimalite_des_transitions}, with $\leqc$ in place of $\leqs$, defines the (left\nobreakdash-)curtain coupling in \cite{BJ}.

{\bf(b)} As $\mu\lfloor_{\op]x,+\infty\clo[}.k_P=\nu-\mu\lfloor_{\op]-\infty,x]}.k_{P}$, $\mu\lfloor_{\op]x,+\infty\clo[}.\como(\mu,\nu)$ is paradoxically {\em maximal} for $\leqs$, hence minimal transitions mean that the mass of $\mu$ is mixed as less as possible when transported on that of $\nu$. \end{rem}

\begin{prodefinition}\label{defi:quantile_process}If $(\mu_\tau)_{\tau\in\Tt}$ is a family of measures, there is a unique measure $\como\in \ma((\mu_{\tau})_{\tau\in\Tt})$ such that for every $\tau\neq\sigma$, the transport plan $\como^{\tau,\sigma}=(\pr^{\{\tau,\sigma\}})_\#\como$ is the quantile coupling $\como(\mu_\tau,\mu_\sigma)$. We call it the {\em quantile measure} of $(\mu_\tau)_{\tau\in\Tt}$.
\end{prodefinition}

\begin{proof}[Proof of Propositions \ref{defi:quantile1} and \ref{defi:quantile_process}.] The existence parts follows from Definition \ref{defi:quantile2} below and are proved just after it; in Proposition \ref{defi:quantile1} uniqueness is clear; let us prove it in Proposition \ref{defi:quantile_process}. It is rather easy to prove that the equality in \eqref{eq:HF} for every pair of measures implies the equality in the Hoeffding-Fr\'echet bound of general dimension:
\begin{align}\label{eq:HF2}
\forall(x_i)_{i=1}^d\in\R^d,F_{P^S}(x_1,\ldots,x_d)\leq \min_{i\in\{1\,\ldots,d\}}(F_{\mu_{s_i}}(x_i)),
\end{align}
where $S=\{s_1,\ldots,s_d\}$ is any finite subset of $\R$. (Take $j$ such that $F_{\mu_{s_j}}(x_j)=\min_{1\leq i\leq d}(F_{\mu_{s_i}}(x_i))$. For every $k\neq j$, $F[\pr^{\{s_j,s_k\}}_\#P^S](x_j,x_k)=F[\mu_{s_j}](x_j)$ so that finally $P^S(\op]-\infty,x])=F[\mu_{s_j}](x_j)$.) Since equalities \eqref{eq:HF2} for all finite $S\subset\R$ are a compatible set of conditions, this proves with Proposition \ref{pro:inductive_limit} that there exists at most one quantile measure $P$ in $\ma((\mu_\tau)_{\tau\in\Tt})$.
\end{proof}

\begin{notation}\label{not:multi} We denote by $\como(\mu_{s_1},\ldots,\mu_{s_d})$ and $\como((\mu_t)_{t\in \R})$ the multidimensional quantile coupling and the quantile measure. 
\end{notation}

\begin{rem}
Definition \ref{defi:quantile_process} uses no order on the set $\Tt$ of indices to define a quantile measure. So if $\Tt=\R$, the order of the marginals does not matter; bijections of $\R$ act naturally on the quantile measures and their marginals. But because the Markov property is based on the order on $\Tt=\R$, it will be different for our Markov-quantile measure. Only monotone bijections act naturally, see Example \ref{ex:reversed}.
\end{rem}

Now here is the definition of the quantile measure through the {\em quantile function}. It ensures the existence of $\como(\mu,\nu)$ in Definition \ref{defi:quantile1}.

\begin{defi}\label{defi:quantile2} Take $\mu\in \p(\R)$. The {\em quantile function} $G_\mu:[0,1]\to \R\cup\{-\infty,+\infty\}$ is the increasing left-continuous function such that $(G_\mu)_\#\Lg=\mu$, i.e.\  the generalized inverse of the cumulative distribution function $F_\mu$:
$$G_\mu(q)=\inf\{x\in\R:\,F_\mu (x)\geq q\}.$$
The quantile measure $\como((\mu_\tau)_{\tau\in\Tt})$ is obtained by pushing forward $\Lg$ on $\p(\R^{\Tt})$ by the map $G=(G_{\mu_\tau})_{\tau\in \Tt}:x\in [0,1]\mapsto (G_{\mu_{\tau}}(x))_{\tau\in\Tt}\in \R^{\Tt}$.
In other words, the functions $G_{\mu_{\tau}}$ can be seen as random variables defined on $([0,1],\Lg)$ and $\como((\mu_\tau)_{\tau\in\Tt})$ is the law of the process $(G_{\mu_{\tau}})_{\tau\in\Tt}$. One can check from this definition the equality in \eqref{eq:HF} and \eqref{eq:HF2}.
\end{defi}

\begin{rem}\label{rem:quantile_up}
When $(\mu_{\tau})_{\tau\in \Tt}=(\mu,\nu)$, $\law (G_\mu,G_\nu)=\como(\mu,\nu)$. Now $\mu\leqs \nu\Leftrightarrow F_\mu\geq F_\nu \Leftrightarrow G_\mu\leq G_\nu$, hence $\mu\leqs \nu$ if and only if $\como(\mu,\nu)$ is an increasing coupling, i.e.\ concentrated on $\{(x,y)\in \R^2:\,x\leq y\}$.
\end{rem}

\begin{rem}\label{rem:quantile_transitions_croissantes}(Products of quantile couplings have increasing kernel and map $\mdec$ on $\mdec$). We will see in \S\ref{subsec:little_t} that $\como(\mu,\nu)$ is a composition of two transport plans with increasing kernel. Therefore by Remark \ref{rem:transitions_croissantes_composee} it has increasing kernel. Since $\trans \como(\mu,\nu)=\como(\nu,\mu)$, Remark \ref{rem:stable} ensures that $\como(\nu,\mu)$ maps $\mdec(\mu)$ to $\mdec(\nu)$.

All this also ensure both properties for products of quantile couplings.

\end{rem}

\subsection{Distances $\rho$ and $\tilde \rho$.}\label{subsec:rho}

\subsubsection{A distance that metricizes $\ma(\mu,\nu)$}

\begin{reminder}\label{remind:billingsley}We remind the reader of the ``Portmanteau theorem'' (see, e.g.,  \cite[Theorem 2.1]{Bi}): the weak convergence on some metric space $E$ endowed with its Borel $\sigma$-algebra is (equivalently) defined by:
\begin{equation}\label{eq:defi-weakconvergence}P_n\underset{n\rightarrow\infty}{\longrightarrow}P\ \text{if }P_n(R)\underset{n\rightarrow\infty}{\longrightarrow}P(R)\ \text{for all $R$ such that }P(\partial R)=0.
\end{equation}
In $\R^n$, it is equivalent to consider in (\ref{eq:defi-weakconvergence}) only sets  $R$ of the form $\prod_{i=1}^n\op]-\infty,x_i\clo]$, see Example 2.3 p.\ 18 of \cite{Bi}. 
\end{reminder}

\begin{pronotation}\label{pro:rho} Let \mb$\rho$\mn~be the function defined on $\p(\R^d)^2$ by:
$$\rho(P,Q)=\|F_P-F_Q\|_\infty.$$
(See Definition \ref{def:cumul} for $F_P$ and $F_Q$.) This is a distance. If $\mu_1$,\ldots,$\mu_d$ are probability measures on $\R$, it induces the weak topology on $\ma((\mu_i)_{1\leq i\leq d})$. More precisely, for any $P\in\ma((\mu_i)_{1\leq i\leq d})$ and any sequence $(P_n)_{n\in\N}$ of elements of $\ma((\mu_i)_{1\leq i\leq d})$, the following are equivalent:\smallskip

{\bf (i)} For all $x=(x_i)_{i=1}^d$, $P\bigl(\partial\bigl( \prod_{i=1}^d\op]-\infty,x_i\clo]\bigr)\bigr)=0$ $\Rightarrow$ $\lim_nF_{P_n}(x)=F_{P}(x)$,\smallskip

{\bf (ii)} For all $x=(x_i)_{i=1}^d$, $\lim_nF_{P_n}(x)=F_{P}(x)$,\smallskip

{\bf (ii')} For all $x=(x_i)_{i=1}^d$, $\lim_n\bigl(P_n\bigl(\prod_i\op]-\infty,x_i\rsemibracket\bigr)\bigr)=P\bigl(\prod_i\op]-\infty,x_i\rsemibracket\bigr)$,\smallskip

{\bf (iii)}  $(F_{P_n})_{n\in \N}$ converges uniformly to $F_P$.
\end{pronotation} 

\begin{proof} The fact that $\rho$ is a distance is immediate. To get (iii) $\Rightarrow$ (ii') use that $\op]-\infty,x\clo[=\cup_{y<x}\op]-\infty,y]$. Now it suffices to prove (i) $\Rightarrow$ (ii) $\Rightarrow$ (iii).

{\bf (i)} \mathversion{bold}$\Rightarrow$\mathversion{normal} {\bf (ii)}. Take $Q\in\ma((\mu_i)_{1\leq i\leq d})$ and $b=(b_i)_{i=1}^d$. For $b'=(b'_i)_{i=1}^d\geq b$:
\begin{align}\label{eq:rectangles}
F_Q(b')-F_Q(b)=Q(R')-Q(R)\leq \mu_1(\op]b_1,b'_1])+\cdots+ \mu_d(\op]b_d,b'_d])
\end{align}
where $R'=\op]-\infty,b']$ and $R=\op]-\infty,b]$, which shows that $(F_Q)_Q$ is ``equicontinuous on the right''. Take $\eps>0$. Since each $b_i\in \R$ can be approached from the right by a sequence of non-atomic points for $\mu_i$, there exists $b'\geqslant b$ such that $\mu_1(\op]b_1,b'_1])+\cdots+ \mu_d(\op]b_d,b'_d])<\eps$ and $\mu_1(b'_1)+\cdots+ \mu_d(b'_d)=0$, hence in particular $Q(\partial R')=0$ for every $Q\in \ma((\mu_i)_{1\leq i\leq d})$ so that (i) applies to $R'$ for $P$ and all the $P_n$. Thus:
\begin{align*}
|P_n(R)- P(R)|&\leq|P_n(R)- P_n(R')|+|P_n(R')-P(R')|+|P(R')- P(R)|\\
&\leq 2\eps+|P_n(R')-P(R')|
\end{align*}
and if (i) holds, $|P_n(R')-P(R')|\leqslant\eps$ for $n$ great enough, hence (ii) follows.

{\bf (ii)} \mathversion{bold}$\Rightarrow$\mathversion{normal} {\bf (iii)}. We apply an adapted version of the prior argument. Suppose (ii) and take $\eps>0$. For every $i$ there exists a finite sequence $-\infty=b_0^{(i)}(\eps)<\cdots< b^{(i)}_{N(\eps,i)}(\eps)=+\infty$ avoiding the big atoms of $\mu_i$, i.e.\  so that $\mu_i(\op]b_k^{(i)},b_{k+1}^{(i)}\clo[)< \eps$ (this is classical and is proved, e.g.,  in \cite[Section 12]{Bi}, which deals with the modulus of continuity of c\`adl\`ag paths). Every $R=\op]-\infty,b]$ contains some rectangle $R^-_\eps$ and is included in the interior of some rectangle $R^+_\eps$, both bounded by consecutive points $(b^{(i)}_k)_{i,k}$. Using again \eqref{eq:rectangles} for the first and last terms:
\begin{align*}
|P_n(R)- P(R)|&\leq\underbrace{|P_n(R)- P_n(R_\eps^-)|}_{\leqslant d\eps}+\underbrace{|P_n(R_\eps^-)-P(R_\eps^-)|}_{\underset{n\to\infty}\longrightarrow0\ \text{by (ii)}\quad\text{\bf$(\ast)$}}+\underbrace{|P(R_\eps^-)- P(R)|}_{\leqslant d\eps},
\end{align*}
where {\bf$(\ast)$} is uniform since there are finitely many rectangles $R_\eps^-$. We get (iii).
\end{proof}

\subsubsection{Another expression for $\rho$; an alternative distance $\tilde \rho$}\label{ss:another}

Let $\mu_1,\ldots,\mu_d$ be probability measures on $\R$ and $P$, $Q$ stand for elements of $\ma((\mu_i)_{i=1}^d)$.

\begin{notation} When $f$, $g$ and $(f_i)_{i=1}^d$ are functions, by a slight abuse in this subsection, $f.g$ denotes the function $(x,y)\mapsto f(x)g(y)$ and\/ $\prod_{i=1}^df_i$ the function $(x_1,\ldots ,x_d)\mapsto \prod_{i=1}^df_i(x_i)$
\end{notation}

\begin{rem}\label{rem:product}By definition, $\rho(P,Q)=\sup\left|\int \prod_{i=1}^df_i \,\dd P-\int \prod_{i=1}^df_i \,\dd Q \right|$, where each $f_i$ ranges over $\{\one_{\op]-\infty,x]}:\,x\in\R\}$. But in fact:
\end{rem}

\begin{prodefinition}\label{pro:product} {\bf(a)} For any $P$ and $Q$ in $\ma((\mu_i)_{i=1}^d)$:
\begin{equation}\label{eq:rho_comme_sup}
\rho(P,Q)=\sup\left|\int\prod_{i=1}^df_i \,\dd P-\int \prod_{i=1}^df_i \,\dd Q \right|
\end{equation}
where each $f_i$ ranges over $\{f:\R\to[0,1]: f\text{ is decreasing}\}$.

{\bf(b)} Proposition \ref{pro:rho} may be stated with a distance $\tilde\rho$ based on $x\in \R^n\mapsto P([x,+\infty\clo[)$ in place of $F[P]$. Then $\tilde\rho$ satisfies \eqref{eq:rho_comme_sup} with {\em increasing} functions $f_i$.

\end{prodefinition}

\begin{proof}Once (a) is shown, (b) is clear. To show  (a), by Remark \ref{rem:product}, we have only to prove $\geqslant$. For $f:\R\to [0,1]$ a decreasing function and $I(t)\eqdef f^{-1}([t,+\infty\clo[)$,  $f(x)=\int_0^1 \one_{I(t)}(x)\dd t$, thus for $(f_i)_{i=1}^d$ such functions:
$$\biggl(\prod_{i=1}^df_i\biggr) (x_1,\ldots, x_d)=\int_{t\in[0,1]^d}\one_{R(t)}(x_1,\ldots,x_d)\dd t,$$
where $R(t)=I(t_1)\times\cdots\times I(t_d)$. Therefore:
\begin{align*}
\left|\int \prod_{i=1}^d f_i\,\dd P-\int \prod_{i=1}^d f_i \,\dd Q\right|&=\left|\int_{t\in[0,1]^d}P(R(t))-Q(R(t))\dd t\right|\\
&\leq \int_{t\in[0,1]^d}|P(R(t))-Q(R(t))|\dd t\\
&\leq \rho(P,Q).\qedhere
\end{align*}
\end{proof}

Proposition \ref{pro:product} has a corollary in the case of transport plans ($d=2$).
\begin{pro}\label{pro:multiplicationcontractante}Take $\mu$, $\mu'$ in $\p(\R)$ and $P$, $Q$ in $\ma(\mu,\mu')$. Then:
$$\rho(P,Q)=\sup\{\rho(\theta. P, \theta .Q):\theta\in\mdec(\mu)\ \text{and $\theta$ has density bounded by 1}\},$$
where the distance $\rho$ on the right of the equality is that on $\p(\R)$.

Hence if, for some $\nu\in \p(\R)$, $R\in\ma(\nu,\mu)$ is a transport preserving $\mdec$, i.e.\ $\forall\eta\in\mdec(\nu)$, $\eta.R\in\mdec(\mu)$, then $\rho(R.P,R.Q)\leqslant\rho(P,Q)$.
\end{pro}

\begin{proof}
For the first equality, as $\int f.g\, \dd P=((f\mu).P).g$, by Proposition \ref{pro:product}: $\rho(P,Q)=\sup_{f,g} |((f\mu).P).g-((f\mu).Q).g|$ and $\rho(\theta. P, \theta .Q)=\sup_g |(\theta.P).g-(\theta.Q).g|$. Now $\theta$ is as in the proposition if and only if $\theta=f\mu$ with $f:\R\to[0,1]$, decreasing. The result follows. Then if $R\in\ma(\nu,\mu)$ is as claimed, $\rho(RP,RQ)=\sup\{\rho(\theta. P, \theta .Q)|\,\theta=\bar\theta.R$ where $\bar\theta\in\mdec(\mu)$ and $\bar\theta$ has density bounded by 1$\}$. This set is included in that of the proposition, since the action of $R$ on $\bar\theta$ does not increase the maximum of its density.
\end{proof}

\subsection{Proof of Lemma \ref{coro} and remarks about it}\label{subsec:proof_coro}

We prove Lemma \ref{coro} on the existence of a process consisting exclusively of increasing paths. Our proof requires the use of $\stoinf$ and $\stosup$ introduced in \S\ref{subsec:stochastic_order}.

\begin{proof}[Proof of Lemma \ref{coro}]
Let $(\mu_t)_{t\in \R}$ be an increasing family of probability measures for $\leqs$ and $P$ be a Markov measure in $\ma((\mu_t)_{t\in \R})$ such that for every $S=\{s_1,\ldots,s_d\}$, the measure $(\pr^S)_\#P$ is concentrated on $\{(x_1,\ldots,\linebreak[1]x_d)\in \R^d:\,x_1\leq\ldots \leq x_d\}$.  Set $\mu_{t^-}\eqdef\stosup_{s<t}\mu_s$ and  $\mu_{t^+}\eqdef\stoinf_{s>t}\mu_s$, that are also the left and right limits of $(\mu_t)_t$ for the weak topology, see Remark \ref{rem:stosup_existe}. Since $(\mu_t)_{t\in \R}$ is increasing for $\leqs$, we have $\mu_{t^-}\leqs \mu_t \leqs \mu_{t^+}$ and $t\in\R$ is a discontinuity time of $(\mu_t)_{t\in \R}$ for the weak topology if and only if $\mu_{t^-}\neq\mu_{t^+}$. Such points are at most countably many. Indeed, the regions in $\R^2$ between the graphs of  $F[\mu_{t^-}]$ and $F[\mu_{t^+}]$, for the discontinuity times $t$, are disjoint and of positive Lebesgue measure, hence are at most countably many, by $\sigma$-additivity of the measure.

Let $C$ be a countable dense subset of $\R$ containing the discontinuity points. Introduce:
$$N=\{x\in \R^\R:\,\exists (s,t)\in C^2, s<t\text{ and }x(s)>x(t)\}
.$$
Being a countable union of $P$-null sets, $N$ is $P$-null. Now take $(X_t)_{t\in \R}$ with law $P$, e.g.,  take the canonical process $\Omega\eqdef\R^\R$ and $X_t=\pr^t:x\in \R^\R\mapsto x(t)$. We define $(\tilde X_t)_{t}$ as null functions on $N$, and as follows on $\Omega\setminus N$:\smallskip

-- for $t\in C$, $\tilde X_t(\omega)=X_t(\omega)$,\smallskip

-- for $t\notin C$, $\tilde X_t(\omega)=\lim_{s<t,\,s\in C}X_s(\omega)$.\smallskip

\noindent Hence for every $\omega\in \Omega$, the curve $t\in C\mapsto X_t(\omega)$ is increasing and, even better, $t\in \R\mapsto X_t(\omega)$ is increasing. We are left to prove that $X_t=\tilde X_t$ almost surely. This is clear for $t\in C$. For each $t\in\R\setminus C$, $\{\omega\in\Omega:\exists s\in C, s<t\ \text{and}\ X_s(\omega)>X_t(\omega)\}$ is a union of null sets, hence is null, so almost surely, $X_t\geqslant \sup_{s<t,\,s\in C}X_s=\tilde X_t$. Besides $X_s\to_{s<t,\,s\in C, s\to t}\tilde{X}_t$ almost surely and thus in law, so that $\law(\tilde X_t)=\mu_{t^-}$. Moreover $t$ is a continuity point of $(\mu_t)_t$, so that $\law(\tilde X_t)=\mu_{t}=\law(X_t)$. Thus, $X_t=\tilde{X}_t$ almost surely.
\end{proof}

\begin{rem}\maz \point\ If $(\mu_t)_t$ is moreover left-continuous for the weak topology, we can adapt the proof of Lemma \ref{coro} so that $s\mapsto \tilde X_s(\omega)$ is increasing and left continuous by using the formula:
$$\tilde X_t(\omega)=\lim_{s<t,\,s\in C}X_s(\omega)\ \text{if}\ \omega \in \Omega\setminus N\ \text{and}\  X_t(\omega)=0\ \text{otherwise}.\smallskip$$

\point\ If $(\mu_t)_t$ is this time moreover right-continuous, using a symmetric construction all the curves can be chosen c\`adl\`ag (right-continuous, with limit on the left at any point).\smallskip

\point\  However, for a continuous $\mu=(\mu_t)_t$, the curves $s\mapsto \tilde X_s(\omega)$ have not to be continuous. A simple example is: $t\in [0,1]\mapsto (1-t)\delta_0+t \delta_1$. But see \S\ref{sec:cont_eq} (adapt the proof of Theorem \ref{them:action}\ref{pt:lag}): if $\mu\in \mathcal{AC}_2$, the curves $s\mapsto \tilde X_s(\omega)$ are continuous.\smallskip

\point\ Carrying on with the similarity between $\leqs$ and $\leqc$ established in Theorem \ref{them:c} we mention that Kellerer's theorem has also been revisited under continuity assumptions, see \cite{Low,HiRoYo14,BHS}. In particular it has been proved that if $(\mu_t)_{t\in \R}$ is right-continuous the associated martingale can be defined in the space of c\`adl\`ag paths.
\end{rem}

\subsection{Some remarks following from Proposition \ref{pro:rho} ; proofs of Lemmas \ref{lem:closed} and \ref{lem:transitions_croissantes_limite}}\label{subsec:proof_lemmas}

\subsubsection{The remarks on Proposition \ref{pro:rho}; proof of Lemma \ref{lem:transitions_croissantes_limite}}

\begin{rem}\label{rem:preserver_mdec_ferme}($\{P\in\ma(\mu,\nu):P\text{ maps } \mdec(\mu)\text{ to }\mdec(\nu)\}$ is closed for the weak topology.) Any $\theta\in\mdec(\mu)$ is an increasing limit of positive combinations of characteristic functions $\one_{\op]-\infty,x]}$, so $P\in\ma(\mu,\nu)$ maps $\mdec(\mu)$ in $\mdec(\nu)$ if and only if it maps $\{\mu\lfloor_{\op]-\infty,x]}:\,x\in\R\}$ in it. Now take a sequence $(P_n)_n\in\ma(\mu,\nu)^\N$ of transport plans having this property and converging weakly to $P\in\ma(\mu,\nu)$. After Proposition \ref{pro:rho}, $\|F[P_n]-F[P]\|_\infty\rightarrow0$; in particular, for any $x\in\R$, $\|F[P_n](x,\,\cdot\,)-F[P](x,\,\cdot\,)\|_\infty\rightarrow0$, i.e.\ $\|F[\mu\lfloor_{\op]-\infty,x]}.P_n]-F[\mu\lfloor_{\op]-\infty,x]}.P]\|_\infty\rightarrow0$, i.e.\  $\mu\lfloor_{\op]-\infty,x]}.P_n$ converges weakly to $\mu\lfloor_{\op]-\infty,x]}.P$. By Remark \ref{rem:mdec}, $\mu\lfloor_{\op]-\infty,x]}.P\in\mdec(\nu)$, we are done.
\end{rem}

Now the little Lemma \ref{lem:transitions_croissantes_limite}, which we use several times, is immediate.
\begin{proof}[Proof of Lemma \ref{lem:transitions_croissantes_limite}]\label{proof:lem:transitions_croissantes_limite} Apply Remarks \ref{rem:preserver_mdec_ferme} and \ref{rem:stable}.
\end{proof}

\begin{rem}\label{rem:diese_contractant} Take $(H_i)_{i=1}^d$ increasing real functions, and $H=H_1\otimes\cdots\otimes H_d:\R^d\rightarrow\R^d$. Then $H_\#:\ma((\mu_i)_{i=1}^d )\rightarrow\ma((H_i(\mu_i))_{i=1}^d)$ is contracting for $\rho$. To check it, think that $(F[H_\#P]-F[H_\#Q])(x_1,\ldots,x_d)=(H_\#P\linebreak[2]-H_\#Q)(\prod_i\op]-\infty,x_i])$ equals a term of the type $(H_\#P-H_\#Q)(\prod_i\op]-\infty,y_i\rsemibracket)$, and use Proposition \ref{pro:rho}. Thus if $(P_n)_n$ converges to $P$ in $\ma(\mu_1,\ldots,\mu_d)$, $(H_\#P_n)_n$ converges to $H_\#P$ in $\ma(H_1(\mu_1),\ldots, H_d(\mu_d))$.
\end{rem}

The next remark is neither related to Proposition \ref{pro:rho} nor to Lemma \ref{lem:transitions_croissantes_limite} but is an analogue of Remark \ref{rem:diese_contractant} with $\leqs$ in place of $\rho$.

\begin{rem}\label{rem:leqlc_et_Gs}If $P,Q$ are in $\p(\R^d)$, the $(H_i)_{i=1}^d$ are increasing functions, and $H\eqdef H_1\otimes\ldots\otimes H_d$ then $H_\#:\p(\R^d)\rightarrow\p(\R^d)$ is increasing for $\leqlc$, i.e.\ $P\leqlc Q\Rightarrow H_\#P\leqlc H_\#Q$. To check it, think that $(F[H_\#P]-F[H_\#Q])(x_1,\ldots,x_d)=(H_\#P-H_\#Q)(\prod_i\op]-\infty,x_i])$ equals a term of the type $(H_\#P-H_\#Q)(\prod_i\op]-\infty,y_i\rsemibracket)$, and, for the indices $i$ such that ``$\rsemibracket$'' stands for ``$[$'', use that $\op]-\infty,x_i\clo[=\cup_{n}\op]-\infty,x_i-1/n]$.
\end{rem}

\begin{rem}\label{rem:multiplicationcontractante} Notice then that $\rho(T,T')=\rho(\trans T,{\trans T'})$ (see Definition \ref{defi:transpose} for $\trans T$). Thus, in Proposition \ref{pro:multiplicationcontractante}, if $R\in\ma(\mu',\nu)$ and if $R$ has increasing kernel (i.e., by Remark \ref{rem:stable}, $\trans\!R$ maps $\mdec(\nu)$ into $\mdec(\mu')$):
$$\rho(PR,QR)=\rho(\trans(PR),{\trans (QR)})=\rho(\trans\!R\,\trans\!P),{\trans\!R\,\trans\!Q)})\leqslant\rho(\trans\!P,\trans Q)=\rho(P,Q).$$
\end{rem}

\subsubsection{Proof of Lemma \ref{lem:closed}}\label{subsec:MaInProcess}
First we prove Lemma \ref{lem:pointwise}, then its consequence Proposition \ref{pro:continuity}, and finally Lemma \ref{lem:closed}.

\begin{lem}\label{lem:pointwise}
For every $n\in \N$, let $P_n\in \ma(\mu,\nu)$ have increasing kernel. Suppose moreover that $(P_n)_{n\in\N}$ converges to $P_0$. Let $h:[0,1]\to [0,1]$ be an increasing function and for every $n\in \N$, $\tilde{h}_n=P_{n}.h$. Then $(\tilde h_n)_{n\in\N}$ converges to $\tilde h_0$, $\mu$-almost surely.
\end{lem}
\begin{proof}By the equivalence proven in Proposition \ref{prodef:transitions}, the sequence $(\tilde h_n)_{n\in \N}$ is increasing.

By Propositions \ref{pro:product}(b) and \ref{pro:rho}, $\tilde{\rho}(P_n,P_0)\to_{n\to\infty} 0$, i.e., for every increasing $g$ with values in $[0,1]$, $\int g(y)h(z) \dd P_n(y,z)\to \int g(y)h(z)\dd P_0(y,z)$. Hence:
\begin{equation}\label{eq:lem:pointwise}
\int g(x)\tilde{h}_n(x)\,\dd \mu(x)\to \int g(x)\tilde{h}_0(x)\,\dd \mu(x),
\end{equation}
by definition of $\tilde h_n$. This also holds if $g$ is the difference of two increasing functions, in particular $g=\one_{[a,b]}=\one_{[a,+\infty\clo[}-\one_{[b,+\infty\clo[}$.

Let $A$ be the set of the elements $x$ of $[0,1]$ such that $\mu[x,1]\mu[0,x]>0$ and for every $n\in \N$, $\mu\text{-}\mathrm{esssup}_{[0,x]}\tilde{h}_n=\mu\text{-}\mathrm{essinf}_{[x,1]}\tilde{h}_n=\tilde{h}_n(x)$. In case $\mu=\la$, since $\tilde h_n$ is increasing, hence has at most a countable number of discontinuity points, $\mu(A)$ is $1$. This also holds in the general case and is given by the increasing functions $\tilde{h}_n\circ G_\mu$, we leave the details to the reader. Now we take any $x\in A$ and prove $\tilde h_n(x)\rightarrow\tilde h_0(x)$. It is enough to prove $\limsup_{n}\tilde{h}_n(x)\leq \tilde{h}_0(x)$, since $\liminf_{n}\tilde{h}_n(x)\leq \tilde{h}_0(x)$ can be proved symmetrically.    Suppose, for contradiction:
$$\limsup_{n}\tilde{h}_n(x)\geq \tilde{h}_0(x)+\eps,\quad \text{for some $\eps>0$.}$$
As $\tilde{h}_{0}$ is increasing and $\tilde{h}_0(x)=\mu\text{-}\mathrm{essinf}_{[x,1]}\tilde{h}_0$ there exists $y>x$ such that $\mu[x,y]>0$ and:
$$\tilde{h}_0(x)\leq\frac1{\mu[x,y]} \int_{[x,y]}\tilde{h}_0\,\dd\mu\leq \tilde{h}_0(x)+\eps/2.$$
This contradicts the facts that $\frac1{\mu[x,y]} \int_{[x,y]}\tilde{h}_n\,\dd\mu\to_{n\to\infty}\frac1{\mu[x,y]} \int_{[x,y]}\tilde{h}_0\,\dd\mu$, obtained with $g=\one_{[x,y]}$ in \eqref{eq:lem:pointwise}, and that $\tilde h_n$ is increasing.
\end{proof}

\begin{pro}\label{pro:continuity}
Let $P_n$ tend to $P_0$ in $\ma(\mu_1,\ldots,\mu_d,\eta)$  and $P_n'$ tend to $P_0'$ in $\ma(\eta,\nu)$. Assume moreover that $P'_n$ has increasing kernel for every $n$. Then $P_n'\circ P_n$ tends to $P_0'\circ P_0$.
\end{pro}

\begin{proof}
By Propositions \ref{pro:product}(b) and \ref{pro:rho}, we must show that the integral $\int f(x)g(y)h(z)\dd (P_n\circ P'_n)(x,y,z)$ tends to $\int f(x)g(y)h(z)\dd (P_0\circ P_0')(x,y,z)$, where $f((x_i)_{i=1}^d)=\prod_if_i(x_i)$ and $g$, $h$ and the $f_i$ are any increasing functions from $\R$ to $[0,1]$. For all $n\geqslant0$:
\begin{align*}
\int f(x)g(y)h(z)\dd (P_n\circ P'_n)(x,y,z)&=\int f(x)g(y)\int h(z) k_{P'_n}(y,\dd z)\dd P_n(x,y)\\
&=\int f(x)(g\tilde h_n)(y)\dd P_n,
\end{align*}
where $\tilde h_n=P_{n}.h$, by Definition \ref{defi:catenation} and \S\ref{subsubsec:action_des_transports}.

By Lemma \ref{lem:pointwise}, $(g\tilde{h}_n)_n$ pointwise converges to $g\tilde{h}_0$, almost surely. Hence by the dominated convergence theorem:
$$\left|\int f(x)(g\tilde h_n)(y)\dd P_0-\int f(x)(g\tilde h_0)(y)\dd P_0\right|\to0.$$
Moreover $\tilde{\rho}(P_n,P_0)$ tends to zero and for every $n$, $g\tilde h_n$ is an increasing functions taking values in  $[0,1]$. Hence, again by Propositions \ref{pro:product}(b) and \ref{pro:rho}:
$$\left|\int f(x)(g \tilde h_n)(y)\dd P_n-\int f(x)(g \tilde h_n)(y)\dd P_0\right|\to0.$$
Therefore with the triangle inequality:
$$\left|\int f(x)(g \tilde h_n)(y)\dd P_n-\int f(x)(g \tilde h_0)(y)\dd P_0\right|\to0,$$
which is exactly what we claimed.
\end{proof}

\begin{rem}[Counterexamples to Proposition \ref{pro:continuity}] {\bf(a)} Case where the mar\-gin\-als are not fixed. Take $d=1$, $\mu=\nu=\frac{1}{2}(\delta_{-1}+\delta_1)$ and $\eta=\frac{1}{2}(\delta_{-1/n}+\delta_{1/n})$ and its limit $\delta_0$ in the parameter $n$. We consider the quantile couplings $P_n=\frac{1}{2}(\delta_{-1,-1/n}+\delta_{1,1/n})\in\ma(\mu_0,\mu_1)$ and $P^{\prime}_n=\trans\!P_n\in\ma(\mu,\eta)$. They have increasing kernel but not fixed marginals. On the one hand $P_n\circ P^{\prime}_n=\frac{1}{2}(\delta_{-1,-1/n,-1}+\delta_{1,1/n,1})\underset{n\to\infty}{\longrightarrow}\frac{1}{2}(\delta_{-1,0,-1}+\delta_{1,0,1})$. On the other hand $(\lim_n(P_n)_{n})\circ(\lim_n(P^{\prime}_n)_{n})=\frac{1}{4}(\delta_{1,0,1}+\delta_{1,0,-1}+\delta_{-1,0,1}+\delta_{-1,0,-1})$. 

{\bf(b)} Case where the transitions are not increasing. See the example in \cite{Ke72} after the proof of Satz 14. Here the marginals are fixed but the transport plans neither have Lipschitz nor increasing kernel.
\end{rem}

We can now prove Lemma \ref{lem:closed}. It follows directly from Proposition \ref{pro:continuity}.

\begin{proof}[Proof of Lemma \ref{lem:closed}]\label{proof:lem:closed}
Consider a sequence $P^n_{1,2}\circ\cdots\circ P^n_{d,d+1}$ with limit $P\in\ma(\mu_1,\ldots,\mu_{d+1})$ and denote $\pr^{i,i+1}_{\#}P$ by $P_{i,i+i}$. Since $(\pr^{i,i+1})_\#$ is continuous, $P_{i,i+1}$ is the limit of $(P_{i,i+1}^n)_n$, hence $P_{i,i+1}\in\mathcal{I}_{i,i+1}$. Proposition \ref{pro:continuity} used by induction shows that $P=P_{1,2}\circ\cdots\circ P_{n,n+1}$.
\end{proof}

\section{Construction and characterization of the Markov-quantile process}\label{sec:construction}

{\bf Now \mb$(\mu_t)_{t\in \R}$ is a family  of probability measures on $\R$, $F_t$ denotes the cumulative distribution function $F_{\mu_t}$ of $\mu_t$ and $G_t$ its quantile function, see \S\ref{ss:minimality}.} In this section we build the Markov-quantile measure $\mq$ and prove Theorems A and B. Our proof of Theorems A--B is based on transport plans $L_{[s,t]}\in \ma(\la,\la)$, defined in \S\ref{subsec:little_t}, that are the 2-marginals of an important auxiliary process law called the quantile level measure $\lev\in\ma(\la_t)_{t\in\R}$, where each $\lambda_t$ is a copy indexed by $t$ of $\la=\la_{[0,1]}$. Here is the link between $\lev$ and $\mq$: $G_t$ maps $([0,1],\la_t)$ to $(\R\cup\{\pm \infty\},\mu_t)$; set $G=\otimes_{t\in\R}G_t$, so that $G_\#$ maps $\ma((\lambda_t)_{t\in\R})$ to $\ma((\mu_t)_{t\in\R})$. Then, as proven in the proof Theorem \ref{them:deuxieme_sec} p.\ \pageref{lieu_preuve_glev}:
\begin{equation}\label{eq:Glev}
G_\#\lev=\mq.
\end{equation}

\begin{rem}
Readers familiar with Mathematical Statistics may compare $L_{[s,t]}$ with the notion of \emph{copula} and \eqref{eq:Glev} with Sklar's theorem. In particular, the (bivariate) cumulative distribution function of $\lev^{s,t}=L_{[s,t]}$ is a copula associated with $\mq^{s,t}$. It is defined as the pointwise maximum of the family of copulas made of the cumulative distribution functions associated with $(L_{[s,t]\cap R})_{R\subset \R}$ where $R$ ranges over finite sets, see Definitions \ref{defi:geqlc} and \ref{def:lcsup}, Lemma \ref{lem:sto_order}, Proposition \ref{pro:rho} and Notation \ref{lem:existenceT}. Although the link is immediate we will not use further the terminology of Copula Theory.
\end{rem}

This section is divided in three. In \S\ref{subsec:little_t} we define the coupling $L_R$ for all $R\subset \R$, using the key monotonicity Lemma \ref{lem:mono2}. In \S\ref{subsec:preuve_premier} we define $\mq$ and prove Theorem \ref{them:a},  purely via  the 2-marginals $(\mq^{s,t})_{s<t}\eqdef((G_s\otimes G_t)_\#L_{\op]s,t\clo[})_{s<t}=((G_s\otimes G_t)_\#L_{[s,t]})_{s<t}$, i.e.\  without introducing $\lev$. Finally in \S\ref{subsec:preuve_deuxieme} we prove Theorem \ref{them:b}, i.e.\ a refinement of \eqref{eq:Glev} and the approximation of $\lev$ and $\mq$ by sequences $(\lev_{R_n})_n$ and $(\como_{[R_n]})_n$, respectively.

\subsection{Transitions kernels to and from the space of quantiles \mb$[0,1]$\mn}\label{subsec:little_t}~
{\small In this paper} we will need to consider the quantile couplings of the measures $\mu_t$ with the reference measure $\la=\Lg$.

\begin{notation}[$q_r$, $k_r$, $\trans k_r$] \label{notation:kernels} For all $r\in R$ we set $q_r= \como(\la,\mu_r)$; thus $\trans q_r$ (see Definition \ref{defi:transpose}) is $\como(\mu_r, \lambda)$. Those couplings admit the respective disintegration kernels:
$$k_r:(\alpha,B)\mapsto \delta_{G_{r}(\alpha)}(B),$$
\begin{align*}
\trans{k}_r(x,\espcdot)=
\begin{cases}
\delta_{F_r(x)}&\text{if }\mu_r(x)=0\\
(\mu_r(x))^{-1}\la\lfloor_{]F_r(x^-),F_r(x)[}&\text{if }\mu_r(x)>0.\\
\end{cases}
\end{align*}
\end{notation}

\begin{rem}\maz\label{rem:finesse}\point\label{p1:rem:finesse} Notice that $\trans{q}_s . q_t=\join(\mu_s,\trans{k}_s . k_t)=\como^{s,t}$, where $\como^{s,t}=\como(\mu_s,\mu_t)$, see Definition \ref{defi:quantile_process}. Indeed, $\trans{q}_s . q_t=\trans{q}_s .\id_{\la,2}. q_t$, then apply \ref{p2:rem:finesse} below,
which will also be useful farther.

\point\label{p2:rem:finesse} If some ordered pair $(U,V)$ of variables has law $T\in \ma(\la,\la)$, then $\ma(\mu_s,\mu_t)\ni\trans{q}_s .T. q_t=\law(G_s(U),G_t(V))=(G_s\otimes G_t)_\#T$. Indeed, $\trans{q}_s .T. q_t=\pr^{1,4}_\#(\trans{q}_s \circ T\circ q_t)$ and the 4-times process $(G_{\mu_s}(U),U,V,G_{\mu_t}(V))$ has law $\trans{q}_s \circ T\circ q_t$ and is Markov, since the $\sigma$-fields spanned by $U$ and $\{U,G_{\mu_s}(U)\}$ are the same.
\point\label{p3:rem:finesse} From \ref{p1:rem:finesse} we get $\trans q_r.q_r=\id_{\mu_r,2}$, so that $\trans q_r.q_r.\trans q_r=\trans q_r$ and $q_r.\trans q_r.q_r=q_r$. {\em However}, $q_r. \trans q_r\neq \id_{\la,2}$. Indeed, $k_r.\trans{k}_r$  maps any quantile level $\alpha\in \op]0,1\clo[$ on itself {\em except when} $G_r(\alpha)$ is an atom of $\mu_r$. Actually, $\mu_r$-almost surely:
\begin{align*}
k_r. \trans{k}_r(\alpha)=
\begin{cases}
\delta_\alpha&\text{if }\mu_r(G_{r}(\alpha))=0\\
(\alpha^+-\alpha^-)^{-1}\la|_{]\alpha^-,\alpha^+[}&\text{if }\alpha\in]\alpha^-,\alpha^+[,
\end{cases}
\end{align*}
where $]\alpha^-,\alpha^+[$ denotes any set $A_{r,x}$ as follows.
\end{rem}

\begin{notation}[Atomic levels, $\ell_r$]\label{nota:Asx}\maz\point\label{p:Asx} We denote by $A_{r,x}$ the interval $\op]F_r(x^-),\linebreak[1]F_r(x)\clo[$ of quantile levels merged by $G_r$ on some atom $x$ of $\mu_r$, and by $A_r$  the set $\bigcup_{\mu_r(x)>0}A_{r,x}\subset [0,1]$ of ``atomic levels'' of $\mu_r$.

\point\label{p:l_R_fini} We denote $k_r. \trans{k}_r$ described in Remark \ref{rem:finesse}\ref{p3:rem:finesse} by $\ell_r$.
\end{notation}

\begin{rem}\label{rem:action_l}Measures $\theta\ll \la$ are transported by $\ell_r$ as follows: $\theta.\ell_r$ coincides with $\theta$ on $\op]0,1\clo[\setminus A_r$, and on each $A_{r,x}$ it has constant density and mass $\theta(A_{r,x})$, i.e.\ equals $(\alpha^+-\alpha^-)^{-1}\theta(A_{r,x})$. Equivalently, $F[\theta. \ell_r]$ is continuous, equal to $F[\theta]$ on $\op]0,1\clo[\setminus A_{r,x}$, and affine on each connected component $A_{r,x}$ of $A_r$.
\end{rem}

The product $\como^{s,r_1}.\como^{r_1,r_2}.\cdots.\como^{r_{m-1},r_m}.\como^{r_m,t}$ appearing in Theorem \ref{them:a} is more deeply analyzed in Theorem \ref{them:b}. Its kernel reads: 
\begin{align}
&(\trans k_s.k_{r_1}).(\trans k_{r_1}.k_{r_2}).\cdots.(\trans k_{r_{m-1}}.k_{r_m}).(\trans k_{r_m}.k_t)\notag\\
=\ & \trans k_s.(k_{r_1}.\trans k_{r_1}).\cdots.(k_{r_m}.\trans k_{r_m}).k_t\notag\\
=\ &\trans k_s. \ell_{r_1}.\cdots.\ell_{r_m}.k_t\label{eq:composee_l}
\end{align}
In Remark \ref{rem:comp_mini}, \eqref{eq:composee_l} is further commented and reexpressed for transports in place of kernels. For now, it leads to introduce the following kernel.

\begin{notation}\label{notation:Ts}
Take $R=\{r_1,\ldots,r_m\}\subset\R$. We denote the kernel $\ell_{r_1}.\ell_{r_2}.\linebreak[1]\cdots.\linebreak[1]\ell_{r_m}$ from $\op]0,1\clo[$ to itself by $\ell_R$, and $\join(\la;\ell_R)\in\ma(\la,\la)$ by $L_R$. If $R=\varnothing$, $\ell_\varnothing$ is the identity kernel and $L_\varnothing$ the identity transport.
\end{notation}

Notice that $\ell_{\{r\}}=\ell_r$ and that for any $R$, $\la.\ell_R=\la$. Moreover $\ell_R$ only depends on $(A_r)_{r\in R}$. The following lemma is particularly simple.

\begin{lem}\label{lem:mono1}\maz
Let $R\subset \R$ be a finite set and $\mu,\,\nu$ be in $\m(\la)$. 

\begin{enumerate}[label=\bgroup\bf(\alph*)\egroup]
\item[\point\label{item:mono1a}] If $\mu\leqs \nu$ then $\mu .\ell_R\leqs \nu. \ell_R$, i.e. $\join(\la; \ell_R)$ has increasing kernel.
\item[\point\label{item:mono1b}] If $\mu$ is in $\mdec(\la)$, then so is $\mu .\ell_R$, and $\mu\leqs \mu. \ell_R$.
\end{enumerate}
\end{lem}
\begin{proof} For $r\in \R$, one easily checks that $\ell_r$ is increasing and stabilizes $\mdec$; (a) and the first point of (b) follow. If $\mu\in \mdec(\la)$, to see that $\mu\leqs \mu. \ell_R$, look at the cumulative distribution functions. They coincide off the components of $A_r$. Now on each of those, $F_{\mu. \ell_R}$ is affine whereas $F_\mu$ is concave since $\mu$ has decreasing density, so necessarily  $F_\mu\geqslant F_{\mu. \ell_R}$.
\end{proof}

\begin{figure}
\begin{center}
\def\svgwidth{\columnwidth}
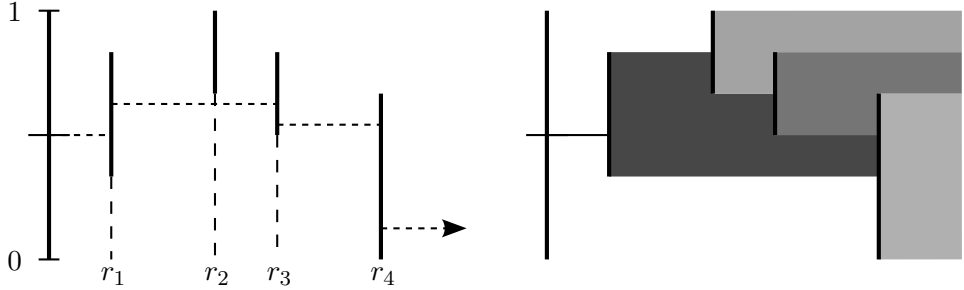
\caption{Composition of kernels $\ell_r$}\label{figure} 
\end{center}
\end{figure}

We add a remark, linked with Figures \ref{figure}--\ref{figure3}, on the principle of Theorem \ref{them:a}'s proof. A reader only looking for the formal proof itself may skip it.

\begin{rem}\label{rem:qualitative}We will not (directly) obtain the couplings $\mq^{s,t}$ as a limit of products $\como^{s,t}_{[R_n]}$, for some finite sets $R_n$ with dense union, as suggested in \S\ref{subsec:intro_quantilemarkovien} p.\ \pageref{page:tentative_limite} ---to show this does not work. We aim at obtaining $\mq^{s,t}$ as a {\em supremum}, of the set $\{\como^{s,t}_{[R]} : R\ \text{finite and } R\subset]s,t[\}$ (see Theorem \ref{them:a}\ref{item:limite_de_compose} for the notation), and actually we do it on the space of quantile levels, i.e.\ we look for a supremum of $\{\ell_R : R\ \text{finite and } R\subset]s,t[\}$. The question is to find the adequate quantity, or order relation, for which a supremum (and hopefully then a maximum) shall be sought.

First, Figure \ref{figure} makes us observe how kernels of the type $\ell_R$ act on measures of $\p(\op]0,1\clo[)$. It displays the action of $\ell_{R}$ with $R=\{r_1,r_2,r_3,r_4\}$ on some Dirac measure $\delta$. The vertical segment on the left is the space $\op]0,1\clo[$ of quantile levels. We suppose that each $\mu_{r_i}$ has a single atom and draw vertically, at abscissa $r_i$, the interval $A_{r_i}$ (see Notation \ref{nota:Asx}\ref{p:Asx}). The drawing is in the case where $\delta=\delta_{x}$ with $x\in A_{r_1}$. Then, see Remark \ref{rem:action_l}: $\ell_{\{r_1\}}$ maps $\delta$ on the uniform probability measure on $A_{r_1}$; in turn, $\ell_{\{r_2\}}$ leaves the latter unchanged outside of $A_{r_2}$ and makes it uniform on $A_{r_2}$, {\em etc.} The first drawing shows a ``possible trajectory of an element of mass at $x$'' transported by the discrete Markov chain with transition kernels $(\ell_{r_i})_{i=1}^4$. Since we take $x \in A_{r_1}$, it is displaced by $\ell_{r_1}$ to $x '$, picked uniformly at random in $A_{r_1}$. In case $x '\not\in A_{r_2}$, as in the figure, it is unchanged by $\ell_{r_2}$; then in case $x '\in A_{r_3}$ (figure), it is displaced by $\ell_{r_3}$ to a random $x ''\in A_{r_3}$, and finally, in case $x ''\in A_{r_4}$, displaced by $\ell_{r_4}$ to a random $x '''\in A_{r_4}$. The second drawing shows the successive measures $\delta_x $, $\delta_x .\ell_{r_1}$, $\delta_x .\ell_{r_1}.\ell_{r_2}$ {\em etc.}, the level of grey being proportional to the value of their density.

So each $\ell_{r_i}$ ``spreads'' a little more the mass of $\delta.\ell_{r_1}.\cdots.\ell_{r_{i-1}}$, replacing it by its mean (measure of constant density) on each connected component of $A_{r_i}$. If $\theta\ll\la$, this averaging process lowers the total variation of the density at each step: at most, you get the measure with constant density one, i.e.\ $\la$ itself, on which all the transports $\ell_{R}$ act trivially. Thus a natural idea is to consider that, if ${R'}\subsetneq R$ are finite and $\theta\ll\la$, the density of $\theta.\ell_{R}$ will be closer to $\one_{\op]0,1\clo[}$, for some adequate distance, than that of $\theta.\ell_{R'}$.

Unfortunately this is the case if $R'=R\cap\op]-\infty,t]$ for some $t$ but not in general. Figure 2 shows, from top to bottom, the Dirac measure $\delta_x$ and the graph of the densities of:\smallskip

-- $\delta_x .\ell_{r_1}$, $\delta_x .\ell_{r_1}.\ell_{r_2}$, $\delta_x .\ell_{r_1}.\ell_{r_2}.\ell_{r_3}$, and $\delta_x .\ell_{r_1}.\ell_{r_2}.\ell_{r_3}.\ell_{r_4}$ (on the left),\smallskip

-- $\delta_x .\ell_{r_1}$, $\delta_x .\ell_{r_1}.\ell_{r_2}$, and $\delta_x .\ell_{r_1}.\ell_{r_2}.\ell_{r_4}$ (on the right),\smallskip

\noindent in the case $x=\frac12$, $A_{r_1}=\op]\frac13,\frac56\clo[$, $A_{r_2}=\op]\frac23,1\clo[$, $A_{r_3}=\op]\frac12,\frac56\clo[$, $A_{r_4}=\op]0,\frac23\clo[$. Then $\delta.\ell_{\{r_1,r_2,r_4\}}=\la$, i.e.\ has exactly density $\one_{\op]0,1\clo[}$, whereas $\delta.\ell_{R}\neq\la$.

A remedy is to look the kernels $\ell_R$ act {\em on measures of density $\frac1x\one_{\op]0,x\clo[}$}. For the latter, and more generally any element of 
$\pdec(\op]0,1\clo[)$, which is stable by the action of the couplings $\ell_R$, the idea above works, with the stochastic order. This is Lemmas \ref{lem:mono2} and \ref{lem:existenceT} below; see also Remark \ref{rem:difference}. Figure \ref{figure3} gives the example of the kernels $\ell_{r_i}$ of Figure \ref{figure2} acting on the measure $\nu$ of density $\frac1x\one_{\op]0,x\clo[}$ with $x=\frac12$: one gets $\nu.\ell_{r_1}.\ell_{r_2}.\ell_{r_3}.\ell_{r_4}\geqs\nu.\ell_{r_1}.\ell_{r_2}.\ell_{r_4}$.
\end{rem}

\begin{figure}
\begin{center}
\def\svgwidth{\columnwidth}
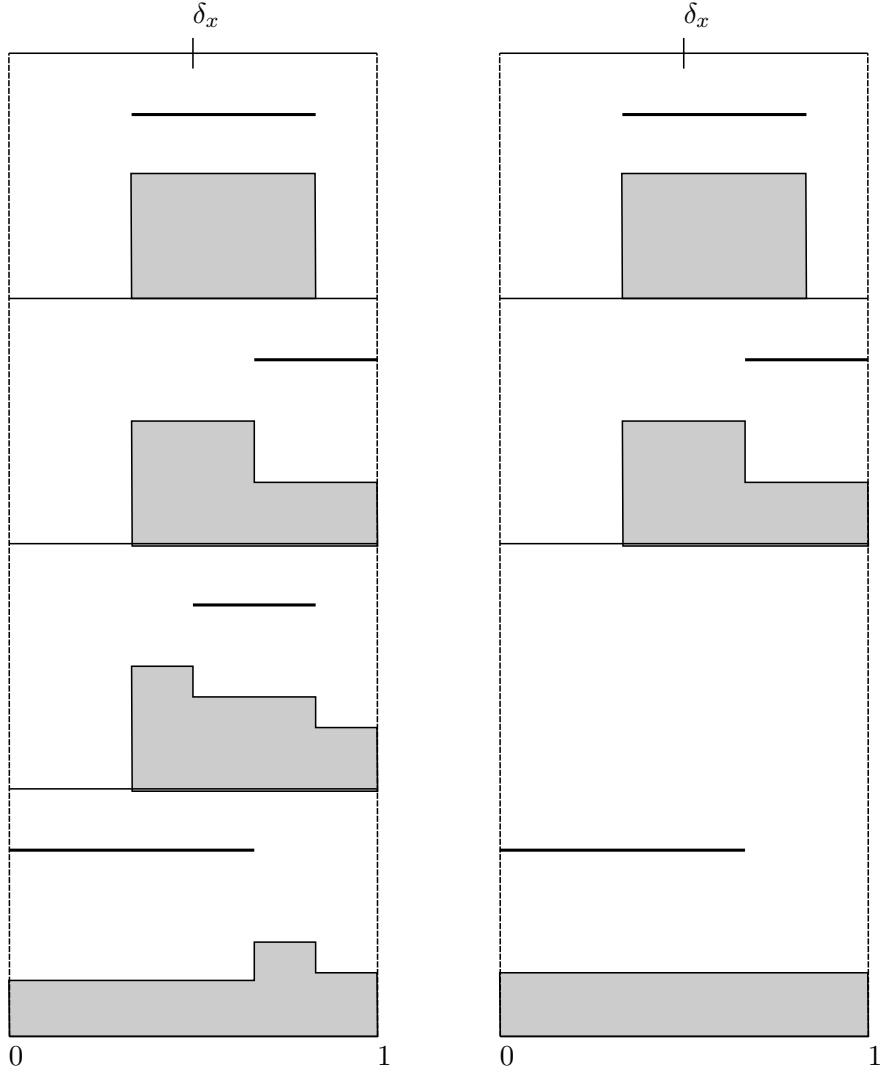
\caption{Composition of kernels $\ell_r$, acting on some Dirac Measure $\delta_x$. On each column, the horizontal interval is the space $\op]0,1\clo[$ of quantile levels. From top to bottom, the horizontal bars represent successive atomic intervals $A_{r_i}$ and, below each one, the density of $\delta_x$ transported by the composition of the successive corresponding kernels $\ell_{r_1}$, \ldots, $\ell_{r_i}$. On the right, $A_{r_3}$ is omitted.}\label{figure2} 
\end{center}
\end{figure}

\begin{figure}
\begin{center}
\def\svgwidth{\columnwidth}
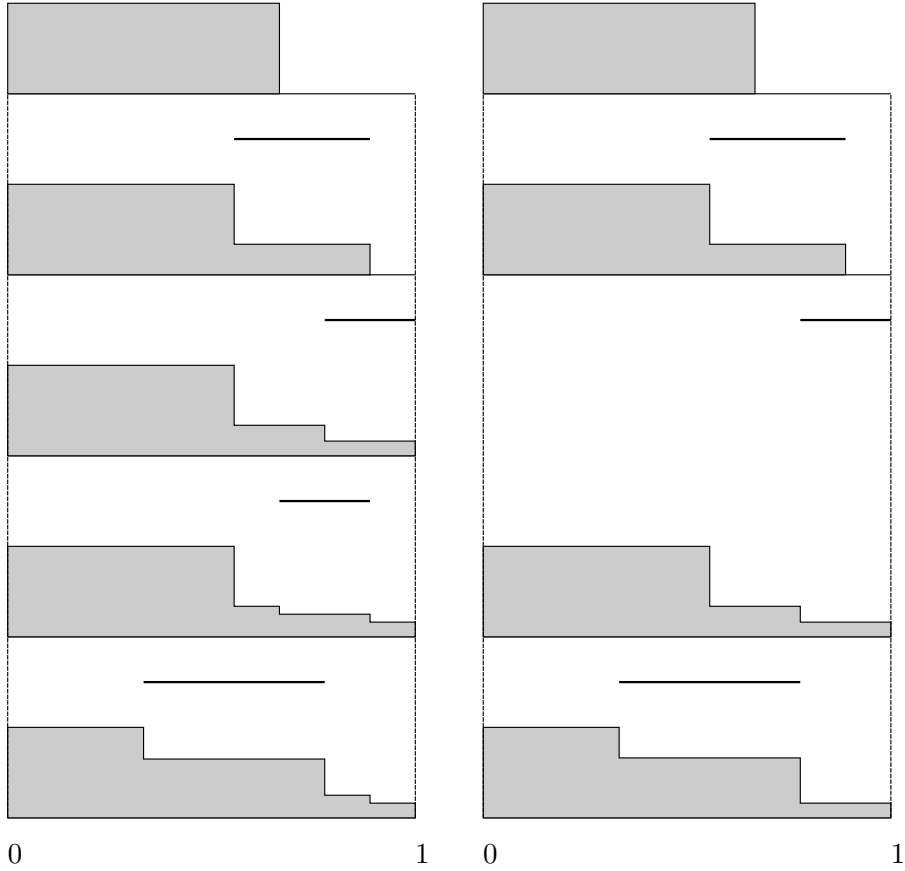
\caption{The kernels $\ell_r$ of Figure \ref{figure2}, acting on the measure of density $\frac1x\one_{\op]0,x\clo[}$. The global setup of this figure is the same as that of Figure \ref{figure2}.}\label{figure3} 
\end{center}
\end{figure}

Though simple, the next lemma is a key of our construction of $\mq$.
\begin{lem}\label{lem:mono2}
Let $R\subset R'$ be two finite subsets of\/ $\R$ and $\mu\in\mdec(\la)$. Then $\mu .\ell_R\leqs \mu .\ell_{R'}$.
\end{lem}
\begin{proof}
Using an induction on the cardinal difference, it is enough to prove this if $R'$ has one more element than $R$, say $r'$. We order it with the elements $r_i$ of $R$: $r_1<\cdots<r_k<r'<r_{k+1}<\cdots<r_m$. By Lemma \ref{lem:mono1}(b), if $\mu$ is in $\mdec(\la)$, so is $\mu_k\eqdef\mu. \ell_{r_1}.\cdots. \ell_{r_k}$ and $\mu_k\leqs \mu_k .\ell_{r'}$. We apply $\ell_{r_{k+1}}.\cdots. \ell_{r_m}$ to each term of this inequality. Lemma \ref{lem:mono1}(a) concludes.
\end{proof}

\noindent The choice of the suitable set $\mdec$ in our monotonicity Lemma \ref{lem:mono2} (see also Remark \ref{rem:qualitative}) enables us to extend the definition of $L_R$ to infinite sets $R$. Recall that $\leqlc$ is interpreted in terms of $\leqs$ in Remark \ref{rem:q_is_mini}.

\begin{lemnotation}\label{lem:existenceT}
\maz
\point\label{p1:lem:existenceT} For all finite subsets $R$, $R^{\prime}$ of\/ $\R$, $L_R\leqlc \la\otimes\la$, and $R'\subset R$ $\Rightarrow$ $L_{R'}\leqlc L_{R}$.\smallskip

\point\label{p2:def_L_R} For any $R\subset\R$, $\lcsup\{L_{R'}:\,R'\subset R$ and $R'$ finite$\}$ exists. We denote it by $L_R$ (which is consistent with Notation \ref{notation:Ts} when $R$ is finite). \smallskip

\point\label{p3:def_L_R} For any $R\subset\R$, there is a nested sequence $(R_n)_n$ of finite subsets of $R$ such that $(L_{R_n})_n$ converges weakly to $L_R$; if $(R_n)_n$ has this property, so has all sequence $(R'_n)_n$ such that $R'_n\supset R_n$. 
\end{lemnotation}

\begin{proof} Let us prove (a). For all $(x,y)\in[0,1]^2$, $F[L_R](x,y)=(\la\lfloor_{[0,x]}).L_R([0,y])$. After Lemma \ref{lem:mono1}\ref{item:mono1b}, $(\la\lfloor_{[0,x]}).L_R\in\mdec$ so $f:y\mapsto\frac1y((\la\lfloor_{[0,x]}).L_R)([0,y])$ is decreasing, thus $f(y)\geq f(1)=x$. Hence: $\forall y\in[0,1], ((\la\lfloor_{[0,x]}).L_R)([0,y])\geqslant xy=F[\la\otimes\la](x,y)$, i.e.\  $L_R\leqlc \la\otimes\la$. Now if $R'\subset R$ for every $x$ we can apply Lemma \ref{lem:mono2} to $\la\lfloor_{[0,x]}\in\mdec$. This gives $F[L_{R}](x,y)=(\la\lfloor_{[0,x]}).L_{R}([0,y])\geqslant(\la\lfloor_{[0,x]}).L_{R'}([0,y])=F[L_{R'}](x,y)$, which is the expected relation for $\leqlc$. Then, (b) and (c) follow from criterion \ref{item:treillis} of Lemma \ref{lem:sto_order}. Indeed, by (a), ${\cal S}=\{L_{R'}:\,R'\subset R$ and $R'$ finite$\}$ is bounded from above by $\la\otimes\la$, and if $\{L_{R'_1},L_{R'_2}\}\subset{\cal S}$, then $L_{R'_1\cup R'_2}\in{\cal S}$ and $L_{R'_1\cup R'_2}\geqlc L_{R'_i}$ for $i=1,2$. Finally, the assertion about $(R'_n)_n$ follows from (a) and the interpretation of both $\rho$ and $\leqs$ with cumulative distribution functions.\end{proof}

\begin{notation}\label{nota:kernel_tS}For all $R\subset \R$, we denote by $\ell_R$ the kernel associated with $L_R$. This is again consistent with Notation \ref{nota:Asx}\ref{p:l_R_fini} when $R$ is finite.
\end{notation}

\begin{rem}For every $\alpha\in \op]0,1\clo[$, due to the definition of $\leqlc$,  $\la\lfloor_{[0,\alpha]}.\ell_R=\stosup_{T\subset R,\ T\text{ finite}} \la\lfloor_{[0,\alpha]}.\ell_T$. Indeed, Lemma \ref{lem:existenceT}(c) actually proves that there exists a nested sequence $(R_n)_n$ of finite subsets of $\R$ such that $F[L_{R_n}]$ pointwise converges to $F[L_R]$.  Therefore $F[\la\lfloor_{[0,\alpha]}.\ell_{R_n}](\cdot)=F[R_n](\alpha,\cdot)$ pointwise converges to $F[\la\lfloor_{[0,\alpha]}.\ell_{R}]$. Moreover $F[\la\lfloor_{[0,\alpha]}.\ell_{T}]\geq F[\la\lfloor_{[0,\alpha]}.\ell_{R}]$ for every finite $T\subset R$. These two facts give the remark.
\end{rem} 

\begin{rem}[complement to Remark \ref{rem:qualitative}]\label{rem:difference} Lemma \ref{lem:existenceT} means that, for $R\subset \R$ and all interval $\op]a,b\clo[\subset\op]0,1\clo[$, $\bigl(\la\lfloor_{\op]a,b\clo[}\bigr).L_{R}$ is not obtained as a supremum, but as the difference of two: $\bigl(\la\lfloor_{[0,\alpha]}\bigr).L_{R}=\stosup\bigl\{\bigl(\la\lfloor_{[0,\alpha]}\bigr).L_{R'}:R'\subset R\ \text{and $R'$ finite}\bigr\} - \stosup\bigl\{\bigl(\la\lfloor_{[0,\alpha]}\bigr).L_{R'}:R'\subset R\ \text{and $R'$ finite}\bigr\}$.
\end{rem}

The following result is crucial to define processes on $(\op]0,1\clo[,\la)$ with Corollary \ref{cor:consist},     as is done in particular  in Definition \ref{defi:lev} for $\lev_R$ and $\lev$.
\begin{pro}\label{pro:compo}If $R$ and $R'$ are subsets of $\R$ such that $r\leqslant r'$ for all $(r,r')\in R\times R'$, then
: $L_{R\cup R'}=L_{R}.L_{R'}$.
In particular, for $s<t<u$, $L_{[s,u]}=L_{[s,t]}.L_{[t,u]}$.
\end{pro}
\begin{proof} Using Lemma \ref{lem:existenceT} (b) and (c) we find sequences $(R_n)_n$ and $(R'_n)_n$ of finite subsets of $R$ and $R'$ respectively, such that $L_{R_n}$ converges weakly to $L_R$, $L_{R'_n}$ to $L_{R'}$, and $L_{R_n\cup R'_n}$ to $L_{R\cup R'}$. Besides, since $r\leqslant r'$ for all $(r,r')\in R\times R'$, since the $R_n$ are finite and since $L_{\{r\}}$ is idempotent (so that in case $R\cap R'\neq\varnothing$ and $R_n\cap R'_n=R\cap R'=\{r\}$, a repetition of $L_{\{r\}}$ does not matter), $L_{R_n\cup R'_n}=L_{R_n}.L_{R'_n}$. Then, using Proposition \ref{pro:rho} and the distance $\rho$ introduced in it:
\begin{align*}
\rho(L_R.L_{R'},L_{R\cup R'})&\leqslant\rho(L_R.L_{R'},L_R.L_{R'_n})+\rho(L_R.L_{R'_n},L_{R_n}.L_{R'_n})\\
&\qquad\qquad\qquad\qquad\qquad\qquad\qquad+\rho(L_{R_n\cup R'_n},L_{R\cup R'})\\
&\leqslant\rho(L_{R'},L_{R'_n})+\rho(L_R,L_{R_n})+\rho(L_{R_n\cup R'_n},L_{R\cup R'})
\end{align*}
by Proposition \ref{pro:multiplicationcontractante} and Remark \ref{rem:multiplicationcontractante}, since $\trans{L_{R'_n}}$ also preserves $\mdec(\la)$. All terms tend to zero when $n$ tends to infinity. The wanted equality follows.
\end{proof}

\subsection{The Markov-quantile process $\mq$; proof of Theorem \ref{them:a}}\label{subsec:preuve_premier}

\begin{reminder} For $(s,t)\in\R^2$ with $t>s$, and $R=\{r_1,\ldots,r_m\}\subset[s,t]$, $\como^{s,t}_{[R]}$ denotes the coupling $\como(\mu_{s},\mu_{r_1}).\linebreak[1]\como(\mu_{r_2},\mu_{r_3}).\linebreak[1]\cdots.\linebreak[1]\como(\mu_{r_{m}},\mu_{t})\in\ma(\mu_s,\linebreak[1]\mu_t)$, see Theorem \ref{them:a}\ref{item:limite_de_compose} p.\ \pageref{them:a}.
\end{reminder}

\maz
\begin{prodefinition}\label{pro:comp_conv} {\bf\point\label{aaa}} The set $\{\como^{s,t}_{[R]}\,|\,R\subset \R\ \text{and}\ \sharp R<\infty\}$ has a lower orthant supremum and there is a nested sequence $(R_n)_{n\in \N}$ of finite sets such that $\como^{s,t}_{[R_n]}$ tends to it. We denote it by $\mu_{s,t}$.\smallskip

{\bf\point\label{bbb}} The family $(\mu_{s,t})_{s<t}$ is consistent in the sense of Definition \ref{defi:consist}, giving rise to a Markov measure $\mq\in\ma((\mu_t)_{t\in\R})$, that we call the \emph{Markov-quantile measure} attached to $(\mu_t)_{t\in\R}$.\smallskip

{\bf\point\label{ccc}} For all $s$ and $t>s$,  $\mq^{s,t}=(G_s\otimes G_t)_\#L_{]s,t[}=(G_s\otimes G_t)_\#L_{[s,t]}$.
\end{prodefinition}

To show Proposition \ref{pro:comp_conv} we first state the following crucial relation between $\como^{s,t}_{[R]}\in \ma(\mu_s,\mu_t)$ and $L_R\in \ma(\la,\la)$ definined in \S\ref{subsec:little_t}.

\begin{rem}\label{rem:comp_mini} Take any finite subset $R$ of $\R$ and $(s,t)$ with $s<t$, then:
$$\como^{s,t}_{[R]}=(G_s\otimes G_t)_\#L_{]s,t[\cap R}=(G_s\otimes G_t)_\#L_{[s,t]\cap R}.$$
Indeed, $\como^{s,t}=\trans q_s.q_t=\law(G_s,G_t)$, see Remark \ref{rem:finesse}\ref{p1:rem:finesse}. It also equals $(G_s\otimes G_t)_\#(\Id_2)$, where $\Id_2=\id_{\la,2}$ is the identity transport form $\la$ to itself, see Notation \ref{notation:transport_id}, since by Remark \ref{rem:finesse}\ref{p2:rem:finesse}, $\como(\mu_s,\mu_t)=\trans q_s . q_t=\trans q_s. \id_2 . q_t=(G_s\otimes G_t)_\#(\Id_2)$. More generally, for any $\{r_1,\ldots,r_m\}\subset\op]s,t\clo[$:
\begin{align*}
&\como(\mu_{s},\mu_{r_1}).\cdots.\como(\mu_{r_m},\mu_{t})\\
=&\trans q_{s}.q_{r_1}.\trans q_{r_1}.q_{r_2}.\cdots.\trans q_{r_{m-1}}.q_{r_{m}}.\trans q_{r_{m}}.q_{t}\\
=&(G_{s}\otimes G_{t})_\#(q_{r_1}.\trans q_{r_1}.q_{r_2}.\cdots.\trans q_{r_{m-1}}.q_{r_{m}}.\trans q_{r_{m}})\\
&\qquad\qquad\qquad\qquad\text{(notice that this writing involves neither $\trans q_{s}$ nor $q_{t}$)}\\
=&(G_{s}\otimes G_t)_\#L_{\{r_1,\ldots,r_m\}}.
\end{align*}

 Besides recall that $q_s.\trans q_s.q_s=q_s$ and $\trans q_t.q_t.\trans q_t=\trans q_t$, so that:
\begin{align*}
\como(\mu_{s},\mu_{r_1}).\cdots.\como(\mu_{r_m},\mu_{t})
&=(G_{s}\otimes G_{t})_\#(q_{s}.\trans q_{s}.q_{r_1}.\cdots.\trans q_{r_m}.q_{t}.\trans q_{t})\\
&=(G_{s}\otimes G_t)_\#L_{\{s,r_1,\ldots,r_m,t\}}.
\end{align*}
\end{rem}

\begin{proof}[Proof of Proposition \ref{pro:comp_conv}] We have only to gather our results. We set ${\cal S}\eqdef\{R\subset \op]s,t\clo[:\,R$ is finite$\}$. By Remark \ref{rem:leqlc_et_Gs} and Lemma \ref{lem:existenceT}\ref{p1:lem:existenceT}, for all $R\in{\cal S}$, $(G_s\otimes G_t)_\#L_R\leqlc (G_s\otimes G_t)_\#(\la\otimes \la)$, and $R'\subset R$ $\Rightarrow$ $(G_s\otimes G_t)_\#L_{R'}\leqlc (G_s\otimes G_t)_\#L_{R}$. Thus, by Lemma \ref{lem:sto_order}\ref{item:treillis}, $\lcsup_{R'\in{\cal S}}(G_s\otimes G_t)_\#L_{R'}$ exists and is the limit of some sequence $((G_s\otimes G_t)_\#L_{R_n})_n$. Hence the limit in \ref{p1:lem:existenceT}  is given by Remark \ref{rem:comp_mini}:
$$\como^{s,r_1}.\como^{r_1,r_2}.\cdots.\como^{r_{m(n)-1},r_{m(n)}}.\como^{r_{m(n)},t}=(G_s\otimes G_t)_\#L_{R_n}.$$

We prove now the first part of (c), i.e.\ $\mu_{s,t}=(G_s\otimes G_t)_\#L_{\op]s,t\clo[}$ and, at the same time, that the sequence $(R_n)_n$ in (a) can be chosen to be nested. Let $(R_n)_n$ be a nested sequence of $\mathcal{S}$ such that $L_{R_n}\rightarrow L_{]s,t[}$, i.e.\ $F[L_{R_n}]$ pointwise converges to $F[L_{]s,t[}]$. If $M\in \ma(\mu_s,\mu_t)$ satisfies $(G_s\otimes G_t)_\#L_{R'}\leqlc M$ for every $T\in {\cal S}$, this  also holds for $R'=R_n$ for all $n$. Now by Remark \ref{rem:diese_contractant}, $(G_s\otimes G_t)_\#L_{R_n}\to (G_s\otimes G_t)_\#L_{]s,t[}$. Therefore, going to the limits at the level of the cumulative distribution functions we get $(G_s\otimes G_t)_\#L_{]s,t[}\leqlc M$. Then Remark \ref{rem:leqlc_et_Gs} gives $(G_s\otimes G_t)_\#L_{R'}\leqlc(G_s\otimes G_t)_\#L_{]s,t[}$, hence:
$$(G_s\otimes G_t)_\#L_{]s,t[}=\lcsup_{R'\subset\R\ \text{and $R'$ finite}}\bigl\{\como^{s,t}_{[R']}=(G_s\otimes G_t)_\#L_{R'}\bigr\}.$$

Corollary \ref{cor:consist} gives (b). Indeed, Proposition \ref{pro:compo} on the composition of transports $L_R$ gives the consistency of  $(\mu_{s,t})_{s,t}$ (see Definition \ref{defi:consist}): $(G_s\otimes G_t)_\#L_{\op]s,t\clo[}=\trans{q}_s .L_{\op]s,t\clo[}.q_t$ and $(G_t\otimes G_u)_\#L_{]t,u[}=\trans{q}_t L_{]t,u[}q_u$. Since $q_t.\trans{q}_t=L_{\{t\}}$, $\mu_{s,u}=\mu_{s,t}.\mu_{t,u}$.

For the second equality of (c), proceed as at the end of Remark \ref{rem:comp_mini}.
\end{proof}

We now prove Theorem \ref{them:a}.

\begin{proof}[Proof of Theorem \ref{them:a}] {\bf (a)} Recall that $\mq$ is Markov and defined in Definition \ref{pro:comp_conv}. By construction, $\mq\in\ma((\mu_t)_{t\in\R})$ and satisfies \ref{item:limite_de_compose}. Then $\mq$ satisfies \ref{item:transitions}, i.e.\ has increasing kernel as quantile couplings have, see Remark \ref{rem:quantile_transitions_croissantes}, and since this property is stable by composition and weak limit, see Remarks \ref{rem:transitions_croissantes_composee} and Lemma \ref{lem:transitions_croissantes_limite}. The last claim of \ref{item:minimal} reads:
\begin{align*}
&\law(X_{t}|\,X_{s}\leqslant x)\\
=\ &\mathrm{stoinf}\{\law(Y_{t}|\,Y_{s}\leqslant x):\,\law(Y)\in\ma(\mu)\ \text{satisfies \ref{item:markov} and \ref{item:transitions}}\}.
\end{align*}
where $(X_t)_{t\in \R}$ has law $\mq$. An alternative writing is that for all $P=\law(Y)$ as above and all $s<t$, $F[\mq^{s,t}]\geqslant F[P^{s,t}]$, i.e.\ $\mq^{s,t}\leqlc P^{s,t}$. To show it, it is sufficient to show that for any strictly increasing $m$-tuple $(r_i)_{i=1}^m$:
\begin{equation}\label{eq:preuve_premier_theoreme_b}
\como^{r_1,r_2}.\cdots.\como^{r_{m-1},r_m}\leqlc P^{r_1,r_m}.
\end{equation}
Indeed $\mq^{s,t}=\lcsup \{\como^{r_1,r_2}.\cdots.\como^{r_{m-1},r_m}:\,s= r_1<\ldots<r_m= t\}$, by definition of $P$ in Proposition \ref{pro:comp_conv}.
We write (\ref{eq:preuve_premier_theoreme_b}) in the following equivalent form, in which we $\bar r_i$ stands for $r_{m+1-i}$: 
\begin{equation}\label{eq:preuve_premier_theoreme_c}
\underbrace{\trans\como^{\bar r_2,\bar r_1}.\cdots.\trans\como^{\bar r_{m},\bar r_{m-1}}}_{\text{denoted by $C$ below}}\leqlc \trans P^{\bar r_m,\bar r_1}.\tag{H$_m$}
\end{equation}
For $A,B\in\ma(\mu,\nu)$, from the definitions, $A\leqlc B$ is equivalent to $\forall x,\mu\lfloor_{\op]-\infty,x]}.A\leqs\mu\lfloor_{\op]-\infty,x]}.B$. Notice also that $\trans\como_{\bar r_i,\bar r_j}=\como_{\bar r_j,\bar r_i}$. We prove (\ref{eq:preuve_premier_theoreme_c}) by induction on $m$. (H$_2$) is true by definition of $\como$. Suppose (H$_m$). Take $x\in\R$. Then (see the justifications below):
\begin{align}
\mu_{\bar r_1}\lfloor_{\op]-\infty,x]}.\como^{\bar r_1,\bar r_2}.\cdots.\como^{\bar r_{m},\bar r_{m+1}}&=\mu_{\bar r_1}\lfloor_{\op]-\infty,x]}.C.\como^{\bar r_{k},\bar r_{m+1}}\nonumber\\
&\leqs\mu_{\bar r_1}\lfloor_{\op]-\infty,x]}.\trans P^{\bar r_m,\bar r_1}.\como^{\bar r_{m},\bar r_{m+1}}\label{eq:preuve_premier_theoreme_d}\\
&\leqs \mu_{\bar r_1}\lfloor_{\op]-\infty,x]}.\trans P^{\bar r_m,\bar r_1}.\trans P^{\bar r_{m+1},\bar r_m}\label{eq:preuve_premier_theoreme_e}\\
&\leqs \mu_{\bar r_1}\lfloor_{\op]-\infty,x]}.\trans P^{\bar r_{m+1},\bar r_1}\label{eq:preuve_premier_theoreme_f}.
\end{align}
Above, (\ref{eq:preuve_premier_theoreme_d}) holds since $\mu_{\bar r_1}\lfloor_{\op]-\infty,x]}.C\leqs\mu_{\bar r_1}\lfloor_{\op]-\infty,x]}.\trans P^{\bar r_m,\bar r_1}$ by (H$_m$) and since $\como_{\bar r_{m},\bar r_{m+1}}$ has increasing kernel, i.e.\ respects $\leqs$. For (\ref{eq:preuve_premier_theoreme_e}), since $P^{\bar r_1,\bar r_m}$ has increasing kernel, $\trans P^{\bar r_1,\bar r_m}$ maps $\mdec(\mu_{\bar r_m})$ on $\mdec(\mu_{\bar r_1})$ by Remark \ref{rem:stable}, so $\mu_{\bar r_1}\lfloor_{\op]-\infty,x]}.\trans P^{\bar r_m,\bar r_1}\in\mdec(\mu_{\bar r_m})$, hence it is an increasing limit of positive combinations of measures of the type $\mu_{\bar r_m}\lfloor_{\op]-\infty,y]}$. Then, for those measures, $\mu_{\bar r_m}\lfloor_{\op]-\infty,y]}.\como^{\bar r_m,\bar r_{m+1}}\leqs\mu_{\bar r_m}\lfloor_{\op]-\infty,y]}.\trans P^{\bar r_{m+1},\bar r_m}$ by definition of $\como$. Finally $(Y_t)_t$ is Markov, which gives (\ref{eq:preuve_premier_theoreme_f}), i.e.\  (H$_{m+1}$). We are done. \medskip

{\bf (b)} By Remark \ref{rem:quantile_up}, $(\mu_t)_{t\in\R}$ is increasing for $\leqs$ if and only if every $\como^{s,t}$ is an increasing coupling, i.e.\ concentrated on $\{(x,y):\,x\leqslant y\}$. This implies the same for their products and the limits of those, so for $\mq$.  Then apply Lemma \ref{coro}.
\end{proof}

\subsection{Proof of Theorem \ref{them:b}: convergence of \mb$\como_{[R_n]}$ to $\mq$\mn}\label{subsec:preuve_deuxieme} For all finite subset $R$ of $\R$, $P\in\ma((\mu_t)_{t\in\R})$ and $(s,t)$ with $s<t$ we introduced the couplings $P_{[R]}^{s,t}\in\ma(\mu_s,\mu_t)$ in Theorem \ref{them:sous_n} ---and actually in Theorem \ref{them:a}\ref{item:limite_de_compose} p.\ \pageref{them:a} in the case $P=\como$. We used them in \S\ref{subsec:preuve_premier}. Now we introduce the measure $P_{[R]}\in\ma((\mu_t)_{t})$ that was announced in \S\ref{subsec:intro_quantilemarkovien} in Notation \ref{pro:markovinise_fini1}. The notation is consistent, i.e.\ for all $s<t$, $\pr^{s,t}_\#P_{[R]}$ is the previously defined $P_{[R]}^{s,t}$. Then we prove Theorem \ref{them:b}, which means than we implement the tentative program introduced p.\ \pageref{page:tentative_limite} sq.\ in \S\ref{subsec:intro_quantilemarkovien}, in a way that avoids the problems explained there.

\begin{defi}\label{defi:quantile_discretement_markovien} If $M\in\ma((\mu_t)_t)$ and if $R=\{r_1,\ldots,r_m\}\subset\R$ we denote by $M_{[R]}\in\ma((\mu_t)_{\in\R})$ the measure \emph{$M$ made Markov at the points of $R$} defined by the data of its finite marginals $(\pr^{S})_\# M_{[R]}$, for all finite $S$ containing $R$, as follows.
$$(\pr^{S})_\# M_{[R]}= M^{s^0_1,\ldots, s^0_{n_0},r_1}\circ M^{r_1,s^1_{1},\ldots,s^1_{n_1},r_2}\circ\ldots\circ M^{r_m,s^m_{1},\ldots,s^m_{n_m}}.$$
Where $S=\{s^0_1,\ldots, s^0_{n_0},r_1,s^1_{1},\ldots,s^1_{n_1},r_2,\ldots,r_m,s^m_{1},\ldots,s^m_{n_m}\}$ and where the first or last term disappears if $n_0$ or $n_m$ is null, respectively. These marginals are consistent in the sense of Definition \ref{defi:consist}. So by Proposition \ref{pro:inductive_limit} this defines  $M_{[R]}$. We also commit an abuse of language: $M_{[R]}$ is rather the ``law of a process $X$ of law $M$, made Markov at the points of $R$''.
\end{defi}

By Proposition \ref{pro:compo}, for any $R\subset\R$, $(L_{R\cap[s,t]})_{s<t}\in \ma(\la,\la)$ is also a consistent family, thus again Proposition \ref{pro:inductive_limit} enables us to define the following processes on the set of quantile levels.

\begin{defi}\label{defi:lev}
For every $R\subset\R$ we denote by $\lev_R\in\ma((\la_t)_{t\in\R})$ ($\la_t$ denotes $\la$ at each $t$) the Markov process with 2-marginals $\lev_R^{s,t}=L_{R\cap[s,t]}$. We call $\lev_\R$ the \emph{level process} attached to $(\mu_t)_{t\in \R}$ and denote it by $\lev$.
\end{defi}

\begin{rem}\label{rem:mini_et_Q}In this subsection, using Definition \ref{defi:quantile_discretement_markovien} with $M=\como$ we obtain measures $\como_{[R]}$ linked with $\lev_R$ as follows. After Remark \ref{rem:comp_mini}, $\como_{[R]}^{s,t}=(G_s\otimes G_t)_\#L_{[s,t]\cap R}$, thus, with $G=(\otimes_{t\in \R} G_{t})$, the measures $G_\#\lev_R$ and $\como_{[R]}$ have the same 2-marginals. They are actually equal, as we prove in the ``claim'' at the beginning of the proof of Theorem \ref{them:deuxieme_sec}\ref{item:them:deuxieme_b} p.\ \pageref{proof:them:deuxieme}.
\end{rem}

The goal of the remaining part of this section is to prove the following statement that  is a more precise and technical version of Theorem \ref{them:b}. After some preparation its part \ref{item:them:deuxieme_a} is proved p.\ \pageref{proof:themBa}. Its parts \ref{item:them:deuxieme_b} and \ref{item:them:deuxieme_c} are proved p.\ \pageref{proof:them:deuxieme} after some more auxiliary results. In the statement below, see Notation \ref{lem:existenceT}\ref{p2:def_L_R} for $L_\ast$, Definition \ref{defi:atomic} for ``atomic times'' and Definition \ref{defi:essentiel}  for ``essential atomic'' intervals or times.

\begin{them}\label{them:deuxieme_sec}\maz Let $(\mu_t)_{t\in \R}$ be a family of probability measures on $\R$. In points \ref{item:them:deuxieme_b} and \ref{item:them:deuxieme_c}, $(R_n)_n$ stands for a nested sequence of finite subsets of\/ $\R$ and $R$ for $\cup_nR_n$.\smallskip

\point\label{item:them:deuxieme_a} There is a countable set $R\subset\R$ satisfying:
\begin{gather}\label{eq:convergence_doug}
\text{for all $(s,t)$ with $s<t$,}\ L_{R\cap\op]s,t\clo[}=L_{\op]s,t\clo[}.\smallskip
\end{gather}

\point\label{item:them:deuxieme_b} If $R$ satisfies \eqref{eq:convergence_doug} then (see Reminder \ref{reminder:weak_convergence} for the weak convergence):
\begin{equation}\label{eq:cv}
\text{$(\como_{[R_n]})_n$ converges weakly to $\mq$.}\smallskip
\end{equation}

\point\label{item:them:deuxieme_c} Conversely, if $(R_n)_n$ satisfies \eqref{eq:cv}, then:
\begin{enumerate}
\item[\souspoint\label{item:ordre_facultatif}] For any nested finite sets $(R'_n)_n$ such that $R=\cup_nR'_n$, $(\como_{[R'_n]})_n\rightarrow_{n\rightarrow\infty}\mq$. In other words,  \eqref{eq:cv} is a property of $R$. Moreover, \eqref{eq:cv} is also satisfied by any countable $R''\supset R$. 
\item[\souspoint\label{item:inatomique_facultatif}] Let $E\subset R$ be the set of non-atomic times of $R$, then $R\setminus E$ satisfies \eqref{eq:cv}. Moreover, for any finite set $E'$ of non-essential atomic times, there is a set $R'$ satisfying \eqref{eq:cv} and such that $R'\cap E'=\varnothing$.
\item[\souspoint\label{item:essentiel_necessaire}] The set $R$ meets each essential atomic interval of $(\mu_t)_t$, hence in particular, it contains all its essential atomic times (which are at most countably many, by Proposition \ref{pro:denombrable}).
\end{enumerate}
\end{them}

\begin{rem}\maz\point\ We see no simple condition on $R$ that, added to Condition \ref{item:them:deuxieme_c}\ref{item:essentiel_necessaire} above, is necessary and sufficient to ensure $\como_{[R_n]}\to\mq$ in Theorem \ref{them:deuxieme_sec}. For instance, density in the union of the atomic times is neither necessary, see Example \ref{ex:intervalle_essentiel}, nor sufficient: take, e.g.,  $\mu_t=\la\lfloor_{[0,\frac12]}+\frac12\delta_1$ if $t\in\Q$ and $\mu_t=\frac12\delta_0+\la\lfloor_{[\frac12,1]}$ otherwise, then there is no essential atomic interval, all time is atomic, and $R$ suits if and only if $R\cap\Q$ is dense in $\Q$ and $R\cap(\R\setminus\Q)$ is dense in $\R\setminus\Q$, so that {\em any} set dense in $\R$ is not suitable. One may think to ``$R$ is the projection on $\R$ of a set dense in the set of atomic levels $A=\cup_{t\in\R}\big(\{t\}\times A_t\bigr)\subset\R\times [0,1]$'' (see Notation \ref{nota:Asx}) as a condition at least sufficient, but it is not. Take $\mu_t=\frac12\delta_0+\frac12\delta_1$ if $t\in\Q$, otherwise $\mu_t=\delta_0$, then $E=(\Q\cap([0,1]\setminus\{\frac12\}))\times\Q$  is dense in $A$ but $R=\pr_{\R}(E)$ does not suit.

\point\ In point \ref{item:them:deuxieme_c}\ref{item:inatomique_facultatif} of Theorem \ref{them:deuxieme_sec}, any non-essential atomic time $t$ may be avoided by a set $R$ satisfying \eqref{eq:cv}, but it is not true that if $R$ satisfies \eqref{eq:cv}, any such $t\in R$ may be removed from $R$ without making \eqref{eq:cv} false: see Example \ref{ex:intervalle_essentiel} where $R$ consists of a single non-essential atomic time.

\point\ Condition \eqref{eq:convergence_doug} implies \eqref{eq:cv}, but notice that it is not necessary. Take, e.g.,  $\mu_t=\la\lfloor_{[0,1]}$ for $t<0$ and $\mu_t=\delta_0$ otherwise. Then $R$ satisfies  \eqref{eq:convergence_doug} if and only if $R\cap \R^+\neq\varnothing$, but $\como$ is Markov, so that \eqref{eq:cv} is true with $R=\varnothing$.
\end{rem}

\begin{lem}\label{lem:croissance_lc} Let $\Tt$ denote a totally ordered set of indices (in practice, $\Tt=\R^+$ or $\Tt=\N$). If $(R_\tau)_{\tau\in \Tt}$ is a family of subsets of\/ $\R$, increasing for the inclusion, setting $R=\cup_\tau R_\tau$, $(L_{R_\tau})_{\tau\in\Tt}$ is increasing for $\leqlc$ and tends weakly to $L_{R}$ when $\tau$ tends to infinity.
\end{lem}

\begin{proof} It rests on the following remark. By definition of $\leqlc$ in Definition \ref{defi:geqlc} and of $\rho$ in Proposition \ref{pro:rho}, if $A$, $A'$, $A''$ are measures of $\p(\R^d)$ with the same marginals and $A\leqlc A'\leqlc A''$, then $\rho(A',A'')\leqslant\rho(A,A'')$. Now we prove the lemma. By the definition of $L_{R_\tau}$ in Lemma \ref{lem:existenceT}\ref{p2:def_L_R} the sequence $(L_{R_\tau})_{\tau\in\Tt}$ is increasing, be the sets $R_\tau$ finite or not. By Lemma \ref{lem:existenceT}(c), and since $\rho$ metrizes the weak topology (Proposition \ref{pro:rho}), for any $\eps>0$ we find a finite $R'\subset R$ such that $\rho(L_R,L_{R'})\leqslant\eps$. Since $R=\cup_\tau R_\tau$ there is a $\tau_0$ such that $R_{\tau_0}\supset R'$. Then: $\tau\geqslant \tau_0\Rightarrow L_{R'}\leqlc L_{R_\tau}\leqlc L_{R}\Rightarrow \rho(L_{R'},L_{R_\tau})\leqslant\eps$ by the remark.
\end{proof}

\begin{lem}\label{lem:convergence_couples}
There is a nested sequence $(R_n)_{n\in\N}$ of finite subsets of $\R$ such that $s<t$ $\Rightarrow$ $L_{R_n\cap\op]s,t\clo[}\rightarrow L_{\op]s,t\clo[}$
\end{lem}
\begin{proof} Let $((u_k,u'_k))_{k\in\N}$ be a dense sequence in $\{(x,y)\in\R^2:\,x< y\}$ and for every $n\geq 1$, $(R_k^n)_{k\in\N}$ a nested sequence of finite subsets of $\R$ such that  $\rho(L_{R_k^n\cap\op]u_k,u'_k\clo[},L_{\op]u_k,u'_k\clo[})\leqslant\frac1n$ for every $k\in \N$. Denote $\cup_{k=1}^n R^n_k$ by $R_n$. Then for any $s$ and $t>s$, $L_{R_n\cap\op]s,t\clo[}\rightarrow_{n\rightarrow\infty}L_{\op]s,t\clo[}$. Let us prove it. Fix $\eps>0$. By Lemma \ref{lem:croissance_lc} (use a subsequence $(k_n)_n$ such that $(\op]u_{k_n},u'_{k_n}\clo[)_n$ is an exhaustion of $\op]s,t\clo[$), there exists $k$ such that $\rho(L_{\op]s,t\clo[},L_{\op]u_k,u'_k\clo[})<\eps/2$. Then $L_{\op]u_k,u_k'\clo[\cap R_n}\leqlc L' \leqlc L_{\op]s,t\clo[}$ where $L'$ stands for $L_{\op]s,t\clo[\cap R_n}$ or $L_{\op]u_k,u_k'\clo[}$ for any $k$. For $n\geq\max\{k,2/\eps\}$:
$$\rho(L_{\op]s,t\clo[},L_{\op]u_k,u'_k\clo[\cap R_n})\leq \rho(L_{\op]s,t\clo[},L_{\op]u_k,u'_k\clo[})+ \rho(L_{\op]u_k,u'_k\clo[},L_{\op]u_k,u'_k\clo[\cap R_n})<\eps.$$
Therefore with $L'=L_{\op]s,t\clo[\cap R_n}$ we obtain the desired convergence to $L_{\op]s,t\clo[}$.
\end{proof}

\begin{defi}\label{defi:essentiel}Let $I$ be an interval. If, for some interval $J\supset I$ such that $J\setminus I$ is disconnected (then for all such smaller $J'$), $L_{J}\neq L_{J\setminus I}$, we call $I$ an \emph{essential atomic interval of $(\mu_t)_{t\in\R}$}. If $I=\{t\}$ is essential, we call $t$ an \emph{essential atomic time of $(\mu_t)_{t\in\R}$}.
\end{defi}

To check the parenthesis in the definition, suppose that $L_{J'}= L_{J'\setminus I}$ and get $L_{J}= L_{J\setminus I}$ by Proposition \ref{pro:compo}.

\begin{pro}\label{pro:denombrable}If a nested sequence $(R_n)_n$ is as in Lemma \ref{lem:convergence_couples} then $\cup_nR_n$ contains all the essential atomic times of $(\mu_t)_t$. In particular, those times are at most countably many.
\end{pro}

\begin{proof}We show a contrapositive result. Suppose that $t$ is an essential atomic time, that $s<t$ and $s'>t$ are such that $L_{\op]s,s'\clo[\setminus\{t\}}\neq L_{\op]s,s'\clo[}$, and that $(R_n)_n$ is a nested sequence of finite sets such that $t\not\in\cup_nR_n$, that $L_{R_n\cap\op]s,t\clo[}\to L_{\op]s,t\clo[}$ and that $L_{R_n\cap\op]t,s'\clo[}\to L_{\op]t,s'\clo[}$. Then:
\begin{align*}
L_{R_n\cap\op]s,s'\clo[}&=L_{R_n\cap\op]s,t\clo[}.L_{R_n\cap\op]s,t\clo[}\ \ \text{by Proposition \ref{pro:compo}}\\
&\to L_{\op]s,t\clo[}.L_{\op]t,s'\clo[}\ \ \text{by assumption and by Proposition \ref{pro:continuity}}\\
&=L_{\op]s,s'\clo[\setminus\{t\}}\neq L_{\op]s,s'\clo[}\ \ \text{by Proposition \ref{pro:compo} and the assumption}.
\end{align*}
Therefore, $(R_n)_n$ cannot be as in Lemma \ref{lem:convergence_couples}.
\end{proof}

\begin{rem}\label{rem:intervalles_essentiels} Be careful that the property to be an essential atomic interval is true in general neither for a union of two such (intersecting) intervals, nor for their intersection, nor for an interval containing a such interval or included in it.
\end{rem}

Remark \ref{rem:essentiel_divers} and Example \ref{ex:intervalle_essentiel} are qualitative comments on the notions introduced in Definition \ref{defi:essentiel}. They are independent of the proof of Theorem \ref{them:deuxieme_sec} and may be skipped in a first reading.

\begin{rem}\label{rem:essentiel_divers}\maz\point\ 
Essential atomic times are of course atomic. Indeed, else, $L_{\{t\}}$ is the identity transport so $L_I.L_{\{t\}}.L_{I'}=L_I.L_{I'}$ for all intervals with $\sup I\leqslant t \leqslant \inf I'$ and $t$ cannot be essential.

\point\label{item:suffisante} Suppose, to simplify, that some $\mu\in\p(\R)$ has exactly one atom $x$ and consider a family $(\mu_t)_{t\in\R}$ such that $\mu=\mu_0$. An obvious sufficient condition for $0$ to be unessential is to chose $\mu_t$ such that for a certain sequence $t_n\to0$, $\mu_{t_n}$ has an atom $x_n$ such that $A_{t_n,x_n}\supset A_{t,x}$ (see Notation \ref{nota:Asx}), or even  only $\forall \varepsilon>0$, $\exists n_0:$ $n\geqslant n_0\Rightarrow A_{t_n,x_n}+\op]-\varepsilon,\varepsilon\clo[\supset A_{t,x}$: ``atoms merging the same levels of quantile as $x$ merges at $t=0$, accumulate on $0$''. Indeed, taking possibly a subsequence, we may suppose that $(t_n)_n$ tends to zero from the right or the left (say, from the right). For any $s>t$, the function $x\mapsto k_{L_{[t_n,s\clo[}}(x,\espcdot)$ (when it is defined i.e.\ $t_n<s$) is constant on $A_{t_n,x_n}$, hence since $L_{\op]0,s\clo[}=\lim L_{[t_n,s\clo[}$ (see Lemma \ref{lem:croissance_lc}) $x\mapsto k_{L_{]0,s\clo[}}(x,\espcdot)$ is constant on $A_{t,x}$ (use Proposition \ref{pro:rho}). Considering $k_{L_{\{0\}}}$ described in Remark \ref{rem:finesse}\ref{p3:rem:finesse}, the composition formula in \S\ref{ss:compos} yields $L_{\op]0,s\clo[}=L_{\{0\}}.L_{\op]0,s\clo[}$; hence $L_{\op]s',s\clo[\setminus\{t\}}=L_{\op]s',s\clo[}$ for all $s'<t$.

If $\mu_0$ has several atoms $(x_i)_{i\in{\cal I}}$, a similar statement can be shown, the condition being that each of the intervals $A_{0,x_i}$ has the property above.

\point\label{item:necessaire} A necessary condition for an atomic time $t$ to be unessential is immediate: $\{t\}\times A_{t}$ shall be included in $E=\overline{A_{\op]-\infty,t\clo[}}\cap\overline{A_{\op]t,+\infty[}}$, where $A_{I}\eqdef\cup_{r\in I}\bigl(\{r\}\times A_{r}\bigr)$. There is some nonempty open interval $J$ such that $\bar{J}\subset A_t\setminus E$. Then for $\varepsilon$ small enough, $L_{\op]t-\eps,t+\eps\clo[}$ restricted to $J^2$ is the identity transport, which prevents $t$ from being essential (see Remark \ref{rem:essentiel_divers_bis}).

\point\ The condition of point \ref{item:suffisante} is not necessary, nor that of point \ref{item:necessaire} sufficient. For instance take $\mu_0=\delta_0$ and for $t\neq0$, $(\mu_t)_t=a(t)\delta_0+(1-a(t))\delta_1$. If $a$ has unbounded total variation on any interval $\op]0,r\clo[$, then any $L_{\op]0,r\clo[}$ is the uniform measure on $[0,1]^2$, see Example \ref{ex:two_atoms}, thus $0$ is not essential (in fact, even not right-essential in the sense given in Remark \ref{rem:essentiel_divers_bis}). On the contrary if $a$ has bounded total variation, none of the  measures $L_{\op]0,r\clo[}$ or $L_{\op]r,0\clo[}$ is uniform, hence $0$ is essential (see Remark \ref{rem:essentiel_divers_bis}).
\end{rem}

\begin{ex}\label{ex:intervalle_essentiel}A family $(\mu_t)_{t\in\R}$ may have an essential atomic interval and no essential atomic time. Let $I$ be any interval (but not a singleton, for our purpose) and take, e.g.,  $\mu_t=\delta_0$ for $t\in I$ and $\mu_t=\la\lfloor_{[0,1]}$ otherwise, then $I$ is the only essential atomic interval of $(\mu_t)_t$. Here, $\como_{[\{t_0\}]}=\mq$ for any $t_0\in I$. 
\end{ex}

We can now prove Theorem \ref{them:deuxieme_sec}\ref{item:them:deuxieme_a}.\medskip

\noindent{\em Proof of Theorem \ref{them:deuxieme_sec}\ref{item:them:deuxieme_a}.}\label{proof:themBa} Take $(R_n)_n$ a sequence as given by Lemma \ref{lem:convergence_couples}. Then, for any $s$ and $t>s$:
\begin{align*}
L_{R\cap\op]s,t\clo[}&=\lcsup\{L_{R_n\cap\op]s,t\clo[}\}\ \ \text{by Lemma \ref{lem:croissance_lc}}\\
&\geqlc\lcsup\{L_{R'_n\cap\op]s,t\clo[}\}\ \ \text{as $R'_n\subset R_n$ and by Lemma \ref{lem:existenceT}\ref{p1:lem:existenceT}}\\
&=\lim_{n\to\infty}L_{R'_n\cap\op]s,t\clo[}=L_{\op]s,t\clo[}\ \ \text{\begin{tabular}[t]{l}by Lemma \ref{lem:croissance_lc} and\\ by property of the sets $R'_n$.\end{tabular}}
\end{align*}
But by definition, $L_{R\cap\op]s,t\clo[}\leqlc L_{\op]s,t\clo[}$, thus $L_{R\cap\op]s,t\clo[}= L_{\op]s,t\clo[}$.
\qed\medskip

To provide a clear proof of Theorem \ref{them:deuxieme_sec}\ref{item:them:deuxieme_b} we introduce an auxiliary notion in Definition \ref{defi:fits} below. 
Suppose that some $P\in \ma(\mu,\mu_1,\ldots,\mu_k)$ is disintegrated as $P=\join(\mu,k_{P})$ and that $g:\R\to \R$ and $h:\R^k\to \R^k$ are measurable maps. May we disintegrate $(g\otimes h)_\#P$? In case $g$ is into one easily checks $(g\otimes \id)_\#P=\join(g_\#\mu,k^g_P)$ where $k^g_P$ is defined by $k^g_P(y,\cdot)= k^g_P(g^{-1}(y),\cdot)$, $g_\#\mu$-almost surely. Otherwise, the next notion and lemma will enable us to obtain a similar disintegration, and associated properties.

\begin{defi}\label{defi:fits}
We say that $g:\R\to \R$ \emph{fits} $P\in \ma(\mu,\mu_1,\ldots,\mu_k)$ if there exists a kernel $k^g_P$ such that $k^g_P(g(x),\cdot)=k_{P}(x,\cdot)$, $\mu$-almost surely.
\end{defi}

\begin{rem}\label{rem:random_from_center}\maz \point\label{item:rem:random_from_center_a} If $g$ fits $P=\join(\mu,k)$ and $h$ is a measurable map from $\R^k$ into itself we can disintegrate $(g\otimes h)_\#P$ as $\join(g_\#\mu,h_\#k^g_P)$
where $h_\#k^g_P(y,\cdot)= k^g_P(y,h^{-1}(\cdot))$.

\point\label{item:rallonge} If $g$ fits  $P$, it also fits $P.P'$ and $P\circ P'$.

\point\label{item:rem:random_from_center_commute_avec_circ} If $g$ fits  $P\in \ma(\mu,\nu_1,\ldots,\nu_k)$ or $Q\in\ma(\mu,\nu'_1,\ldots,\nu'_{k'})$, and if $f:\R^k\rightarrow\R^k$ and $h:\R^{k'}\rightarrow\R^{k'}$ are measurable maps, then:
$$(f\otimes g \otimes h)_\#(\trans P\circ Q)=(f\otimes g)_\#\trans P\circ (g \otimes h)_\#Q.$$

The proofs are direct, using the definitions of a kernel, composition and catenation in Subsection \ref{subsec:caten}. Notice that, in case $\mu=\la\lfloor_{[0,1]}$ and $g$ is a quantile function (which is the only case in which we will use the remark), point \ref{item:rem:random_from_center_commute_avec_circ} is a particular case of Lemma \ref{lem:distributivite}\ref{item:lem:distributiviteB} below. In the language of this lemma, ``$g$ fits $\trans P$ or $Q$'' means that $\trans P.L=\trans P$ or $L.Q=Q$, which both imply in particular that $\trans P.L.Q=\trans P.Q$
\end{rem}

We will need the following little technical result.

\begin{lem}\label{lem:int_fg}Take $a<b$ in $\R$ and $f$ and $g$ two positive bounded increasing functions from $\op]a,b\clo[$ to $\R$. Then $\frac1{b-a}\int_a^b fg\,\dd\la\geq \Bigl(\frac1{b-a}\int_a^b f\dd\la\Bigr)\Bigl(\frac1{b-a}\int_a^b g\dd\la\Bigr)$, with equality if and only if $f$ or $g$ is constant.
\end{lem}

\begin{proof} The two measures $\mu=f\la$ and $\nu=\frac1{b-a}\int_a^b f\,\dd\la$ have the same mass $\int_a^b f\,\dd\la$. Besides, $\nu\leqs\mu$; indeed, for any $c\in\op]a,b\clo[$, since $f$ is increasing:
$$\frac1{c-a}\mu(\op]-\infty,c])=\frac1{c-a}\int_{a}^c f\,\dd\la\leq\frac1{b-a}\int_{a}^b f\,\dd\la=\frac1{c-a}\nu(\op]-\infty,c]).$$
Hence, since $g$ is  increasing, $\int g\,\dd\mu\geq\int g\,\dd\nu$, equivalent to the wished inequality.

Plainly, equality occurs if $f$ or $g$ is constant. If $f$ is not constant, we in fact proved that equality in $\int g_c \dd \mu\leq \int g_c \dd \nu$ holds for $g_c=\one_{[c,b]}$ if and only if $c\in\{a,b\}$. This remains true for positive combinations of functions $g_c$ and, being a little careful, for limits of them.
\end{proof}

\begin{lem}\maz\label{lem:distributivite} Take $\mu\in\p(\R)$, denote by $g=G_\mu$ its quantile function and $F=F_\mu$ its cumulative distribution function.  Take $L=\como(\la,\mu).\como(\mu,\la)$ and, similarly as in Notation \ref{nota:Asx}, let $(b_i)_{i\in{\mathcal{I}}}$ be the atoms of $\mu$ and for each $i$, $A_i$ be the interval $\op]F(b_i^-),F(b_i)\clo[$ of quantile levels merged by $g$ on $b_i$. Finally set $A=\cup_{i\in{\cal I}}A_i$. Moreover, take $((\nu_i)_i,(\nu'_i)_i)\in\p(\R)^{k+k'}$, $P\in \ma(\la\lfloor_{[0,1]},\nu_1\otimes\ldots\otimes\nu_k)$ and $Q\in\ma(\la_{[0,1]},\nu'_1\otimes\ldots\otimes\nu'_{k'})$. Suppose that $P$ and $Q$ have increasing kernel (for $\leqlc$ in place of $\leqs$ if $k\geq 2$ or $k'\geq2$). Then:\smallskip

\point\label{item:lem:distributiviteA}  $\trans P.L.Q$ equals $\trans P.Q$ if and only if for each $i\in{\mathcal I}$, at least one of the two kernel functions $x\mapsto k_P(x,\ \cdot\ )$ or $x\mapsto k_Q(x,\ \cdot\ )$ is constant on $A_i$,\smallskip

\point\label{item:lem:distributiviteB} For any measurable maps $f:\R^k\rightarrow\R^k$ and $h:\R^{k'}\rightarrow\R^{k'}$, $\trans P.L.Q=\trans P.Q\ \Rightarrow\ (f\otimes g \otimes h)_\#(\trans P\circ Q)=(f\otimes g)_\#\trans P\circ (g \otimes h)_\#Q.$
\end{lem}

\begin{proof} We recall particularly here that we have throughout in mind the analogy between products of transport plans, or of their kernels, and products of matrices, hinted at in Remark \ref{rem:matrix}. In the following, $B_1\subset\R^k$ and $B_3\subset\R^{k'}$ stand for any sets of the type $\prod_j\op]-\infty,d_j]$ and $B_2$ for any interval $\op]-\infty,b]$; $(a,b,c)$ stand for variables in the target of $(f,g,h)$, and $(x,y,z)$ for variables in their sources. The coupling $L$, equal to $\trans q_r.q_r$ (with $\mu=\mu_r$) introduced in Remark \ref{rem:finesse}, is described in this remark. It is such that $\trans\!P.L.Q=(\trans\!P.L).(L.Q)=\trans((g\otimes \Id_{k})_\#P).((g\otimes \Id_{k'})_\#Q)$. In turn, for any measurable functions $f:\R^k\to\R^k$ and $h:\R^{k'}\to\R^{k'}$:
\begin{align}
&\int_{B_2}k_{(g\otimes f)_\#P}(b,B_1).k_{(g \otimes h)_\#Q}(b,B_3)\,\dd\mu(b)\notag\\
=\ &\sum_{i\in\mathcal{I},\,b_i\in B_2}\mu(b_i)[k_{(g\otimes f)_\#P}(b_i,B_1).k_{(g\otimes h)_\#Q}(b_i,B_3)]\notag\\
&+\int_{B_2\setminus\cup_i\{b_i\}}k_{P}(g^{-1}(b),f^{-1}(B_1)).k_{Q}(g^{-1}(b),h^{-1}(B_3))\,\dd\mu(b)\notag\\
&\qquad\qquad\text{as $g$ is injective outside of $A$}\notag\\
&=\sum_{i\in \mathcal{I},\,b_i\in B_2}\frac1{\la(A_i)}\int_{A_i}k_{P}(y,f^{-1}(B_1))\,\dd\la(y)\int_{A_i}k_{Q}(y,h^{-1}(B_3))\,\dd\la(y)\notag\\
&+\int_{g^{-1}(B_2)\setminus A}k_{P}(y,f^{-1}(B_1)).k_{Q}(y,h^{-1}(B_3))\,\dd\la(y)\label{eq:prelim_distributivite}
\end{align}
Then, using for instance the expression \eqref{eq:catenation3} given in Definition \ref{defi:catenation} of the catenation, and the fact that for any transport plans $R$ and $R'$, $R.R'=\pr^{1,3}_\#(R\circ R')$ we get:
\begin{align*}
&(\trans P.Q-\trans((g\otimes \Id)_\#P).(g\otimes \Id)_\#Q)(B_1\times B_3)\\
=\ &\int_{[0,1]}k_{P}(y,B_1)k_Q(y,B_3)\dd y-\int_{\R}k_{(g\otimes \Id)_\#P}(b,B_1)k_{(g \otimes {\rm Id})_\#Q}(b,B_3)\dd\mu(b)\\
=\ &\sum_{i\in{\cal I}}\int_{A_i}k_{P}(y,B_1)k_{Q}(y,B_3)\dd y-\frac1{\la(A_i)}\int_{A_i}k_{P}(y,B_1)\dd y\int_{A_i}k_{Q}(y,B_3)\dd y
\end{align*}
by \eqref{eq:prelim_distributivite} with $B_2=\R$, $f=\Id$, $h=\Id$ ---notice in particular that the two terms of the difference, viewed as integrals on $[0,1]$, differ possibly only on $A\subset[0,1]$. Now, after Lemma \ref{lem:int_fg} applied on each $A_i$ with $f_i:y\mapsto k_{P}(y,B_1)$ and $g_i:y\mapsto k_{Q}(y,B_3)$, which are increasing since $P$ and $Q$ have increasing kernel, each term of the sum is non-negative. Therefore, equality holds if and only if each vanishes, which again by Lemma \ref{lem:int_fg} means that for each $i$, $f_i$ or $g_i$ is constant. This, holding for any $B_1$ and $B_3$, means point \ref{item:lem:distributiviteA}. For point \ref{item:lem:distributiviteB}:

\begin{align*}
&\trans((g\otimes f)_\# P)\circ ((g \otimes h)_\#Q)(B_1\times B_2\times B_3)\\
=\ &\int_{B_2} k_{(f\otimes g)_\#P}(b,B_1)k_{(g\otimes h)_\#Q}(b,B_3)\,\dd \mu(b)\quad\text{by \eqref{eq:catenation3}}\\
=\ &\int_{g^{-1}(B_2)\setminus A} k_{P}(y,f^{-1}(B_1))k_{Q}(y,h^{-1}(B_3))\dd y\\
&\!\!\!\!+\sum_{i\in \mathcal{I},\,b_i\in B_2}\!\!\frac1{\la(A_i)}\int_{A_i}k_{P}(y,f^{-1}(B_1))\dd y\int_{A_i}k_{Q}(y,h^{-1}(B_3))\dd y\quad\text{by \eqref{eq:prelim_distributivite},}\\
=\ &\int_{g^{-1}(B_2)\setminus A} k_{P}(y,f^{-1}(B_1))k_{Q}(y,h^{-1}(B_3))\dd y\\
&\quad+\sum_{i\in \mathcal{I},\,b_i\in B_2}\!\int_{A_i}k_{P}(y,f^{-1}(B_1))k_{Q}(y,h^{-1}(B_3))\dd y\quad\text{as on each $A_i$,}\\
&\qquad\text{by \ref{item:lem:distributiviteA}, $k_{P}(\,.\,,f^{-1}(B_1))$ or $k_{Q}(\,.\,,h^{-1}(B_3))$ is a constant function}\\
=\ &\int_{g^{-1}(B_2)} k_{P}(y,f^{-1}(B_1))k_{Q}(y,h^{-1}(B_3))\dd y\\
=\ &(f\otimes g \otimes h)_\#(\trans P\circ Q)(B_1\times B_2\times B_3).\qedhere
\end{align*}
\end{proof}

\begin{rem}\label{rem:essentiel_divers_bis}From Lemma \ref{lem:distributivite}\ref{item:lem:distributiviteA}, we see that $t$ is an essential atomic time if and only if $\mu_t$ has at least one ``essential atom'' $x$ in the following sense:\smallskip

-- $x$ is a ``left-essential'' atom, in the sense that for some $\eps>0$, if $s\in\op]t-\eps,t\clo[$, the kernel $x\mapsto \trans\ell_{\op]s,t\clo[}(x,\espcdot)$ (see Notation \ref{nota:Asx}) is not constant on $A_{t,x}$,\smallskip

-- and $x$ is a ``right-essential'' atom, in the sense that  for some $\eps>0$,  if $s\in\op]t,t+\eps\clo[$, the kernel $x\mapsto\ell_{\op]t,s\clo[}(x,\espcdot)$ is not constant on $A_{t,x}$.\smallskip

\noindent Calling ``left-essential'' an atomic time $t$ such that $\mu_t$ has at least one left essential atom (and symmetrically ``right-essential'' atomic times), left- or right-essential atomic times must not be essential. However, left- or right-essential atomic times are also at most countably many. We provide a sketch of proof of that, written for right-essential times: show that any such time is a discontinuity point of some function $\varphi_s:\op]-\infty,s]\ni u\mapsto L_{[u,s]}$. Now if some $\varphi_s$ is discontinuous at $t$, all $\varphi_{s'}$ are, for $s'\in\op]t,s]$ (use Remark \ref{rem:multiplicationcontractante}). Hence, the union of the sets of discontinuity points of the functions $\varphi_s$, for $s\in\R$, is the same as their union for $s\in\Q$. Finally the claim below, left to the reader, implies that for each $\varphi_s$, this set is at most countable, which gives the result.\smallskip

\noindent{\em Claim.} If $\psi:u\mapsto M_u\in \ma(\la,\la)$ is increasing or decreasing for $\leqlc$, it has at most countably many (weak) discontinuity points.
\end{rem}

Now we can prove the end of Theorem \ref{them:deuxieme_sec}, i.e.\ its parts \ref{item:them:deuxieme_b} and \ref{item:them:deuxieme_c}.
\begin{proof}[Proof of Theorem \ref{them:deuxieme_sec}\ref{item:them:deuxieme_b}-\ref{item:them:deuxieme_c}]\label{proof:them:deuxieme}

Let us prove point \ref{item:them:deuxieme_b}. Take $R\subset\R$ and $\lev_R$ and $\lev$ in $\ma((\la_t)_{t\in\R})$ given by Definition \ref{defi:lev}. We will need the following claim, that extends Remark \ref{rem:mini_et_Q}, and relies eventually on Remark \ref{rem:finesse}\ref{p2:rem:finesse}:\medskip

\noindent{\em Claim.} If $R$ is finite, $G_\#\lev_{R}=\como_{[R]}$, where $G=(\otimes_{t\in\R} G_t)$.\medskip

\noindent Let us prove it. Take $R\subset\R$ finite and $n$ its cardinal. We must prove that for any finite $S=\cup_{i=1}^N\{s_i\}$, $G_\#\lev_R^{s_1,\ldots,s_N}=\como_{[R]}^{s_1,\ldots,s_N}$ where, by an abuse of notation we will often make use of, $G$ stands for $\otimes_{i=1}^N G_{s_i}$. It suffices to prove it  in the case $S\supset R$, which we suppose now. We introduce the cardinals $\ell_0,\ldots,\ell_n$ of the subsets of $S\setminus R$ situated between to consecutive elements $r_k$ and $r_{k+1}$ of $R\cup\{\pm\infty\}$. We reindex those subsets as $\{s_1^k,\ldots,s_{\ell_k}^k\}$. Since $\lev_R$ is Markov, $(\pr^S)_\#\lev_R=B_0\circ\cdots\circ B_m$, where:
\begin{align*}
B_k=L_{r_{k}}\circ\id_{\ell_{k}}\circ L_{r_{k+1}}=
\begin{cases}
\id_{\ell_0}\circ L_{r_1}&\text{if }k=0\\
L_{r_m}\circ \id_{\ell_m}&\text{if }k=m\\
L_{[r_k,r_{k+1}]}=L_{r_{k}}. L_{r_{k+1}}&\text{if }\ell_k=0\\
L_{[r_k,r_{k+1}]}=L_{r_{k}}\circ L_{r_{k+1}}&\text{if }\ell_k=1\\
L_{r_{k}}\circ\underbrace{\id_2\circ\cdots\circ\id_2}_{\ell_k-1}\circ L_{r_{k+1}}&\begin{nnarray}[t]{l}\text{otherwise,}\\\text{\quad i.e.\ if }\ell_k\geq 2,\end{nnarray}
\end{cases}
\end{align*}
hence we must prove that: $G_\#(B_1\circ\cdots\circ B_m)=\como_{[R]}^{s_1,\ldots,s_N}$. For all $r\in R$, the quantile function $G_r=G_{\mu_r}$ fits $q_r=\join(\mu_r,x\mapsto \delta_{G_r(x)})$, so we may take $k^{G_r}_{q_r}:x\mapsto \delta_x$ as given by Definition \ref{defi:fits}. Now, by Remark \ref{rem:random_from_center}\ref{item:rallonge}, it also fits $L_{\{r\}}=\trans L_{\{r\}}=q_r.\trans{q_r}$. Therefore, by Remark \ref{rem:random_from_center} \ref{item:rallonge}  and \ref{item:rem:random_from_center_commute_avec_circ}, $G_\#(B_1\circ\cdots\circ B_m)=B'_1\circ\ldots\circ B'_m$, where:
\begin{align*}
&B'_k=G_\# B_k=\\
&\begin{cases}
(\otimes_{s\leqslant r_1} G_{s})_\#\id_{\ell_0}\circ L_{r_1}&\text{if }k=0\\
(\otimes_{s\geqslant r_m} G_{s})_\#L_{r_m}\circ \id_{\ell_m}&\text{if }k=m\\
(G_{r_k}\otimes G_{r_{k+1}})_\#(L_{r_{k}}.L_{r_{k+1}})&\text{if }\ell_k=0\\
(G_{r_k}\otimes G_{s_1^k}\otimes G_{r_{k+1}})_\#(L_{r_{k}}\circ L_{r_{k+1}})&\text{if }\ell_k=1\\
(\otimes_{r_k\leqslant s\leqslant r_{k+1}}G_s)_\#(L_{r_{k}}\circ\id_{\ell_k}\circ L_{r_{k+1}})&\text{otherwise, i.e. }\ell_k\geq 2.
\end{cases}
\end{align*}

\noindent Proving that $B'_k=\como(\mu_{r_k},\mu_{s^k_1},\ldots,\mu_{s^k_{\ell_k-1}},\mu_{r_{k+1}})$ will now prove the claim. For simplicity we assume $k\notin \{0,m\}$ and $\ell_k\geq 1$ but the other cases, which are simpler,  can be proved similarly. Observe that:\smallskip

-- by definition of $\como$ and $G$, $\como(\mu_{r_k},\mu_{s_1^k},\ldots,\mu_{s^k_{\ell_k}},\mu_{r_{k+1}})=\law((G_{r_k}(U),\linebreak[1]\ldots,\linebreak[1]G_{r_{k+1}}(U))$ where $U$ is a variable of law $\la$ on $[0,1]$,\smallskip

-- $(\otimes_{r_k\leqslant s\leqslant r_{k+1}}G_s)_\#(L_{r_{k}}\circ\id_{\ell_k}\circ L_{r_{k+1}})$ is the law of $(G_{r_k}(U_1),\linebreak[1]G_{s^k_1}(U_2),\linebreak[1]\ldots,\linebreak[1]G_{s^k_{\ell_k}}(U_2),\linebreak[1]G_{r_{k+1}}(U_3))$ where $\law(U_1,U_2,\ldots,U_2,U_3)=B_k$. In particular, $\law(U_i)=\la$ for $i=1,2,3$, $\law(U_1,U_2)=L_{r_k}$, and $\law(U_2,U_3)=L_{r_{k+1}}$.\smallskip

\noindent So only the first and last variables may differ. Now, notice that by Remark \ref{rem:finesse}\ref{p2:rem:finesse} applied to $T=L_{\{r\}}=q_r.\trans q_r$, if two variables $(U,V)$ satisfy $\law(U,V)=L_{\{r\}}$, the law of $(G_{r}(U),G_r(V))$ is $\trans q_r.L_{\{r\}}.q_r=\trans q_r.q_r.\trans q_r.q_r=\trans q_r.q_r=\id_{\mu_r,2}$, therefore $G_{r}(U)=G_r(V)$ almost surely. Using this with $r={r_k}$ and $(U,V)=(U_1,U_2)$, respectively $r=r_{k+1}$ and $(U,V)=(U_2,U_3)$, we get respectively $G_{r_k}(U_1)=G_{r_k}(U_2)$ and $G_{r_{k+1}}(U_2)=G_{r_{k+1}}(U_3)$ almost surely. The claim is proven.\medskip

Now take $R\subset \R$ satisfying \eqref{eq:convergence_doug} and $(R_n)_n$ a nested exhaustion of it by finite sets. Let $\widetilde\lev\in\ma((\la_t)_{t\in\R})$ be the Markov process having $\widetilde\lev^{s,t}= \linebreak[1]L_{R\cap[s,t]}$ as 2-marginals. Recall that $L_{R\cap[s,t]}=L_{R\cap \{s\}}.L_{R\cap\op]s,t\clo[}.L_{R\cap \{t\}}=L_{R\cap \{s\}}.L_{\op]s,t\clo[}.L_{R\cap \{t\}}=\lim_{n\rightarrow\infty}\lev_{R_n\cap[s,t]}=\lim_{n\rightarrow\infty}\lev_{R_n}^{s,t}$. It exists by Corollary \ref{cor:consist}, whose assumption is satisfied by Proposition \ref{pro:compo}. All those transport plans
have increasing kernel, hence by Lemma \ref{lem:closed}\label{appel_lemme_closed}, for every $S=\cup_{i=1}^k\{s_i\}$:
\begin{align}
\lev_{R_n}^{s_1,\ldots,s_k}&=\lev_{R_n}^{s_1,s_2}\circ\cdots\circ \lev_{R_n}^{s_{k-1},s_k}\label{eq:lev_egal_catenation}\\
&\to\widetilde \lev^{s_1,s_2}\circ\cdots\circ \widetilde\lev^{s_{k-1},s_k}=\widetilde\lev^{s_1,\ldots,s_k}.\notag
\end{align}
Thus $\lev_{R_n}$ converges weakly to $\widetilde\lev$, and then by Remark \ref{rem:diese_contractant}, $G_\# \lev_{R_n}\to G_\# \widetilde\lev$. We are left with the tasks to prove $G_\# \lev_{R_n}=\como_{[R_n]}$ and $G_\# \widetilde\lev=\mq$. The former is our claim above. Let us prove the latter.\smallskip

\noindent{\em Note}:\label{lieu_preuve_glev} At the beginning of \S\ref{sec:construction} we announced \eqref{eq:Glev}, i.e.\ $G_\#\lev=\mq$. In fact, we prove $G_\# \widetilde\lev=\mq$, which is a bit more difficult. To get \eqref{eq:Glev} the same arguments work, the final reasoning with Lemma \ref{lem:distributivite}\ref{item:lem:distributiviteB} being replaced by a direct use of Remark \ref{rem:random_from_center}\ref{item:rem:random_from_center_commute_avec_circ}, as for all $\{r_1,\ldots,r_\ell\}$, $G_{r_\ell}$ fits $\lev^{r_1,\ldots,r_\ell}$.\smallskip

\noindent We recall that $\mq$ was defined as the unique Markov law with the same marginals of dimension 2 as $G_\# \lev$. But by \eqref{eq:convergence_doug} and Proposition \ref{pro:comp_conv}\ref{ccc}, the 2-marginals of $G_\#\widetilde\lev$ and $G_\#\lev$ are equal. Hence it is sufficient to prove that
$G_\# \widetilde \lev$ is Markov, i.e.\ that for all $(s_1,\ldots,s_k)$, $(G_\#\widetilde\lev)^{s_1,\ldots,s_k}=(G_\#\widetilde\lev)^{s_1,s_2}\circ\ldots\circ(G_\#\widetilde\lev)^{s_{1-1},s_k}$. Since $\widetilde\lev$ is Markov, $(\widetilde\lev)^{s_{2},\ldots,s_{k}}=\widetilde\lev^{s_{2},s_{3}}\circ\ldots\circ\widetilde\lev^{s_{k-1},s_{k}}$; besides, notice the following fact, that we will prove a bit below:\smallskip

\noindent{\em Fact.} For all $(s_i)_{i=1}^k\in\R^k$, $\lev^{s_1,\ldots,s_k}$ viewed as a transport plan from $(\R,\la)$ to $(\R^{k-1},\la^{\otimes k-1})$ has increasing kernel (for the order $\leqlc$ instead of $\leqs$).\smallskip

\noindent Hence we may conclude by using $k-1$ times Lemma \ref{lem:distributivite}\ref{item:lem:distributiviteB}. Let us check the first step. The measures $\trans(\widetilde\lev^{s_1,s_2})$ and $\widetilde\lev^{s_2,\ldots,s_k}$ have increasing kernel so that we have only to show that $\widetilde\lev^{s_1,s_2}.\widetilde\lev^{s_{2},s_3}=\widetilde{\lev}^{s_1,s_{2}}.L_{\{s_{2}\}}.\widetilde{\lev}^{s_{2},s_3}$, i.e., by definition of $\widetilde\lev$, that $L_{R\cap[s_1,s_2]}.L_{R\cap[s_2,s_3]}=L_{R\cap[s_1,s_2]}.L_{s_2}.L_{R\cap[s_2,s_3]}$. This amounts to check that:
$$\left\{\begin{array}{ll}L_{\semi s_1,s_2\clo[}.L_{\op]s_2,s_3\semi}=L_{\semi s_1,s_2\clo[}.L_{s_2}.L_{\op]s_2,s_3\semi}&\ \text{if}\ s_2\not\in R\\  L_{\semi s_1,s_2]}.L_{[s_2,s_3\semi}=L_{\semi s_1,s_2]}.L_{s_2}.L_{[s_2,s_3\semi}&\ \text{otherwise.}\end{array}\right.$$
The second point is true by Proposition \ref{pro:compo}, and if $s_2\not\in R$, by Proposition \ref{pro:denombrable}, $s_2$ is not an essential atomic time, which is the wanted equality. We finally must prove the fact stated above. Actually it is true for any catenation $P_1\circ\ldots\circ P_k$ of couplings $P_i$ from $\R$ to $\R$ with increasing kernel —and $\lev^{s_1,\ldots,s_k}$ is of this type, see \eqref{eq:lev_egal_catenation}. We check this for $k=2$; the same argument, applied by induction, gives the general case. Take $B_2$ and $B_3$ two intervals of the type $\op]-\infty,a]$ and $x\leqslant x'$; we must show that: $k_{P_1\circ P_2}(x,B_2\times B_3)\geqslant k_{P_1\circ P_2}(x',B_2\times B_3)$. As $P_2$ has increasing kernel, the function $k_{P_2}(\espcdot,B_3)$ is decreasing (thus also $\one_{B_2}k_{P_2}(\espcdot,B_3)$). As $P_1$ has increasing kernel, $k_{P_1}(x,\espcdot)\leqs k_{P_1}(x',\espcdot)$. Thus:
$$k_{P_1\circ P_2}(x,B_2\times B_3)\!=\!\int k_{P_1}(x,\dd y) \one_{B_2}k_{P_2}(y,B_3)\!\geqslant\! \int k_{P_1}(x',\dd y) \one_{B_2}k_{P_2}(y,B_3),$$
the wanted result. This proves part \ref{item:them:deuxieme_b} of the theorem.

Now we prove part \ref{item:them:deuxieme_c}. For point \ref{item:ordre_facultatif}, each 2-margin $\como_{[R_n]}^{s,t}$ and $\como_{[R'_n]}^{s,t}$ tends respectively to $\lcsup_n\{\como_{[R_n]}^{s,t}\}$ and $\lcsup_n\{\como_{[R'_n]}^{s,t}\}$, which are equal by Lemma \ref{lem:croissance_lc} since $\cup_nR_n=\cup_nR'_n$. Besides, if $R''\supset R$, taking a nested sequence $(R''_n)_n$ of finite sets such that $\cup_n R''_n=R''$ and $R''_n\supset R_n$ for all $n$, we get that, for any $s$ and $t>s$, $\lim_n\como_{[R''_n]}^{s,t}\geqlc \lim_n\como_{[R_n]}^{s,t}=\mq^{s,t}$, but by the minimality property of Theorem \ref{them:a}\ref{item:minimal}, for all $n$, $\mq^{s,t}\geqlc\como_{[R''_n]}^{s,t}$. Therefore, $\lim_n\como_{[R''_n]}^{s,t}=\mq^{s,t}$ and the result follows. For \ref{item:inatomique_facultatif}, it is sufficient to show that for any finite subsets $R$ and $E$ of $\R$ with $\mu_t$ diffuse for all $t\in E$, $\como_{[R\cup E]}^{s,t}=\como_{[R]}^{s,t}$. This follows plainly from the definitions. After Definition \ref{defi:quantile_discretement_markovien}  it amounts to show that for all $t\in E$, all $\{s_1,\ldots,s_k\}\subset\op]-\infty,t\clo[$ and all $\{s'_1,\ldots,s'_{k'}\}\subset\op]t,+\infty\clo[$, $\como^{s_1,\ldots,s_k,t}\circ\como^{t,s'_1,\ldots,s'_{k'}}=\como^{s_1,\ldots,s_{k},t,s'_1,\ldots,s'_{k'}}$. This comes from a trivial case of Remark \ref{rem:random_from_center}\ref{item:rem:random_from_center_commute_avec_circ} applied with all the measures equal to $\la_{[0,1]}$, $P=\trans\Id_{\la,k+1}$ (i.e.\ $P=\Id_{\la,k+1}$), $Q=\Id_{\la,k'+1}$ and $(f,g,h)=(\otimes_{i=1}^k G_{s_i},G_t,\otimes_{i=1}^{k'} G_{s'_i})$. Indeed, as $\mu_t$ is diffuse, $G_t$ is injective, hence trivially fits $P$ and $Q$. Then:
\begin{align*}
&\como^{s_1,\ldots,s_k,t}\circ\como^{t,s'_1,\ldots,s'_{k'}}\\
=\ &(\otimes_{i=1}^k G_{s_i}\otimes G_t)_\#\Id_{\la,k+1}\circ(G_t\otimes (\otimes_{i=1}^{k'} G_{s'_i}))_\#\Id_{\la,k'+1}\ \ \text{by def.\ of $\como$}\\
=\ &((\otimes_{i=1}^k G_{s_i})\otimes G_t\otimes (\otimes_{i=1}^{k'} G_{s'_i}))_\#\Id_{\la,k+1}\circ\Id_{\la,k'+1}\ \ \text{by the remark}\\
=\ &((\otimes_{i=1}^k G_{s_i})\otimes G_t \otimes (\otimes_{i=1}^{k'} G_{s'_i}))_\#\Id_{\la,k+k'+1}\\
=\ &\como^{s_1,\ldots,s_k,t,s'_1,\ldots,s'_{k'}}.
\end{align*}
To alleviate the writing, we prove the rest of \ref{item:inatomique_facultatif}, with $\# E'=1$ (the general proof is alike). Take $R$ given by point \ref{item:them:deuxieme_a} and $t\in\op]s,s'\clo[$ some unessential atomic time of $(\mu_t)_t$, then $L_{R_n\cap \op]s,s'\clo[\setminus\{t\}}=L_{R_n\cap \op]s,t\clo[}.L_{R_n\cap \op]t,s'\clo[}\rightarrow L_{\op]s,t\clo[}.L_{\op]t,s'\clo[}$ by Lemma \ref{lem:closed}. Now $L_{\op]s,t\clo[}.L_{\op]t,s'\clo[}=L_{\op]s,s'\clo[}$ since $t$ is unessential. Thus $R'=R\setminus\{t\}$ satisfies \eqref{eq:convergence_doug}, hence \eqref{eq:cv}, by point \ref{item:them:deuxieme_b}.

Let us prove \ref{item:essentiel_necessaire}. Suppose that $I$ is some essential atomic interval, i.e.\ there is an interval $J\supset I$ such that $J\setminus I$ is disconnected and  $L_{J} \neq L_{J\setminus I}$, and assume that $I\cap R=\varnothing$. Since $L_{J} \geqlc L_{J\setminus I}$, this means that there is some $(a,a')\in\op]0,1\clo[^2$ such that $L_{J}([0,a]\times[0,a'])< L_{J\setminus I}([0,a]\times[0,a'])$. Denote $(\inf J,\sup J)\in\R^2$ by $(s,s')$. If $a\not\in A_{s}$ and $a'\not\in A_{s'}$ (see Notation \ref{nota:Asx}), i.e.\ if $G([0,a^{(\prime)}])=[0,G(a^{(\prime)})]$, then, pushing the inequality by $G$, and reminding that, since $I\cap R=\varnothing$, $L_{J\setminus I}\geqlc \lim_n L_{R_n\cap J}$, so that $L_{J\setminus I}([0,a]\times[0,a'])\leqslant \lim_n L_{R_n\cap J}([0,a]\times[0,a'])$,  we get:
\begin{equation}\label{eq:mq_neq_como_Rn}
\mq^{s,s'}(\op]-\infty,G(a)]\times\op]-\infty,G(a')])< \lim_{n\to\infty}\como_{[R_n]}^{s,s'}(\op]-\infty,G(a)]\times\op]-\infty,G(a')]).
\end{equation}
Therefore $ \lim_{n\to\infty}\como_{[R_n]}$ cannot be equal to $\mq$ and we are done. If $a\in A_s$, let $A=\op]a_0,a_1\clo[$ be the connected component of $A_s$ containing $a$; notice that $\{a_0,a_1\}\cap A=\varnothing$. We prove that \eqref{eq:mq_neq_como_Rn} holds with $a_0$ or $a_1$ in place of $a$. By construction of $L_E$ for any set $E$ having $s$ as minimum, the functions $b\mapsto k_{L_E}(b,\,.\,)$, the values of which are measures on $[0,1]$, are constant on $A$. Applying this for $E=L_{\{s\}\cup J}$ and $E=L_{\{s\}\cup J\setminus I}$, we get that either (i) below is true, or the restrictions of $L_{ J}$ and $L_{ J\setminus I}$ to $[0,a_0]\times[0,a']$ coincide, thus necessarily (ii) is true (we let $\{s\}\cup J$ appear instead of $J$ but this does not matter by Remark \ref{rem:comp_mini}):\smallskip

{\bf(i)} $L_{\{s\}\cup J}([0,a_0]\times[0,a'])< L_{\{s\}\cup J\setminus I}([0,a_0]\times[0,a'])$,\smallskip

{\bf(ii)} on $A$, $k_{L_{\{s\}\cup J}}(b,[0,a'])< k_{L_{\{s\}\cup J\setminus I}}(b,[0,a'])$,\smallskip

\noindent and if (i) is false and (ii) is true then $L_{ J}([0,a_1]\times[0,a'])< L_{ J\setminus I}([0,a_1]\times[0,a'])$. Hence anyway \eqref{eq:mq_neq_como_Rn} holds with $a_0$ or $a_1$ in place of $a$. Proceed symmetrically for $a'$ if $a'\in A_{s'}$. This shows point \ref{item:essentiel_necessaire}.%
\end{proof}

\begin{rem}\label{rem:demo_nonatomique} The proof of Theorem \ref{them:deuxieme_sec}\ref{item:them:deuxieme_c}\ref{item:inatomique_facultatif} shows directly the last sentence of Remark \ref{rem:como_markovien}\ref{p1:rem:como_markovien}. Indeed it shows that if every $\mu_t$ is diffuse, then for all finite $R\subset\R$, $\como_{[R]}=\como$, which gives an expression of the Markov property, see Definition \ref{defi:markov2}.
\end{rem}

\section{A Markov probabilistic representation of the continuity equation on $\R$}\label{sec:cont_eq}

\subsection{Introduction}\label{subsec:intro5} As we briefly mentioned in \S \ref{ss:minimality} and explain below in Reminder \ref{reminder:wasserstein}, quantile couplings are optimal transport plans for the quadratic cost function. This suggests that the quantile process $\como$ or even the Markov-quantile process $\mq$ could be minimizers of dynamical optimal transport problems. This is true and rather well-known for $\como$; one approach is in \cite{Pass}. In this section we show that this also makes sense for $\mq$, and in which terms it can be formulated. Those terms make sense in dimension greater than one, leading to the question of a generalization of $\mq$ in those dimensions, see Item \ref{p2:intro5} in the list below in this introduction.

Here is the minimization problem at stake. We consider a now classical action introduced by Benamou and Brenier in the context of the incompressible Euler equations, see Definition \ref{defi:action}. If $X=(X_t)_{t\in \R}$ is a process, its action is:
\[
\A(X)=\E\int_{0}^1|\dot{X_t}|^2\,\dd t=\int_{0}^1\E|\dot{X_t}|^2\,\dd t.
\]
Note however that the original definition by Benamou and Brenier involves the velocity vector fields (one usually calls it ``Eulerian'') while we present its ``Lagrangian'' dual action involving the trajectories $t\mapsto X_t$. As it will become clear in this section, this action for infinitely many marginals is simply related to the quadratic transport problem with two marginals.

The origin of this research goes back to the interpretation by Arnold in \cite{Arn} of the solutions of the incompressible Euler equations on a compact Riemannian manifold as geodesic curves in the space of diffeomorphisms preserving the volume $\mathrm{Vol}$. In \cite{BeBr}, Benamou and Brenier relaxed the minimisation problem attached to those geodesics and introduced \emph{generalized geodesics} that are, in probabilistic terms, continuous processes $X=(X_t)_{t\in[0,1]}$ with $\law(X_t)=\vol$ at every time. Their minimisation property is encoded in the fact that they minimize $\A$ under the constraint that the marginals $\law(X_t)$ and $\law(X_0,X_1)$ are prescribed. 

Later, see \cite{JKO, Ott}, Otto and his coauthors discovered that the solutions of some PDEs, in particular the Fokker--Planck and porous medium equations can be thought of as curves of maximal (negative) slope for some functionals $F$ in the space of probability measures endowed with the 2-transport distance (alias Wasserstein distance). It catches a comprehensive picture of the infinite dimensional manifold of measures used in optimal transport, building a differential calculus on it, called ``Otto calculus''. In this context, the derivative of the curve $(\mu_t)_t$ at time $t$ shall be seen as a vector field $v_t$ of gradient type, square integrable with respect to $\mu_t$, such that the transport (or continuity) equation:
\begin{equation}\label{eq:continuity_intro}\frac{\dd}{\dd t}\mu_t+\mathrm{div}(\mu_t v_t)=0
\end{equation}
is satisfied. The speed of the curves of maximal slope of $F$ is $\sqrt{\int |v_t|^2\,\dd\mu_t}$, which corresponds to $\sqrt{\E(|\dot{X}_t|^2)}$ in Benamou--Brenier's action; it has to coincide with the opposite of the slope of $F$ at $\mu_t$, hence the derivative of $t\mapsto F(\mu_t)$ is $-\int |v_t|^2\,\dd\mu_t$.\medskip 

A thorough study of those questions has been conducted in the monograph \cite{AGS} by Ambrosio, Gigli and Savar\'e (see also \cite{Bernard,AS}) under very loose assumptions on the curve $(\mu_t)_t$ or the vector field $(v_t)_t$. They proved, in particular, that the vector field $(v_t)_t$ is uniquely determined if $(\mu_t)_t$ is absolutely continuous of order 2 (see ``$\mathcal{AC}_2$'' in \S\ref{subsec:ac_deux}). They showed also that a process minimizing the action, for prescribed marginals $\mu_t$, exists, by using limits of solutions of mollified versions of \eqref{eq:continuity_intro}. Almost every trajectory of the process is in fact solution of the Cauchy problem $\dot{X}_t=v_t(X_t)$. In a further work \cite{Lis}, Lisini studied the $\mathcal{AC}_2$ curves of probability measures on a metric space. In this context where the continuity equation is not defined, he also proved that the action can be minimized. 

\medskip

Now here is the link with our work: In both the results by Ambrosio-Gigli-Savar\'e and Lisini, no statement is given on the \emph{uniqueness} of the minimizing process $(X_t)_t$. But on $\R$, the Markov-quantile process turns out to be a minimizing process, which yields a canonical minimizer. That notion depends of course on the chosen criterion that makes it canonical:  for instance, the quantile process is a minimizer and can also be considered canonical. Our criterion and its interest are as follows. In this context where $(\mu_t)_t\in{\cal AC}_2([0,1],\R^d)$, i.e.\  has finite energy, using Theorems A and B gives rise to the two following results when $d=1$. The first one is a slightly enhanced version of Theorem \ref{them:d} given in the introduction.\maz\smallskip

\point\label{p1:intro5} Theorem \ref{them:action} makes explicit under which assumptions and in which sense $\mq$ is a canonical minimizer of the action. The existence of such a minimizer, in any dimension $d$, is classical, and our work adds a uniqueness result when $d=1$, under the assumption that it is {\em Markov}, and obtained as a limit of products of couplings.\smallskip

\point\label{p2:intro5} Theorem \ref{them:commun_d_et_1} obtains the process $\mq$, which is a minimizer of ${\cal A}$, as a limit of interpolating processes belonging to $\disp_{R_n}$ (see Definition \ref{defi:disp}) instead of the limit of $(\como_{[R_n]})_n$ as in Theorem B. Using limits of interpolating processes is the classical way to obtain minimizers (see  \cite[Chapter 7]{Vi2}, \cite{Lis}) in any dimension $d$, so this places our work within this context. The interest of doing it is that then, our process (that exists for $d=1$) satisfies a uniqueness property that makes sense for any $d$. It opens the question of the existence and uniqueness of a minimizer satisfying it, for any $d$, i.e.\ of a counterpart of the Markov-Quantile process in any metric space ---see Open question \ref{ss:unique}.\smallskip

In particular, point \ref{p1:intro5} shows that $\mq$ is concentrated on absolutely continuous curves $\gamma$. In point \ref{p2:intro5}, notice that in the general geodesic Polish metric space, the notions from Optimal Transport Theory yield good extensions of the quantile notions but, whereas a quantile catenation does not make sense, the (Markov) catenation defined in \S\ref{subsec:caten} does, as well as the Markov property. That is why the Markov-quantile process appears to be a better canonical minimizing process (open to generalizations), than the quantile process.\smallskip

In \S\ref{subsec:ac_deux}--\ref {subsec:continuity} we introduce the notions we need, which are mostly classical, and prove the propositions leading to Theorems \ref{them:action} and \ref{them:commun_d_et_1}. In \S\ref{subsec:them5} we state and prove them. As told above, in this section we give specific results in dimension $d=1$, inside a general framework making sense in all dimension $d$. Hence we work in $\R^d$ everywhere this makes sense.

\begin{notation}\label{nota:courbes_continues}In this section $(\mu_t)_{t}$ is still a family of probability measures on $\R$ or $\R^d$; in this \S\ref{sec:cont_eq} we index it by $[0,1]$; $(\XX,d)$ denotes some metric space ---the related notions will be used with $\XX=\R^d$ or $\XX=\p_2(\R^d)$ introduced below--- and $\mathcal{C}([0,1],\mathcal{X})$, or briefly $\mathcal{C}$ in case $\XX=\R$, the space of continuous curves from $[0,1]$ to $\XX$, with the $\sigma$-algebra induced by the topology of\/ \mbox{$\|\cdot\|_\infty$}. Instead of\/ $\ma((\mu_t)_t)$, we work here on the set\/ $\ma_\mathcal{C}((\mu_t)_t)$ of real probability measures on $\mathcal{C}$ (or $\mathcal{C}([0,1],\R^d)$, according to the context) with marginal $\mu_t$ for every $t\in [0,1]$.\end{notation}

\begin{rem}\label{rem:concentrated}To any $\Gamma\in\ma_\mathcal{C}((\mu_t)_t)$ corresponds $\Gamma'\in\ma((\mu_t)_t)$ defined by $\Gamma'(B)=\Gamma(B\cap\mathcal{C}([0,1],\R^d))$ for any $B$ in the cylindrical $\sigma$-algebra of $(\R^d)^{[0,1]}$. Notice that for any dense countable set $D$ of $[0,1]$, $\Gamma'((\pr^D)^{-1}\linebreak[1](\mathcal{C}(D,\R^d)))=1$. Conversely, suppose that some $Q\in\ma((\mu_t)_t)$ satisfies this property. Then we say that $Q$ is ``concentrated on $\mathcal{C}((\mu_t)_t)$'' and there is a unique $\Gamma_Q\in\ma_\mathcal{C}((\mu_t)_t)$ such that $\Gamma_Q'=Q$. So by a slight abuse, we will not distinguish $\Gamma$ and $\Gamma'$ or $Q$ and $\Gamma_Q$. For $\Gamma\in\ma_\mathcal{C}((\mu_t)_t)$ and $R$ a finite subset of $\R$, this gives sense, e.g.,  to $\Gamma_{[R]}$ after Definition \ref{defi:quantile_discretement_markovien}.
\end{rem}

\subsection{Absolutely continuous curves of order 2 and the Wasserstein distance}\label{subsec:ac_deux} Recall Convention \ref{conv:croissant}: ``$R=\{r_1,\ldots,r_m\}$'' implies $r_1<\ldots<r_m$.

\label{defi:parti}
\begin{defi}A {\em partition} of an interval $[a,b]$ is a finite subset $R=\{r_0,\ldots,r_{m+1}\}$ of $[a,b]$ with $(r_0,r_{m+1})=(a,b)$. We denote the set of partitions of $[a,b]$ by $\parti([a,b])$. The {\em mesh} $|R|$ of $R$ is $\max_{k=0}^m|r_{k+1}-r_k|$.
\end{defi}

\label{nota:length}
\begin{reminder-notation}\maz\ Let $\gamma$ be a curve in $\mathcal{C}([0,1],\mathcal{X})$. \point\ For $0\leq a<b\leq 1$ the (possibly infinite) length of $\gamma$ on $[a,b]$ is defined as $\Var_a^b(\gamma)=\sup_{R\in \parti([a,b])} \sum_{k=0}^{m} d(\gamma(r_k),\gamma(r_{k+1}))$, where $R=\{r_0,r_1,\ldots,r_m,r_{m+1}\}$.

\point\ The curve $\gamma$ is said to be absolutely continuous if for every $\delta>0$ there exists $\eps$ such that for any family of intervals $\op]a_k,b_k\clo[$ satisfying $\sum(b_k-a_k)\leq \eps$ it holds $\sum d(\gamma(a_k),\gamma(b_k))\leq \delta$. We denote by $\mathcal{AC}([0,1],\mathcal{X})$ the space of such curves. As explained for instance in \cite{AGS}, where the definition is slightly different but equivalent, these curves admit for almost every $t$ a metric derivative which we denote by $|\dot\gamma|(t)$:
$$|\dot\gamma|(t)=\lim_{h\to 0}\frac{d(\gamma(t+h)-\gamma(t))}{h}$$
(if $\mathcal{X}=\R^n$ and $\gamma$ is derivable this is $|\dot\gamma(t)|$, so the notation is consistent). Then $\Var_a^b(\gamma)=\int_a^b|\dot\gamma|(t)$ and $\Var_a^b(\gamma)$ coincides with the total variation of $\gamma$ on $[a,b]$. Equivalent definitions of absolutely continuous curves are that $t\mapsto\Var_0^t(\gamma)$ is absolutely continuous, or that there exists an integrable function $m:[0,1]\to \R^+$ such that $d(\gamma(a),\gamma(b))\leq \int_a^b m\,\dd\la$ for every $a<b$.

\point We also introduce the space $\mathcal{AC}_2([0,1],\mathcal{X})\subset\mathcal{AC}([0,1],\mathcal{X})$ of {\em absolutely continuous curves $\gamma$ of order two}, i.e.\ such that $\int_0^1|\dot{\gamma}|^2<+\infty$. Notice that Lipschitzian curves are absolutely continuous of order two.
 \end{reminder-notation}

Now we introduce the notion of energy and the subsequent Proposition \ref{pro:relax}, which seems classical but for which we could not find any reference in the literature. Similar results concerning the length, in particular for geodesic curves, can be found in \cite{AT, AGS}. We will consider them as known.

\begin{defi}\label{defi:energie}Let $\gamma$ be a mapping from $[0,1]$ to a metric space $(\mathcal{X}, d)$. For $0\leq a<b\leq 1$ the energy $\EE_a^b(\gamma)$ of $\gamma$ on $[a,b]$ is defined as:
\begin{align}\label{eq:defenergy} 
\EE_a^b(\gamma)=\sup_{R\in\parti([a,b])} \EE_a^b(\gamma, R),\ \text{where:}
\end{align}
$$\EE_a^b(\gamma,\{r_0,\ldots,r_{m+1}\})=\sum_{k=0}^{m}d(\gamma(r_{k}),\gamma(r_{k+1}))^2/(r_{k+1}-r_k).$$
\end{defi}

\noindent{\em Note.} For $[a,b]=[0,1]$ we may denote $\EE_0^1$ and $\Var_0^1$ by $\EE$ and $\Var$.

\begin{pro}\label{pro:relax}\maz
Let $\gamma$ be a mapping from $[0,1]$ to $\mathcal{X}$. Then:\smallskip

\point\label{item:pro:relaxA} If $\EE(\gamma)<\infty$ then $\gamma$ is continuous, i.e.\ $\mathcal{AC}_2\subset{\cal C}$.\smallskip

\point\label{item:pro:relaxB}\  {\bf(i)} If a partition $R'\in\parti([a,b])$ is finer than $R$, $\EE_a^b(\gamma,R)\leq\EE_a^b(\gamma,R')$. {\bf(ii)} If $\gamma$ is continuous, the limit $\lim_{|R|\to 0} \EE(\gamma,R)$ is well-defined and equals $\EE(\gamma)$. {\bf(iii)} $\EE(\gamma)$ is finite if and only if $\gamma\in \mathcal{AC}_2([0,1],\mathcal{X})$; in this case $\EE_a^b(\gamma)=\int_a^b|\dot{\gamma}|^2(t)\dd t$ for all $a<b$.\smallskip

\point\   $\EE(\gamma)$ is lower semi-continuous for the  uniform convergence.
\end{pro}
\begin{proof} {\bf(a)} If $\EE(\gamma)<\infty$, there is a bound $M$ such that for any $s$ and $t>s$, $\frac{d(\gamma(s),\gamma(t))^2}{|s-t|}\leqslant M$, i.e.\ $d(\gamma(s),\gamma(t))\leqslant M\sqrt{|s-t|}$, which gives the result.\medskip

{\bf(b)}{\bf(i)} This follows from the fact that, for $\alpha, \beta>0$ and $a+b\geqslant c$:
$$\frac{1}{\alpha}a^2+\frac{1}{\beta}b^2\geqslant \frac1{\alpha+\beta}c^2,$$
itself given by the inequality $\bigl(a\sqrt{\beta/\alpha}-b\sqrt{\alpha/\beta}\bigr)^2\geqslant0$.\medskip

{\bf(ii)} We treat {(ii)} in the case $\EE(\gamma)<\infty$, letting the reader adapt the details in the case $\EE(\gamma)=\infty$. Take $\eps>0$ and $R=\{r_0,\ldots,r_{m+1}\}\in\parti([0,1])$ such that $\EE(\gamma,R)\geqslant\EE(\gamma)-\eps$. Set $\alpha=\frac{|R|}{3}$, so that if $R'=\{r'_0,\ldots,r'_{m'+1}\}\in\parti([0,1])$ and $|R'|\leq\alpha$ then for all $k\in\{0,\ldots,m\}$, $\sharp\{i\in\N:\,r'_i\in[r_k,r_{k+1}]\}\geqslant 2$. For any $R'\in\parti([0,1])$ such that $|R'|\leq\alpha$, we denote $(\min(R'\cap[r_k,r_{k+1}]),\max(R'\cap[r_k,r_{k+1}]))$ by ($r^+_k,r^-_{k+1})$. Since $\lim_{(s,t)\rightarrow(r_k,r_{k+1})}\frac{d(\gamma(s),\gamma(t))^2}{t-s}=\frac{d(\gamma(r_k),\gamma(r_{k+1}))^2}{r_{k+1}-r_k}$ for all $k$, there is an $\alpha_1>0$ such that $|R'|\leq\min(\alpha,\alpha_1)$ ensures the second inequality below, hence {(ii)}:
\begin{equation}
\EE(\gamma,\subd')\geq\sum_{k=0}^{m} d(\gamma(r_k^+)-\gamma(r_{k+1}^-))^2/(r_k^+-r_{k+1}^-)\geq \EE(\gamma, \subd)-\eps\geq \EE(\gamma)-2\eps.\nonumber
\end{equation}

{\bf(iii)} Notice that a similar argument as above gives the Chasles relation $\EE_a^c(\gamma)=\EE_a^b(\gamma)+\EE_b^c(\gamma)$ for $a<b<c$. Then we proceed in three steps.\smallskip

-- First, $\EE_0^1(\gamma)<\infty$ implies that $t\mapsto\Var_0^t(\gamma)$ is absolutely continuous, i.e.\ $\gamma $ is. By contradiction, assume that $\EE_0^1(\gamma)<\infty$ and that for some $\eps>0$ and every $\delta>0$, there exists disjoint intervals $[a_k,b_k]$ with $\sum (b_k-a_k)\leqslant \delta$ and $\sum \Var_{a_k}^{b_k}(\gamma)>\eps$. Take now $\delta<\eps^2/2\EE(\gamma)$. The convexity of the scalar square gives that  $\Var_a^b(\gamma)^2\leqslant (b-a)\EE_a^b(\gamma)$, hence, together with the Chasles relation, the last inequality in (\ref{eq:absolument_continu_absurd}) below. The second inequality of \eqref{eq:absolument_continu_absurd} is the Cauchy-Schwarz inequality. Since (\ref{eq:absolument_continu_absurd}) is a contradiction, we are done.
\begin{align}\label{eq:absolument_continu_absurd}
\eps<\sum \Var_{a_n}^{b_n}(\gamma)\leqslant \sqrt{\sum \Var_{a_n}^{b_n}(\gamma)^2/(b_n-a_n)\sum (b_n-a_n)}<\eps/\sqrt2.\smallskip
\end{align}

-- Now $\gamma\in \mathcal{AC}_2$ $\Rightarrow$ $\EE(\gamma)<\infty$. Indeed, $|\gamma(b)-\gamma(a)|^2/(b-a)\leqslant (\int_a^b |\dot\gamma|)^2/(b-a)\leqslant \int_a^b |\dot\gamma|^2$, so if $\gamma\in \mathcal{AC}_2$, $\int_0^1 |\dot\gamma|^2<\infty$ so that $\EE(\gamma)\leqslant \int_0^1 |\dot\gamma|^2<\infty$.\smallskip

-- Finally suppose that $\EE(\gamma)<\infty$. Then  $\gamma\in \mathcal{AC}_2([0,1],\mathcal{X})$. Indeed, we showed above that $\gamma\in\cal{AC}$. Now take $\eps>0$ and $h\in\op]0,\eps]$. For all $h$ let $n$ be an integer such that $(n+1)h\geq 1-\eps$. Then:
$$\begin{array}{r@{\,\,\leqslant\,\,}l}\displaystyle\int_0^{1-\eps}\biggl(\underbrace{\frac{d(\gamma(t+h)-\gamma(t))}{h}}_{a_h}\biggr)^2\dd t&\displaystyle\frac1h\int_0^h\sum_{i=0}^n\frac{d(\gamma(ih+t),\gamma((i+1)h+t)^2}{h}\dd t\\[-1ex]
&\displaystyle\EE(\gamma).
\end{array}$$
Now, since $\gamma\in\mathcal{AC}$, $|\dot\gamma|$ is almost surely defined, hence $\liminf_{h\rightarrow0}{a_h}=|\dot\gamma|$. By the Fatou lemma we get:
$$\int_0^{1-\eps} |\dot\gamma|^2(t)\leqslant \EE(\gamma).$$
This holds for every $\eps$, so that $\gamma \in \mathcal{AC}_2([0,1],\mathcal{X})$ and that the announced formula $\int_0^{1} |\dot\gamma|^2(t)=\EE(\gamma)$ is satisfied.\medskip

{\bf(c)} This holds since $\EE$ is a supremum of functions continuous on $\mathcal{C}$ for the uniform topology on $\mathcal{C}([0,1],\mathcal{X})$.
\end{proof}

\begin{reminder}\label{reminder:wasserstein}
On $\p(\R^d)^2$ the following infimum (minimum by the Prokhorov Theorem) has all the properties of a distance except that it may be infinite; it is called the $2$-Wasserstein distance:
\begin{equation}\label{eq:def_w}
W_2(\mu,\nu)=\min_{P\in \ma(\mu,\nu)}\sqrt{\int \|y-x\|^2\dd P(x,y)}.
\end{equation}
On the Wasserstein space $\mathcal{P}_2(\R^d)=\{\mu\in\mathcal{P}(\R^d)|\,\int \|x\|^2\,\dd \mu(x)<\infty\}$, $W_2$ is finite, thus is a true distance. Consequently, if $(\mu_t)_t\in\mathcal{C}([0,1],\p_2(\R^d))$, then $\Var_a^b((\mu_t)_t,R)$ and $\EE_a^b((\mu_t)_t,R)$ are finite for every $R\in \parti([a,b])$. 

A minimizer $P$ of \eqref{eq:def_w} is called an \emph{optimal transport plan} between $\mu$ and $\nu$. If $d=1$ and $W_2(\mu,\nu)<\infty$ the quantile coupling $\como(\mu,\nu)$ introduced in \S\ref{ss:minimality} is the unique optimal transport plan, see for instance \cite{RR1}. Therefore, for the quantile process $\como((\mu_t)_t)\in\ma((\mu_t)_t)$:
$$W_2(\mu_s,\mu_t)=\sqrt{\int |y-x|^2\,\dd \como^{s,t}(x,y)}.$$
\end{reminder}

\subsection{Action -- expected energy of a random curve}\label{subsec:action}

\begin{notation}Now $\mu$ denotes a family $(\mu_t)_{t\in[0,1]}$ of probability measures on $\R^d$.
\end{notation}

\begin{defi}\label{defi:action} If $\Gamma\in\p((\R^d)^\R)$ is concentrated (see Remark \ref{rem:concentrated}) on $\mathcal{C}([0,1],\linebreak[1]\R^d)$ its action $\A(\Gamma)$ is defined as:
$$\A(\Gamma)=\int_{\mathcal{C}}\EE(\gamma)\dd \Gamma(\gamma).$$
\end{defi}

Now a series of natural remarks leads to wonder about the behaviour of the measure $P$ of Theorem \ref{them:a} in this framework. Proposition \ref{pro:ap_egale_emu} gives it.

\begin{rem}\label{rem:about_action} {\bf(a)} If $\A(\Gamma)<+\infty$, $\Gamma$ is in fact concentrated on $\mathcal{AC}_2$.

{\bf(b)} If $\Gamma$ is a measure on $\mathcal{C}$, e.g., an element of $\ma_\mathcal{C}(\mu)$, then:
\begin{align}\label{eq:action_energy}
\A(\Gamma)&\eqdef\int_{\mathcal{C}}\lim_{|R|\to 0}\EE(\gamma,R)\,\dd\Gamma(\gamma)=\lim_{|R|\to 0}\int_{\mathcal{C}}\EE(\gamma,R)\,\dd\Gamma(\gamma)
\end{align}
because of the monotone convergence theorem: use a monotone sequence of partitions and Proposition \ref{pro:relax}\ref{item:pro:relaxB}.

{\bf (c)} If $\Gamma\in\ma_\mathcal{C}(\mu)$, then:
\begin{equation}\label{eq:a_superieur_a_e}
\A(\Gamma)\geqslant \EE(\mu).
\end{equation}
In case $d=1$, this is an equality if (but in general not only if) $\Gamma=\como(\mu)$. Indeed:
\begin{align}
\int_{\mathcal{C}}\EE(\gamma,R)\,\dd\Gamma(\gamma)&=\int_\mathcal{C} \sum_{k=1}^m\|\gamma(r_{k})-\gamma(r_{k+1})\|^2/(r_{k+1}-r_k)\,\dd \Gamma(\gamma)\nonumber\\\label{ici}
&= \sum_{k=1}^m \left(\int_\mathcal{C} \|\gamma(r_{k})-\gamma(r_{k+1})\|^2/(r_{k+1}-r_k)\,\dd \Gamma(\gamma)\right)\\ &\geqslant \sum_{k=1}^m W_2(\mu_{r_{k}},\mu_{r_{k+1}})^2/(r_{k+1}-r_k)=\EE(\mu,R). \nonumber
\end{align}
The inequality comes from the fact that $(\pr^{r_k,r_{k+1}})_\#\Gamma$ is in $\ma(\mu_{r_k},\mu_{r_{k+1}})$, so that $\int_\mathcal{C} \|\gamma(r_k)-\gamma(r_{k+1})\|^2\,\dd \Gamma(\gamma)\geqslant W_2(\mu_{r_k},\mu_{r_{k+1}})^2$ with equality, when $d=1$, if $(\pr^{r_k,r_{k+1}})_\#\Gamma=\como(\mu_{r_k},\mu_{r_{k+1}})$.
Now, thanks to (\ref{eq:action_energy}), when $|R|$ tends to 0 this provides $\A(\Gamma)\geqslant \EE(\mu)$, with the announced equality case.

{\bf(d)} If equality occurs in {(c)} for some measure $\Gamma$, it holds also for any $\Gamma_{[R]}$ introduced in Definition \ref{defi:quantile_discretement_markovien} ---hence, if $d=1$, for the measures $\como_{[R]}(\mu)$. Indeed, consider \eqref{ici} only for partitions finer than $R$ (by Proposition \ref{pro:relax}\ref{item:pro:relaxB}(i), the minimum in \eqref{eq:defenergy} remains the same): you get that $\A(\Gamma)=\A(\Gamma_{[R]})$.
\end{rem}

Remark \ref{rem:about_action} {(d)} ``passes to the limit'' when $(R_n)_n$ is such that $\como_{[R_n]}(\mu)\underset{n\rightarrow\infty}{\longrightarrow} P$, where $P\in\ma_{{\cal C}}(\mu)$ is the measure given by Theorem \ref{them:a}.

\begin{pro}\label{pro:ap_egale_emu}
The Markov-quantile process $\mq\in \ma(\mu)$ satisfies $\A(\mq)=\EE(\mu)$. Moreover for every $(R_n)_n$ as in Theorem \ref{them:b}, $(\como_{[R_n]})_n$ tends to $\mq$ in $\ma_\mathcal{C}(\mu)\subset\p(\mathcal{C})$.
\end{pro}

\begin{proof}
We need the following classical claim.\smallskip

\noindent{\em Claim.} If $f:(\mathcal{C},\|\cdot\|_\infty)\rightarrow \R^+$ is lower semi-continuous, then $F:\p(\mathcal{C})\rightarrow\R$ defined by $F(\Gamma)=\int_{\mathcal{C}}f\dd\Gamma$ is also lower semi-continuous.

To check it, take $(\Gamma_n)_{n\in\N^\ast}\in\p(\mathcal{C})^{\N^\ast}$ tending weakly to some $\Gamma_0$. Then $\liminf_n F(\Gamma_n)\geqslant F(\Gamma_0)$ by Lemma 4.3 of \cite{Vi2}, the wanted result.\smallskip

The claim and Proposition \ref{pro:relax} (c) give that the action $\A:\p(\mathcal{C})\rightarrow \R$ is lower semi-continuous. Now take $(R_n)_{n\in\N^\ast}$ a sequence of finite subsets of $\R$ as given by Theorem \ref{them:b} and set $\Gamma_n\eqdef\como_{[R_n]}(\mu)$. If we show that $\Gamma_n$ converges weakly to $\mq$ in $\p(\mathcal{C})$, we will get that $\A(\mq)\leqslant \liminf_n \A(\Gamma_n)$, hence the result since $\EE(\mu)\leqslant \A(\mq)$ by Remark \ref{rem:about_action} (c) and $\A(\Gamma_n)=\EE(\mu)$ for all $n$ by Remark \ref{rem:about_action} (d). So let us show this. By the Chebyshev inequality, for every $\eps$ there exists $\alpha>0$ such that, for all $n\in\N^\ast$:
$$\Gamma_n(\{\gamma\in \mathcal{C}:\,\EE(\gamma)>\alpha\})< \eps\AND \Gamma_n(\{\gamma\in \mathcal{C}:\,|\gamma(0)|>\alpha\})< \eps.$$
Therefore $\mathcal{N}\eqdef\{\gamma\in \mathcal{C}:\,\EE(\gamma)\leqslant\alpha\}\cap \{\gamma\in \mathcal{C}:\,|\gamma(0)|\leqslant \alpha\}$ has $\Gamma_n$-mass greater than $1-2\eps$ for all $n$. It follows from its definition that on $\mathcal{N}$, $\int_0^1|\dot\gamma|^2$ and thus also $\int_0^1|\gamma|^2$ are bounded, hence $\mathcal{N}$ is included in a ball of the Sobolev space $\mathcal{W}^{1,2}([0,1])$. This Banach space is compactly embedded in $\mathcal{C}$, see \cite[Theorem 8.8]{Brez}, so that $\mathcal{N}$ is relatively compact in $\mathcal{C}$. So according to the Prokhorov theorem any subsequence of $(\Gamma_n)_n$ has a (weak) limit point. But by Theorem \ref{them:b}, each finite marginal of $(\Gamma_n)_n$ tends weakly to the corresponding marginal of $\mq$, hence all such limit point must be $\mq$, hence $(\Gamma_n)_n$ tends weakly to $\mq$.
\end{proof}

\subsection{The continuity equation}\label{subsec:continuity}

\begin{notation} For all $t\in[0,1]$, $\pr^t$ is the projection $(\R^d)^{[0,1]}\rightarrow(\R^d)^{\{t\}}=\R^d$, i.e.\ $\pr^t(\gamma)=\gamma(t)$. On $\mathcal{AC}([0,1],\R^d)$ we also define $\pr_2^t$ by $\pr_2^t(\gamma)=\dot{\gamma}(t)$ on the set where $\dot{\gamma}$ is defined and $\pr_2^t(\gamma)=0$ on its (null) complement.
\end{notation}

As defined in \cite[Definition 5.4.2]{AGS} we introduce the barycentric projection.

\begin{defi}\label{defi:barycentre} Take $\Gamma\in\p(\mathcal{AC}([0,1],\R^d)$ and for all $t\in[0,1]$ denote $(\pr^t)_\#\Gamma$ by $\mu_t$, $(\pr^t\times \pr_2^t)_\#\Gamma$ by $M_t$ and by $\kappa_t$ a kernel such that $M_t=\join(\mu_t,\kappa_t)$. The barycentric projection of $M_t$ is the $\mu_t$-almost surely defined vector field $u^\Gamma_t$ on $\R^d$ such that $u^\Gamma_t(x)$ is  the barycentre of $\kappa_t(x,\,.\,)$. Alternatively, it is defined by the equation:
\begin{align}\label{eq:bary}
\int\langle v,u_t^\Gamma\rangle(x) \dd\mu_t(x)=\int \langle v(x),u\rangle \dd M_t(x,u)=\int \langle v(\gamma_t), \dot{\gamma_t} \rangle \dd \Gamma(\gamma),
\end{align}
for every continuous bounded vector field $v$.
\end{defi}

\begin{reminder}If $(\mu_t)_t=(f_t\la_{\R^d})_t$ is a family of measures on $\R^d$ with density $(x,t)\mapsto f_t(x)$ smooth with compact support, a smooth vector field $v_t$ transports the measure $\mu_t$, in the sense that its flow $\Phi^t$ makes $\mu_t(\Phi^t(B))$ constant for any Borel set $B$, if and only if $v_t$ satisfies the continuity equation:
\begin{equation}\label{eq:continuity}
\partial_t \mu_t+\mathrm{div}_{\mu_t} (v_t)=0\ \ \text{(or $\partial_t \mu_t+\mathrm{div}(v_t \mu_t)=0$, see below),}
\end{equation}
$\operatorname{div}_{\mu_t}(v_t)$ standing for the signed measure $\mathcal{L}_{v_t}\mu_t$, where $\mathcal L$ is the Lie derivative. Now \eqref{eq:continuity} keeps a weak meaning in $\R^d\times[0,1]$ in our framework, namely:
\begin{align}\label{eq:continuity_deux}
\int_0^1\int_{\R^d}\left(\partial_t\varphi(x,t)+\langle v_t(x), \nabla_x \varphi(x,t)\rangle\right)\dd \mu_t(x)\,\dd t=0
\end{align}
for every smooth function $\varphi:\R^d\times [0,1]\rightarrow\R$ with compact support in $\R^d\times\op]0,1\clo[$. In \eqref{eq:continuity}, $\operatorname{div}_{\nu}v$ depends only on the product $v\nu$, so may be written $\mathrm{div}(v\nu)$. Indeed, for $g,h\in{\mathcal{C}}^\infty(\R^d)$, $\operatorname{div}_{g\nu}(hv)=(\dd(gh).v)\nu+gh\operatorname{div}_\nu(v)$.
\end{reminder}

\begin{reminder}\label{reminder:important}An important result proved in \cite{AGS} (see Theorem 8.2.1) is that for every solution $(\mu_t,v_t)$ of \eqref{eq:continuity} with $\int_0^1\int |v_t|^2\,\dd \mu_t\,\dd t<\infty$ there exists $\Gamma$ with: $$\dot{\gamma}(t)=v_t$$ $\Gamma\otimes \lambda$-almost surely. In particular $\int_0^1 \int |v_t|^2\,\dd \mu_t\,\dd t=\A(\Gamma)$ and therefore:
\begin{equation}\label{eq:int_v_superieure_a_e}
\int_0^1 \int |v_t|^2\,\dd \mu_t\,\dd t\geqslant \EE(\mu).
\end{equation}
Notice that, unlike for Lipschitz ODE, $\Gamma$ is not unique in general.
\end{reminder}

\begin{pro}\label{pro:barycentre}\maz
\point\ Let $\Gamma$ be a probability measure on $\mathcal{AC}([0,1],\R^d)$ such that ${\cal A}(\Gamma)<\infty$ and denote $(\pr^t)_\#\Gamma$ by $\mu_t$. Then $(\mu_t,u_t^\Gamma)_{t\in [0,1]}$ (see Definition \ref{defi:barycentre})  satisfies the continuity equation \eqref{eq:continuity_deux}.

\point\ \label{pt2:bary} If moreover $\Gamma$ minimizes $\A(\Gamma)$ on $\ma((\mu_t)_{t\in [0,1]})$, then:
$$\dot{\gamma}(t)=u_t^\Gamma,$$
$\Gamma\otimes \la$-almost surely. In particular $\Gamma$ is concentrated on integral curves of the time-dependent vector field $u_t^\Gamma$.
\end{pro}

\begin{proof}

{\bf(a)} We have to show that: $\int_0^1\!\int\partial_t\varphi(x,t)+\nabla_x\varphi(x,t).u_t^\Gamma\dd\mu_t\dd t=0$, with $\varphi$ as in \eqref{eq:continuity_deux}. Let $M$ be an upper bound for $\|\nabla_x\varphi\|$. We first consider $F:(\gamma,t)\mapsto \partial_t\bigl(\varphi(\gamma(t), t)\bigr)=\partial_t\varphi(\gamma,t)+\nabla_x\varphi(\gamma,t).\dot\gamma$, then $\|F\|^2$ is bounded by $2M^2(1+\|\dot\gamma(t)\|^2)$ and has integral on $\mathcal{C}\times [0,1]$ bounded by $2M^2(1+\A(\Gamma))<\infty$. Thus $F\in L^2(\Gamma\otimes \la)$. Integrating firstly with respect to $t$ and secondly with respect to $\Gamma$, we see that $\iint F(\gamma,t)=0$. If we now use the Fubini theorem, with \eqref{eq:bary} we obtain the wished equality.

{\bf(b)} Take $\kappa_t$ the kernel given in Definition \ref{defi:barycentre}, then for almost all $(x,t)$, $\int\|v\|^2\dd\kappa_t(x,v)\geqslant\|u_t^\Gamma(x)\|^2$. But by (a), $(\mu_t,u_t^\Gamma)$ satisfies \eqref{eq:continuity} hence by Reminder \ref{reminder:important}, $\int_{0}^1 \int \|u_t^\Gamma\|^2\,\dd \mu_t\, \dd t\geqslant \EE(\mu)$. This gives the  inequality below:
$$\A(\Gamma)=\int_0^1\int \|\dot \gamma\|^2(t) \,\dd\Gamma(\gamma)\,\dd t
=\int_0^1\int\int\|v\|^2\kappa_t(x,\dd v)\,\dd\mu_t(x)\,\dd t
\geqslant \EE(\mu).$$
Now if $\A(\Gamma)$ is minimal, i.e.\ $\A(\Gamma)=\EE(\mu)$, all the inequalities above are equalities, which ensures that $\dot\gamma=u_t^\Gamma$ almost surely.
\end{proof}

\subsection{Our resulting theorems on \mb$\mq$\mn\ as a minimizer in this context}\label{subsec:them5}

The following theorem gathers:\smallskip

-- well-known facts, actually true on any $\p_2(\R^d)$ for $d\geqslant 1$, namely \ref{pt:eul} and \ref{pt:lag}\ref{item1:pt:lag}, i.e. the existence of measures $\Gamma$ for which   \eqref{eq:a_superieur_a_e} is an equality,\smallskip

-- enhancements of them following from our theorems A and B and Proposition \ref{pro:ap_egale_emu}, notably the uniqueness in the {\em Lagrangian} statement for $d=1$ ---the uniqueness of the field $v_t$ in \ref{pt:eul} is classical, but remind that it does not imply that of the minimizing process $\Gamma$ tangent to it.\smallskip

\noindent Notice that it is not known whether the process can be chosen Markov for $d\geqslant2$ (see Open question \ref{ss:unique}). Moreover $\como_{[R_n]}$ is only defined for $d=1$.

\begin{them}[Existence and uniqueness of representations]\label{them:action}\maz
Take a curve $\mu=\linebreak[1](\mu_t)_{t\in[0,1]}$ in Wasserstein space $\mathcal{P}_2(\R)$ with finite energy $\EE(\mu)$. Then:\smallskip

\point\label{pt:eul} {\em (Eulerian statement.)} There exists a vector field $v_t$ satisfying the continuity equation \eqref{eq:continuity} and such that Inequality \eqref{eq:int_v_superieure_a_e}:
$$\int_0^1\int|v_t|^2\,\dd \mu_t\,\dd t\geqslant \EE(\mu)$$
is an equality. This vector field is unique.\smallskip

\point\label{pt:lag} {\em (Lagrangian statement.)} There exists $\Gamma\in\ma_\mathcal{C}(\mu)$ such such that:
\begin{itemize}
\item[\souspoint\label{item1:pt:lag}] Inequality \eqref{eq:a_superieur_a_e}:
$\A(\Gamma)\geqslant \EE(\mu)$
is an equality,
\item[\souspoint] the measure $\Gamma$ is Markov,
\item[\souspoint] it is the limit in $\p(\mathcal{C})$ of a sequence $(\como_{[R_n]})_n$.
\end{itemize}
Such a $\Gamma$ is unique in $\ma_\mathcal{C}(\mu)$; it is the Markov-quantile process $\mq$.\smallskip

\point\label{pt:ens} {\em (Link between them.)} For any $\Gamma$ minimizing the action, i.e.\ making \eqref{eq:a_superieur_a_e} an equality, the curve $\gamma\in \mathcal{C}$ is $\Gamma$-almost surely a solution of the ODE:
$$\dot\gamma(t)=v_t(\gamma_t),$$
for almost every time. 
\end{them}
\begin{proof} {\bf\ref{pt:eul}} With $u_t^\Gamma$ given by Definition \ref{defi:barycentre}, note that $\A(\Gamma)=\int^1_0\int |u_t^\Gamma|^2\dd \mu_t\,\dd t$ for every $\Gamma$, so that Proposition \ref{pro:barycentre} gives the existence of the field. Its uniqueness comes from a standard argument: if $u_t$ and $v_t$ satisfy \eqref{eq:continuity}, so does $w_t\eqdef(u_t+v_t)/2$, but if they both make \eqref{eq:int_v_superieure_a_e} an equality and differ on a non-null subset, $\int_0^1\int|w_t|^2\,\dd \mu_t\,\dd t<\EE(\mu)$, which contradicts \eqref{eq:int_v_superieure_a_e}.

{\bf\ref{pt:lag}} Proposition \ref{pro:ap_egale_emu} shows that $\Gamma=\mq$ suits. By Theorem \ref{them:b}, the conditions of Theorem \ref{them:action}(b) characterize the Markov-quantile process, which ensures the uniqueness.

{\bf\ref{pt:ens}} Use Proposition \ref{pro:barycentre}(a) and the uniqueness in \ref{pt:eul}.
\end{proof}

In Definition \ref{defi:disp}, remember that an optimal transport is defined in Reminder \ref{reminder:wasserstein}.

\begin{defi}\label{defi:disp}\maz
Let $R=\{r_0,r_1,\ldots,r_m,r_{m+1}\}$ be a partition in $\parti([0,1])$. We denote by $\disp_{R}$ the set of measures $M\in\p(\mathcal{C})$ that are \emph{dynamical transports made Markov at the points of $R$, and linearly} (hence in fact optimally) \emph{interpolating $(\mu_t)_{t\in [0,1]}$ between them}, defined as follows.\smallskip

\point\label{p1defi:disp} For each $i\in\{0,\ldots, m\}$, the coupling $M^{r_i,r_{i+1}}\in\ma(\mu_{r_i},\mu_{r_{i+1}})$ is an optimal transport plan between $\mu_{r_i}$ and $\mu_{r_{i+1}}$,\smallskip

\point\ for $\{\lambda_1,\ldots,\lambda_n\}\subset[0,1]$ and $i^\lambda:(x,y)\in (\R^d)^2\mapsto \lambda y+(1-\lambda)x$, we have:
$$(i^{\lambda_1},\ldots, i^{\lambda_n})_\#M^{r_i,r_{i+1}}=M^{\lambda_1 r_i+(1-\lambda_1)r_{i+1},\ldots,\lambda_n r_i+(1-\lambda_n)r_{i+1}},$$

\point\ for all finite $S$ containing $\{r_1,\ldots,r_m\}$,
$$(\pr^{S})_\# M= M^{s^0_1,\ldots, s^0_{n_0},r_1}\circ M^{r_1,s^1_{1},\ldots,s^1_{n_1},r_2}\circ\ldots\circ M^{r_m,s^m_{1},\ldots,s^m_{n_m}},$$
where $S=\{s^0_1,\ldots, s^0_{n_0},r_1,s^1_{1},\ldots,s^1_{n_1},r_2,\ldots,r_m,s^m_{1},\ldots,s^m_{n_m}\}$ and where the first and/or last terms disappear if $n_0$ and/or $n_m$ is null. 
\end{defi}

\begin{rem}\label{rem:singleton}Note that $\#\disp_{R}=1$ if and only if each set $\ma(\mu_{r_i},\mu_{r_{i+1}})$, appearing in \ref{p1defi:disp}, contains a unique optimal transport. It is the case when $d=1$, where $\ma(\mu_{r_i},\mu_{r_{i+1}})=\{\como(\mu_{r_i},\mu_{r_{i+1}})\}$, see Reminder \ref{reminder:wasserstein}.
\end{rem}

\begin{them}\label{them:commun_d_et_1}
Let $d$ be a positive integer and $\mu=(\mu_t)_{t\in[0,1]}$ a curve of finite energy in $\p_2(\R^d)$. For every nested sequence $(R_n)_n$ with $R_\infty\eqdef\cup_n R_n$ dense in $[0,1]$, and $\Gamma_n\in \disp_{R_n}$ for all $n\in \N$, there exists $\Gamma\in \ma_{\mathcal{C}}(\mu)$ that is the limit in $\p(\mathcal{C}([0,1],\R^d))$ of a subsequence of $(\Gamma_n)_n$. Moreover for every $\Gamma$ obtained in this way the action $\A(\Gamma)$ is minimal, i.e.\ such that Inequality \eqref{eq:a_superieur_a_e} is an equality.

Moreover, in dimension $d=1$, a Markov limit $\Gamma$ exists and if a limit $\Gamma$ is Markov, it is the Markov-quantile measure in $\ma_\mathcal{C}((\mu_t)_{t\in [0,1]})$.
\end{them}

\begin{proof}
Adapting \cite[Chapter 7]{Vi2}, \cite{Lis} or Proposition \ref{pro:ap_egale_emu} to our context we obtain the first part of the theorem for every $d\geq 1$. This requires slight modifications that we do not detail: Villani's chapter is in fact written for geodesic curves $(\mu_t)_t$ between prescribed $\mu_0$ and $\mu_1$ whereas Lisini's processes are attached to curves $(\mu_t)_{t\in [0,1]}$ of finite energy but the processes of the sequence are constant on each interval between two consecutive points of the partition, whereas ours is linear. Note, as an indication, that our measures $\Gamma_n$ minimize ${\cal A}$ in $\{\Gamma\in \p(\mathcal{C}([0,1],\R^d)):\,\forall r\in R_n,\,\Gamma^r=\mu_r\}$, the minimum being $\A(\Gamma_n)=\EE(\mu,R_n)$.

In case $d=1$, take as before a nested sequence $(R_n)_n$ such that $\como_{[R_n]}$ converges to $
\mq$ in $\p(\C)$. Up to taking a subsequence, the same sequence of partitions permits $\Gamma_n\in\disp_{R_n}$ to converge to some $\Gamma$. By Definitions \ref{defi:quantile_discretement_markovien} and \ref{defi:disp}, for every $S\subset R_n$ the measure $(\pr^S)_\#\Gamma_n$ coincides with $(\pr^S)_\#\como_{[R_n]}$ and:
\[
(\pr^S)_\#\Gamma=(\pr^S)_\#\mq.
\]
As $R_\infty$ is dense in $[0,1]$ and the measures are concentrated on $\mathcal{C}$ it follows that $\Gamma=\mq$. This proves the existence part in case $d=1$

For the uniqueness statement, take as before a nested sequence $(R_n)_n$ and let $\Gamma_n$ be the single element of $\disp_{R_n}$ (see Remark \ref{rem:singleton}). Assume that $(\Gamma_n)_n$ has a Markov limit $\Gamma$. By Definitions \ref{defi:quantile_discretement_markovien} and \ref{defi:disp}, for every $S\subset R_n$ the measure $(\pr^S)_\#\Gamma_n$ coincides with $(\pr^S)_\#\como_{[R_n]}$. Using the same argument as for Proposition \ref{pro:ap_egale_emu}, up to taking a subsequence, $(\como_{[R_n]})_n$ converges to an element of $\ma_\C(\mu)$ that we denote by $\Gamma'$. Hence for every $S\subset R_\infty$,
\[
(\pr^S)_\#\Gamma=(\pr^S)_\#\Gamma'.
\]
As $R_\infty$ is dense in $[0,1]$ and the measures are concentrated on $\mathcal{C}$ it follows $\Gamma'=\Gamma$. Note now that for every $n\in \N$, the measure $\como_{[R_n]}$ has increasing kernel, so that it also holds for $\Gamma'$; similarly $\mdec$ is stabilised by  $\como_{[R_n]}$ and $\Gamma'$ (Remark \ref{rem:preserver_mdec_ferme} and Lemma \ref{lem:transitions_croissantes_limite}). Finally $\Gamma'$ is a process satisfying \ref{item:markov} and \ref{item:transitions} of Theorem \ref{them:a}. For $s<t$ we have, on the one hand, $\Gamma^{s,t}\geqlc \mq^{s,t}$ because the Markov-quantile measure is minimal (Theorem \ref{them:a}\ref{item:minimal}). On the other hand, for every $s<t$ in $R_\infty$ we have $\mq^{s,t}\geqlc \Gamma^{s,t}$ because $(\Gamma')^{s,t}$ is a limit of products of quantile couplings, and $\mq^{s,t}$ is defined in Proposition \ref{pro:comp_conv} as a supremum in this class for $\leqlc$. Therefore the Markov processes $\Gamma=\Gamma'$ and $\mq$ have the same law on $R_\infty$, hence coincide as measures on $\C$. 
\end{proof}

\begin{rem}
We have seen that both $\como\in \ma(\mu)$ and $\mq\in\ma(\mu)$ minimize ${\cal A}$. The quantile measure $\como$ is also the optimizer of another multi-marginal transportation problem raised by Brendan Pass in \cite{Pass}. It minimizes:
$$\Gamma\in \ma(\mu_t)\mapsto \int_{\mathcal{C}([0,1],\R)}\left(\varphi \left(\int_0^1\gamma(t)\,\dd t \right)\right)\dd\Gamma(\gamma)$$
where $\varphi$ is strictly convex and is the unique minimizer. Some assumptions are required (see \cite[Part 2]{Pass}) but some of them can probably be relaxed.  The paper is based on the fact that $P\in\ma(\mu_1,\ldots,\mu_d)\mapsto \int\,\varphi(x_1+\cdots+x_d)\,\dd P(x_1,\ldots,x_d)$ is minimized by $\como(\mu_1,\ldots,\mu_d)$.
\end{rem}

\section{Examples and open questions}\label{sec:exemples}
\subsection{Example of Markov-quantile processes attached to discrete measures on $\N$}

In this section $x_+$ is the positive part $\max\{0,x\}$ of any $x\in\R$.

\begin{ex}[Discrete measures]\label{ex:integers}
Let $(\mu_t)_{t\in [0,1]}$ be concentrated on $\N$ for every $t$ and assume that for every $k\in \N$ the map $A_k:t\mapsto \sum_{i=0}^k \mu_t(i)$ is in $\mathcal{C}^1([0,1])$ and piecewise monotone (e.g., $A_k$ is analytic). Let moreover $A_{-1}$ be the zero constant function. We assume that:
$$\max\left(\frac{-A'_k(t)}{\mu_t(k)},\frac{A'_{k-1}(t)}{\mu_t(k)}\right)$$
is bounded from above for $(t,k)\in[0,1]\times \mathbb{N}$. Then, using the characterization of the Markov-quantile process as a limit of quantile couplings, namely Theorem \ref{them:a}\ref{item:limite_de_compose}, it can be proved that the Markov-quantile process $(X_t)_{t\in [0,1]}$ is the time continuous Markov chain with jump rate $q_{k,k+1}=\frac{(-A'_k(t))_+}{\mu_t(k)}$ from $k$ to $k+1$, and $q_{k,k-1}=\frac{(A'_{k-1}(t))_+}{\mu_t(k)}$ from $k$ to $k-1$ and $q_{k,j}=0$ for $|j-k|\neq 1$. Denoting $\P(X_t=k)$ by $p_k$ it means that the so-called forward Kolmogorov--Chapman system is satisfied:
\begin{align*}
\frac{\dd p_k}{\dd t}(t)=
\begin{cases}
p_{1}q_{1,0}-p_0q_{0,1}&\text{if }k=0,\\
p_{k+1}q_{k+1,k}+p_{k-1}q_{k-1,k}-p_k(q_{k,k-1}+q_{k,k+1})&\text{if }k\in \N^*,\
\end{cases}
\end{align*}
where the derivative is a right derivative. Recall that the jump rate is defined for $i\neq j$ by:
\[
q_{i,j}(t)=\lim_{h\to 0^+}\frac{\P(X_{t+h}=j|\,X_t=i)}{h}.
\]
The classical theory that can be read in Feller's book \cite[Chapter XVII, section 9]{Fe_t1} and the references therein (see also \cite{Doe}) ensures that our process is solution of the forward Kolmogorov--Chapman system. The uniqueness of the solution for a Markov process is obtained from the uniform bound on the rates $q_{i,j}(t)$.

In place of a complete proof let us compute the jump rate in a typical case. Notice before that similar computations can be found in \cite[Section 4]{Ju_ejp}. We are looking for the jump rate $q_{k,k+1}(t)$ in the case of $A:=A_{k-1}$ and $B:=A_k$ locally decreasing on the right of $t$. At every time $t$ the atomic measure $\mu_t$ is completely described by the partition of the interval $[0,1]$ of quantile levels through the sequence $(A_k(t))_{k\in \N}$. Indeed, $]A_{k-1}(t),A_k(t)[\subset [0,1]$ is the interval
 of the quantile levels of the atom $\mu_t(k)\delta_k$. Recall that both $A$ and $B$ are in $\mathcal{C}^1([0,1])$. We can assume that $h$ is so small that $B(t+h)>A(t)$. For $\{r',r''\}\subset[t,t+h]$ with $r'<r''$ and $r''$ close to $r'$, the quantile coupling between $\mu_{r'}$ and $\mu_{r''}$ transports the main part of the mass of the atom $\mu_{r'}(\{k\})\delta_k$ on itself and the rest on the atoms $\mu_{r''}(\{k'\})\delta_{k'}$ with $k'>k$. We aim at proving that the conditional probability to be still in $k$ at time $t+h$ is:
\begin{align}\label{eq:grando}
1-\frac{B'(t)}{(B-A)(t)}h+O(h^2).
\end{align}
Since the probability to jump more than twice is $O(h^2)$, \eqref{eq:grando} furnishes the announced jump rate $q_{k,k+1}(t)=(-A'_k(t))_+/\mu_t(k)$ in the case of decreasing functions. So let us prove \eqref{eq:grando}.

We consider a partition $R=\{r_0,\ldots,r_{m}\}$ of $[t,t+h]$ with $(r_0,r_{m})=(t,t+h)$ and the discrete quantile Markov chain associated with it. As $A$ and $B$ are decreasing, note that no mass can leave the quantile level interval $[A,B]$ and come back on it in the same time interval $[t,t+h]$. On $[r_n,r_{n+1}]$ the probability to stay in the interval is $\frac{B(r_{n+1})-A(r_n)}{B(r_n)-A(r_n)}=1-\frac{B(r_n)-B(r_{n+1})}{B(r_n)-A(r_n)}$ as one can easily convince oneself with a picture similar to the left part of Figure \ref{figure}. We let the proof of the following fact to the reader: there exists $\delta=O(h)$ such that for every $r'<r''$ in $[t,t+h]$ we have:
$$\e^{\left(-(1+\delta)\frac{B'(t)(r''-r')}{(B-A)(t)}\right)}\leq \frac{B(r'')-A(r')}{B(r')-A(r')}\leq \e^{\left(-(1-\delta)\frac{B'(t)(r''-r')}{(B-A)(t)}\right)}.$$
We obtain this estimate for each interval $[r_n,r_{n+1}]\subset [t,t+h]$. Multiplying all together, we see that the probability to stay on the same state after $m$ steps is in $\bigl[\exp\bigl(-(1+\delta)\frac{B'(t)}{(B-A)(t)}h \bigr),\exp\bigl(-(1-\delta)\frac{B'(t)}{(B-A)(t)}h \bigr)\bigr]$, where $\delta=O(h)$. A simple Taylor expansion gives \eqref{eq:grando}.\end{ex}

\begin{ex}[Poisson distributions]\label{ex:Poisson}
Elaborating on the last example we consider, for $t\in \R^+$, $\mu_t=\mathcal{P}(t)$ where $\mathcal{P}(t)$ is the Poisson law of parameter $t$. In this case $A_k(t)=\sum_{i=0}^k\exp(-t)t^i/i!$ so that the jump rate $q_{k,k+1}(t)$ is constantly $1$ for every $k$ and $t$, and the other rates are zero. We recover the Poisson process. Note that the Poisson laws are in stochastic order, which matches with the increasing trajectories of the Poisson counting process.
\end{ex}

\begin{ex}[Binomial distributions]\label{ex:Binomial}
In this example $\mu_t=\mathcal{B}(n,t)$ for $t\in [0,1]$. Let us define a Markov process $X=(X_t)_{t\in[0,1]}\in\ma((\mu_t)_t)$ and compute its jump rates; we will then see that $\law(X)=\mq$. We define $X$ on the probability space $[0,1]^n$ by $X:(\alpha_1,\ldots,\alpha_n)\mapsto\sum_{k=0}^n \one_{\alpha_k\in[0,t]}$, so its law is $\mu_t$. The fact that $(X_t)_{t\in[0,1]}$  is Markov comes from the following coarse argument: provided $k$ coordinates of $\alpha=(\alpha_1,\ldots,\alpha_n)$ are smaller than $t$, the distribution is uniform on $[0,t]^k$ for the $k$ coordinates of the past  of $t$ and on $[t,1]^{n-k}$ for the $n-k$ of its future. Between $t$ and $t+h$ the probability to have (at least, as well as exactly) one jump is $(n-k)\frac{h}{1-t}+O(h^2)$. As $A_k(t)=\sum_{i=0}^k {n \choose i} t^i(1-t)^{n-i}$ with the notation of Example \ref{ex:integers}, it can easily be checked that $\frac{(n-k)}{1-t}=\frac{A'_k(t)}{\mu_t(k)}$, which proves that $(X_t)_{t\in[0,1]}$ is the Markov-quantile process attached to $(\mu_t)_{t\in[0,1]}$. This example could be of interest with respect to previous works on the entropic interpolation on graphs as, e.g.,  \cite{Hil_ejp,Leo16}.
\end{ex}

\subsection{Example of Markov-quantile transport processes}

The following examples are related to \S\ref{sec:cont_eq}. In particular we will consider processes tangent to a non-autonomous vector field on $\R$. Basically, in the examples, $\mu_t$ is made of two parts that are translated in opposite directions and cross. We examine three crossing situations for atomic or diffuse measures.

\begin{ex}[One atom crossing a diffuse measure]\label{ex:strongmarkov}
Consider $\mu=(\mu_t)_{t\in [0,1]}$ with $\mu_t=\frac12 \lambda\lfloor_{[t-3/4,t-1/4]}+\frac12 \delta_{0}$. This is the family of marginals of a simple process $\Gamma$ with affine trajectories, defined by $\Gamma(t\mapsto 0)=1/2$ and $\Gamma(\{t\mapsto x_0+t:\,x_0\in A\})=\lambda\lfloor_{[-3/4,-1/4]}(A)$. This is not the Markov-quantile process attached to $\mu$ but it is a Markov process and it is tangent to the optimal vector field of Theorem \ref{them:action}\ref{pt:eul}, namely:
\begin{align*}
v_t(x)=
\begin{cases}
0&\text{if }x=0,\\
1&\text{otherwise,}
\end{cases}
\end{align*}
so the action $\A(\Gamma)$ equals the minimal value $\mathcal{E}(\mu)$.

The Markov-quantile process $(X_t)_{t\in [0,1]}$ attached to $(\mu_t)_{t\in[0,1]}$ can be described as follows: the trajectories start according to $\mu_0$ and are piecewise affine, with pieces taken from the affine curves above. Provided $X_0\in [-3/4,-1/4]$, the first piece is $X_t=X_0+t$ on $[0,\tau]$ where $-X_0=\tau$. The second affine piece is constant equal to zero on $[\tau,\min(\tau+\eta,1)]$ where $\eta$ is an exponential random variable of parameter $2$, independent from $X_0$. The third piece, if it exists, is affine of slope $1$, namely $X_t=t-(\tau+\eta)$ on $[\tau+\eta,1]$.

Unlike $\Gamma$, the process $(X_t)_{t\in [0,1]}$ has increasing kernels and is a \emph{strongly} Markov process.
\end{ex}

\begin{ex}[Crossing of two purely atomic measures]\label{ex:rationnal}
Consider two measures $\alpha$ and $\beta$ of mass $1/2$, concentrated on the rational numbers of $[0,1]$, with finite or infinite support. Let $\tau_t$ be the translation of vector $t$ in $\R$. Set $\mu=(\mu_t)_{t\in\R}=((\tau_t)_\#\alpha+(\tau_{-t})_\#\beta)_{t\in\R}$. As in Example \ref{ex:strongmarkov} the measure $\Gamma\in \ma(\mu)$ is concentrated on the space of piecewise affine paths (of slopes $1$ and $-1$) is a minimizer of the action. The two measures $(\tau_t)_\#\alpha$ and $(\tau_{-t})_\#\beta$  are concentrated on $\Q$ when $t\in \Q$ and they are singular if $t\notin \Q$. Hence according to Proposition \ref{pro:barycentre} the optimal vector field $(v_t)_{t\in [0,1]}$ satisfies $\Gamma\otimes \lambda\lfloor_{[0,1]}$-almost surely $v_t=\pm 1$. It can be checked that the Markov-quantile process is again piecewise affine with a random finite number of changes of slope. Interesting exercises on the Markov-quantile process can be considered, as for instance finding the probability for a trajectory coming from $-\infty$ in $-\infty$ to tend to $+\infty$ in $+\infty$. Note that the situation seems to be well approached by truncating the measure to finitely many `big' atoms. This corresponds to the case of $\alpha$ and $\beta$ with finite support. In this particular case the above mentioned exercise reduces to the so-called `gladiator game' \cite{KLN} that is a stochastic version of Borel's Blotto game \cite{RSY}.
\end{ex}

\begin{ex}[Crossing of two diffuse measures]\label{ex:continuous}
Consider $\mu_t=\la\lfloor_{[t-2,t-1]}+\linebreak[1]\la\lfloor_{[1-t,2-t]}$ and again $\Gamma$ such that $\Gamma(\{t\mapsto t+x_0:\,x_0\in A\})=\la\lfloor_{[-2,-1]}(A)$ and $\Gamma(\{t\mapsto x_0-t:\,x_0\in A\})=\la\lfloor_{[1,2]}(A)$. Unlike in the previous examples, $\Gamma$ does not minimize $\A$ on $\ma_\C(\mu)$. All the measures $\mu_t$ are continuous so that the Markov-quantile process $(X_t)_{t\in \R}$ is the quantile process. It is affine by part and continuous. With probability $1/2$, in fact if $X_0\leq 0$, first it has slope $1$, then slope $0$ on $[\frac{1-X_0}{2},\frac{5+X_0}{2}]$ and finally slope $-1$. If $X_0>0$, the process $(X_t)_{t}$ starts with slope $-1$, is flat on $[\frac{1+X_0}{2},\frac{5-X_0}{2}]$ and continues with slope $1$ after $\frac{5-X_0}{2}$.
\end{ex}

\subsection{Theoretic Markov-quantile processes}

\begin{ex}[One atom with regular level functions]\label{ex:swimming_atom}
Take $(\mu_t)_{t\in{[0,1]}}$ such\linebreak[4] that for every $t$, $\mu_t$ has exactly one atom $x_t\in \R$ and the interval of quantile levels of this atom at time $t$ is $\op]A(t),B(t)\clo[$. Assume moreover that $A$ and $B$ are of class $\mathcal{C}^1$ and piecewise monotone. Then the Markov-quantile process $(X_t)_{t\in [0,1]}$ can be described using two Poisson point processes of jump rates $(A')_+/(B-A)$ and $(B')_-/(B-A)$. Conditionally on $F_{\mu_t}(X_t)\in \op]A(t), B(t)\clo[$, we have $X_t=x_t$ until the next time $t_0\geq t$ in the point process. Then the process $(X_t)_t$ leaves $x_t$ and starts a piece of quantile trajectory constant in the space $[0,1]$ of quantile levels with value $A(t_0)$ or $B(t_0)$. The process may hit again $x_t$ if there exists some $t_1>t_0$ with $A(t_1)=x_{t_0}$, or $B(t_1)=x_{t_0}$ respectively.
\end{ex}

The next remark is of general interest and particularly significant with respect to Example \ref{rem:a_bouger}. It presents the Markov-quantile process as one end of the spectrum of processes of law in $\ma(\mu)$ that satisfies \ref{item:transitions} of Theorem \ref{them:a}, i.e.\ have increasing kernels, the other end of which is the independent process.

\begin{rem}\label{rem:a_bouger}
The minimality condition \ref{item:minimal} of Theorem \ref{them:a} satisfied by the Markov-quantile process $(X_t)_t$ attached to some $(\mu_t)_t$ can also be stated as follows. For every process $(Y_t)_{t\in \R}$ satisfying \ref{item:markov} and \ref{item:transitions} of Theorem \ref{them:a}, for every $s<t$ and every $x\in \R$ it holds:
\[
\law(X_t|\,X_s\leq x)\leqs \law(Y_t|\,Y_s\leq x).
\]
A similar relation that concerns maxima of $\leqs$ in place of minima is satisfied by the independent process $(Z_t)_{t\in \R}$. If a process $(Y_t)_t$ has increasing kernels, we obtain:
$$\law(Y_t|\,Y_s\leq x)\leqs \law(Y_t|\,Y_s< +\infty)=\law(Y_t)=\law(Z_t|\,Z_s\leq x).$$
We conclude with the following result: Assume that for some $s<t$ and $(X_t)_t$ the Markov-quantile process, $X_s$ is independent of $X_t$. Then for any process $(Y_t)_t$ satisfying \ref{item:markov} and \ref{item:transitions} of Theorem \ref{them:a}, we have for every $x\in\R$
\[
\mu_t=\law(X_t|\,X_s\leq x)\leqs \law(Y_t|\,Y_s\leq x) \leqs \law(Z_t|\,Z_s\leq x)=\mu_t
\]
so that, due to the Markov property, for every $s'\leq s$ and $t'\geq t$, $Y_{s'}$ and $Y_{t'}$ are independent.
\end{rem}

\begin{ex}[Two atoms]\label{ex:two_atoms}
We set $\mu_t=a(t)\delta_0+b(t)\delta_1$ with $a+b=1$ but do not assume any regularity on the functions $a$ and $b$. Let $(X_t)_t$ be the Markov-quantile process. We shall show that $X_s$ and $X_t$ are independent if and only if the total variation of $a$ (or $b$) on $[s,t]$ is infinite or $m_{s,t}\eqdef \min(\inf_{[s,t]}a,\inf_{[s,t]} b)=0$. If $m_{s,t}=0$ the independence is true for any Markov process. Indeed, by assumption, for any $\eps>0$, we may take $r$ such that $\P(X_r=0)$ is small enough so that $\P(X_r=0|X_s=0)\leqslant\eps$, $\P(X_r=1|X_s=0)\geqslant1-\eps$, and $|\P(X_t=0|X_r=1)-\P(X_t=0)|\leqslant\eps$. Then:
$$\begin{array}{r@{\,\,}c@{\,\,}l}
|\P(X_t=0|X_s=0)-\P(X_t=0)|&=&\bigl|\P(X_t=0|X_r=0)\P(X_r=0|X_s=0)\\[.7ex]
\multicolumn{3}{l}{\quad+\P(X_t=0|X_r=1)\P(X_r=1|X_s=0)-\P(X_t=0)\bigr|\ \ \text{since $X$ is Markov}}\\[.7ex]
&\leqslant&|\P(X_t=0|X_r=1)-\P(X_t=0)|+2\eps\\[.7ex]
&\leqslant&3\eps,\\[.7ex]
\end{array}$$
which is the wanted independence.

Hence, we assume that $m_{s,t}>0$ and that $a$ takes values in $[m_{s,t},1-m_{s,t}]$ in $\mu_t=a(t)\delta_0+b(t)\delta_1$. We are left with the task to prove that independence is equivalent to a infinite total variation of $a$ on $[s,t]$. Let $\theta_0$ be the uniform measure on $[0,a(s)]$.  Our goal reduces to establishing $\la=\theta_0 \ell_{[s,t]}\eqdef\stosup_R \theta_0\ell_{r_1}\ell_{r_2}\cdots \ell_{r_m}$ where $R$ ranges among the partitions $\{r_0,\ldots, r_m\}$ with $(r_0,r_m)=(s,t)$. For the measures under consideration, if $a(r_{k-1})\leq a(r_k)\leq a(r_{k+1})$ or $a(r_{k-1})\geq a(r_k)\geq a(r_{k+1})$ it holds $\ell_{r_{k-1}} \ell_k \ell_{k+1}=\ell_{r_{k-1}} \ell_{r_{k+1}}$. Therefore we can assume without loss of generality that the sequence $(a(r_k))_{k=0,\ldots,m}$ has increments with alternating sign, for instance $a(r_{2k+1})\geq a(r_{2k})$ for every $k$. We define $\theta_n=\theta_0 \ell_{r_0}\cdots \ell_{r_n}$.

The measure $\theta_n$ can be written in the form:
\[
\theta_n=d_n\la_{[0,a(r_n)]}+d'_n\la_{[a(r_n),1]}\in \mdec
\]
where $d_n=a(r_n)^{-1}\theta_n([0,a(r_n)])$ in fact parametrizes the complete measure. Note, after Remark \ref{rem:a_bouger} that $\theta_n\leqs \la$, which means $d_n\geq 1\geq d'_n$. As $a(r_n)\in[m_{s,t}, 1-m_{s,t}]$, the sequence converges to $\la$ if and only if $d_n\to 1$.

Recalling the effect of the kernel $\ell_{r_{n+1}}$, described on Figure \ref{figure2} and defined in Notation \ref{nota:Asx}\ref{p:l_R_fini}, we find:
\begin{align*}
d_{n+1}-1=
\begin{cases}
(d_{n}-1)\frac{a(r_{n+1})}{a(r_n)}&\text{if }a(r_{n+1})<a(r_n),\\
(d_n-1)\frac{b(r_{n+1})}{b(r_n)}&\text{otherwise.}
\end{cases}
\end{align*}
The product $\Pi_{n=1}^m\min(\frac{a(r_{n+1})}{a(r_n)},\frac{1-a(r_{n+1})}{1-a(r_n)})$ can be arbitrarily close to zero (over all partitions of $[s,t]$) if and only if $a\in[m_{s,t},1-m_{s,t}]$ has infinite total variation. This proves the claimed equivalence.
\end{ex}

\begin{ex}[One atom on the lower levels]\label{ex:exponential}
Consider $(\mu_t)_{t\in{[0,1]}}$ such that for every $t$, $\mu_t$ has exactly one atom and this atom is between the quantile levels $A(t)=0$ and $B(t)$. An example is $\mu_t=B(t)\delta_0+(1-B(t))\mathcal{E}(1)$ where $\delta_0$ is the Dirac mass in zero and $\mathcal{E}(1)$ the exponential law of parameter $1$. No regularity assumption is made on $B$. Similar observations as in Example \ref{ex:two_atoms} permit us to specify the kernel between time $s$ and $t>s$. Let $\alpha_{s,t}$ be $\sup_{r\in[s,t]}B(r)$. Then $L_{[s,t]}$ is simply the uniform measure of mass $\alpha_{s,t}$ on $[0,\alpha_{s,t}]^2$ plus the one-dimensional uniform measure of mass $1-\alpha_{s,t}$ on the diagonal between $(\alpha_{s,t},\alpha_{s,t})$ and $(1,1)$. The same for the kernel $\ell_{s,t}$ reads:
\[
\ell_{s,t}(x,\cdot)=
\begin{cases}
\alpha_{s,t}^{-1}\la\lfloor_{[0,\alpha_{s,t}]}&\text{if }x\leq \alpha_{s,t},\\
\delta_x&\text{if }x>\alpha_{s,t}.
\end{cases}
\]
A particle of quantile value $\leq \alpha_{s,t}$ at time $s$ is uniformly mapped at time $t$ on the particles of quantile levels $[0,\alpha_{s,t}]$. If the quantile value at time $s$ is greater that $\alpha_{s,t}$, the particle keeps on with the same level until time $t$ as if it were the quantile process. 
\end{ex}

\subsection{Transformations of Markov-quantile processes}

\begin{ex}[Markov-quantile processes]\label{ex:quantile}
According to $(\mu_t)_t$, a quantile process may be Markov or not. Recalling Remark \ref{rem:como_markovien} \ref{p2:rem:como_markovien}, if $\como$ is Markov it coincides with the Markov-quantile process. As proved in \cite[Proposition 3]{Ju_seminaire}, the criterion is the following: the process is \emph{not} Markov if and only if there exists $\alpha\neq \alpha'\in [0,1]$ and $t_1<t_2<t_3$ such that the $\alpha$-quantile and the $\alpha'$-quantile of $\mu_{t_2}$ are equal but that those of $\mu_{t_1}$ and $\mu_{t_3}$ differ. This can be summarized saying that ``X's are forbidden'' where X refers to the shape of the letter, the four ends being $G_{\mu_1}(\alpha)$, $G_{\mu_1}(\alpha')$, $G_{\mu_3}(\alpha)$ and $G_{\mu_3}(\alpha')$, the intersection being $G_{\mu_2}(\alpha)=G_{\mu_2}(\alpha')$. Other letters like O, Y and Z are allowed.
\end{ex}

\begin{rem}[Reversal of time and twist of space and time]\label{ex:reversed}
If $(X_t)_{t\in \R}$ is\linebreak[4] the Markov-quantile process attached to $(\mu_t)_{t\in\R}$ then $X_{-t}$ is the Markov-quantile process in $\ma((\mu_{-t})_{t\in \R})$. This comes from Theorem \ref{them:a} \ref{item:limite_de_compose} on the limit of compounds of quantile couplings and the fact that $\trans \como(\mu,\nu)=\como(\nu,\mu)$. More generally, for every homeomorphism $\varphi$ from $\R$ into $\R$, $X_{\varphi(t)}$ has law the Markov-quantile measure of $\ma((\mu_{\varphi(t)})_{t\in \R})$. Of course non-injective monotone continuous map $\varphi$ may be used too. 

Moreover if for all $t$, $f_t:\R\to \R$ are strictly monotone functions with the same orientation for all $t$, the process $f_t(X_t)$ is a Markov-quantile process.
\end{rem}

\begin{rem}
The Markov-quantile process being time-reversible, and as a coupling $P\in \ma(\mu,\nu)$ has increasing kernel if and only if $\trans P\in\mdec(\nu,\mu)$ and $P \leqlc Q$ is equivalent to $\trans P\leqlc \trans Q$ (recall Remark \ref{rem:stable} and Definition \ref{defi:geqlc}), points \ref{item:transitions} and \ref{item:minimal} of in Theorem \ref{them:a} can be replaced by {\bf (ii')}: For every $s<t$, $\mq^{s,t}\in\mdec(\mu_s,\mu_t)$, and {\bf (iii')}: for every $s<t$, $\mq^{s,t}$ is a minimal coupling for $\leqlc$ among the processes satisfying \ref{item:markov} and (ii').
\end{rem}

\subsection{Open questions}

\subsubsection{Markovinification}\label{ss:marko}
\maz
\point\ We may interpret the resulting process in Theorem \ref{them:sous_n} as the Markov process that has infinitesimally the same transitions as $P$. However, as it depends on the choice of the partitions, this Markov process is not a priori uniquely determined. At which conditions is this Markov process \emph{uniquely determined} and how can it be characterized?  If the initial process is a quantile process, we proved in Theorem \ref{them:a} that the answer is yes, without condition,  and characterized it using orderings.

\point\ In Theorem \ref{them:b} we proved that Markovinification of the quantile process occurs for processes $(\como_{[R_n]})_n$ in place of consistent (see Definition\ \ref{defi:consist}) transport plans. Does an \emph{existence} statement analogous to Theorem \ref{them:sous_n} happen when we consider sequences $(P_{[R_n]})_{n}$?.

\subsubsection{The Kellerer theorem in dimension $d\geq 2$}\label{ss:Kel_dimsup}
Kellerer proved Theorem \ref{them:kellerer_intro} for martingales in $\R$. For measures $(\mu_t)_t$ on $\R^d$, increasing in convex order, it is known that there exists an associated martingale \cite{HiRo13,BHS} but not whether one of them is Markov. This is a major question. Note that another interpretation of the question in higher dimension is possible when considering martingales indexed by a multidimensional set (see \cite[Problem 7b]{HPRY}). This problem was solved in \cite{Ju_seminaire}.

\subsubsection{Markov Kamae--Krengel theorem}\label{ss:marko_kk}
Kamae and Krengel proved in \cite{KaKr} that if $(\mu_t)_{t\in \R}$ are measures on a partially ordered Polish space $E$ such that $t\mapsto\mu_t$ is increasing for the stochastic order, in the sense that $t\mapsto \int f\dd\mu_t$ is increasing for any increasing bounded $f:E\mapsto \R$, there exists an increasing process $(X_t)_t$ with law in $\ma(\mu)$. We proved in Theorem \ref{them:a} and \ref{them:c} that if $E$ is $\R$, the process can moreover be Markov. A natural problem is whether this is also true for any $E$. 

\subsubsection{A Markov minimizer for the action in metric spaces}\label{ss:unique}
\maz
\point\ In part \ref{sec:cont_eq}, devoted to the continuity equation, Theorem \ref{them:commun_d_et_1} presents a uniqueness statement for the Markov-quantile process. The second part is for $d=1$ only but the statement makes sense in $\R^d$ for $d\geq 2$. Is it still true then? With Proposition \ref{pro:barycentre} \ref{pt2:bary} this would in particular mean that there exists a \emph{Markov} representation of the continuity equation.

\point\ The action $\A$ and energy $\EE$ are defined on metric spaces. Definition \ref{defi:disp} can also be extended to geodesic Polish metric spaces $\mathcal{X}$ in a natural way based on processes representing the geodesics of $\p_2(\mathcal{X})$ as in \cite[Corollary 7.22]{Vi2} or \cite[\S 2.2]{AGguide}. The first part of Theorem \ref{them:commun_d_et_1} is also true in this setting; we did not prove it to avoid technicalities. However, the question of the existence of an analogue to the Markov-quantile process on such a metric space seems us very interesting from an Optimal Transport perspective.

\subsubsection{Strong Markov property}\label{ss:strong}
The Markov-quantile process is not the unique process that minimizes the energy $\EE$ as is shown in Example \ref{ex:strongmarkov}. However the process in this example is strongly Markov. Is the Markov-quantile process strongly Markov? Does this characterize it?

\bibliographystyle{abbrv}

\begin{thebibliography}{}

\end{thebibliography}


\begin{thebibliography}{10}

\bibitem{Albin}
J.~M.~P. Albin.
\newblock A continuous non-{B}rownian motion martingale with {B}rownian motion
  marginal distributions.
\newblock {\em Statist. Probab. Lett.}, 78(6):682--686, 2008.

\bibitem{AGguide}
L.~Ambrosio and N.~Gigli.
\newblock A user's guide of optimal transport.
\newblock {\em preprint}, 2011.

\bibitem{AGS}
L.~Ambrosio, N.~Gigli, and G.~Savar{\'e}.
\newblock {\em Gradient flows in metric spaces and in the space of probability
  measures}.
\newblock Lectures in Mathematics ETH Z\"urich. Birkh\"auser Verlag, Basel,
  second edition, 2008.

\bibitem{AS}
L.~Ambrosio and G.~Savar{\'e}.
\newblock Gradient flows of probability measures.
\newblock In {\em Handbook of differential equations: evolutionary equations.
  {V}ol. {III}}, Handb. Differ. Equ., pages 1--136. Elsevier/North-Holland,
  Amsterdam, 2007.

\bibitem{AT}
L.~Ambrosio and P.~Tilli.
\newblock {\em Topics on analysis in metric spaces}, volume~25 of {\em Oxford
  Lecture Series in Mathematics and its Applications}.
\newblock Oxford University Press, Oxford, 2004.

\bibitem{Arn}
V.~Arnold.
\newblock Sur la g\'eom\'etrie diff\'erentielle des groupes de {L}ie de
  dimension infinie et ses applications \`a l'hydrodynamique des fluides
  parfaits.
\newblock {\em Ann. Inst. Fourier (Grenoble)}, 16(fasc. 1):319--361, 1966.

\bibitem{BHS}
M.~Beiglb\"ock, M.~Huesmann, and F.~Stebegg.
\newblock Root to {K}ellerer.
\newblock In {\em S\'eminaire de {P}robabilit\'es {XLVIII}}, volume 2168 of
  {\em Lecture Notes in Math.}, pages 1--12. Springer, Cham, 2016.

\bibitem{BJ}
M.~Beiglb{\"o}ck and N.~Juillet.
\newblock On a problem of optimal transport under marginal martingale
  constraints.
\newblock {\em Ann. Probab.}, 44(1):42--106, 2016.

\bibitem{BeJu16}
M.~Beiglb{\"o}ck and N.~Juillet.
\newblock Shadow couplings.
\newblock {\em Preprint}, 2016.

\bibitem{BeBr}
J.-D. Benamou and Y.~Brenier.
\newblock A numerical method for the optimal time-continuous mass transport
  problem and related problems.
\newblock In {\em Monge {A}mp\`ere equation: applications to geometry and
  optimization ({D}eerfield {B}each, {FL}, 1997)}, volume 226 of {\em Contemp.
  Math.}, pages 1--11. Amer. Math. Soc., Providence, RI, 1999.

\bibitem{Bernard}
P.~Bernard.
\newblock Young measures, superposition and transport.
\newblock {\em Indiana Univ. Math. J.}, 57(1):247--275, 2008.

\bibitem{Bi}
P.~Billingsley.
\newblock {\em Convergence of probability measures}.
\newblock Wiley Series in Probability and Statistics: Probability and
  Statistics. John Wiley \& Sons Inc., New York, second edition, 1999.
\newblock A Wiley-Interscience Publication.

\bibitem{Brez}
H.~Brezis.
\newblock {\em Functional analysis, {S}obolev spaces and partial differential
  equations}.
\newblock Universitext. Springer, New York, 2011.

\bibitem{Doe}
W.~{Doeblin}.
\newblock {Sur certains mouvements al\'eatoires discontinus.}
\newblock {\em {Skand. Aktuarietidskr.}}, 22:211--222, 1939.

\bibitem{Fe}
H.~Federer.
\newblock {\em Geometric measure theory}.
\newblock Die Grundlehren der mathematischen Wissenschaften, Band 153.
  Springer-Verlag New York Inc., New York, 1969.

\bibitem{Fe_t1}
W.~Feller.
\newblock {\em An introduction to probability theory and its applications.
  {V}ol. {I}}.
\newblock Third edition. John Wiley \& Sons, Inc., New York-London-Sydney,
  1968.

\bibitem{GH}
N.~Gigli and B.-X. Han.
\newblock The continuity equation on metric measure spaces.
\newblock {\em Calc. Var. Partial Differential Equations}, 53(1-2):149--177,
  2015.

\bibitem{Gy}
I.~Gy{\"o}ngy.
\newblock Mimicking the one-dimensional marginal distributions of processes
  having an {I}t\^o differential.
\newblock {\em Probab. Theory Relat. Fields}, 71(4):501--516, 1986.

\bibitem{HaKl}
K.~Hamza and F.~C. Klebaner.
\newblock A family of non-{G}aussian martingales with {G}aussian marginals.
\newblock {\em J. Appl. Math. Stoch. Anal.}, pages Art. ID 92723, 19, 2007.

\bibitem{HTT}
P.~{Henry-Labordere}, X.~{Tan}, and N.~{Touzi}.
\newblock {An Explicit Martingale Version of the One-dimensional Brenier's
  Theorem with Full Marginals Constraint}.
\newblock {\em preprint}, Feb. 2014.

\bibitem{Hil_ejp}
E.~Hillion.
\newblock {$W_{1,+}$}-interpolation of probability measures on graphs.
\newblock {\em Electron. J. Probab.}, 19:no. 92, 29, 2014.

\bibitem{HPRY}
F.~Hirsch, C.~Profeta, B.~Roynette, and M.~Yor.
\newblock {\em Peacocks and associated martingales, with explicit
  constructions}, volume~3 of {\em Bocconi \& Springer Series}.
\newblock Springer, Milan, 2011.

\bibitem{HiRo13}
F.~Hirsch and B.~Roynette.
\newblock On {$\Bbb R^d$}-valued peacocks.
\newblock {\em ESAIM Probab. Stat.}, 17:444--454, 2013.

\bibitem{HiRoYo14}
F.~Hirsch, B.~Roynette, and M.~Yor.
\newblock Kellerer's theorem revisited.
\newblock In Springer, editor, {\em Asymptotic Laws and Methods in Stochastics.
  Volume in Honour of Miklos Csorgo}, Fields Institute Communications Series,
  2014.

\bibitem{Hob16}
D.~Hobson.
\newblock Mimicking martingales.
\newblock {\em Ann. Appl. Probab.}, 26(4):2273--2303, 2016.

\bibitem{Hob_fake}
D.~G. Hobson.
\newblock Fake exponential {B}rownian motion.
\newblock {\em Statist. Probab. Lett.}, 83(10):2386--2390, 2013.

\bibitem{JKO}
R.~Jordan, D.~Kinderlehrer, and F.~Otto.
\newblock The variational formulation of the {F}okker-{P}lanck equation.
\newblock {\em SIAM J. Math. Anal.}, 29(1):1--17, 1998.

\bibitem{Ju_seminaire}
N.~Juillet.
\newblock Peacocks parametrised by a partially ordered set.
\newblock In {\em S\'eminaire de {P}robabilit\'es {XLVIII}}, volume 2168 of
  {\em Lecture Notes in Math.}, pages 13--32. Springer, Cham, 2016.

\bibitem{Ju_ejp}
N.~Juillet.
\newblock Martingales associated to peacocks using the curtain coupling.
\newblock {\em Electron. J. Probab.}, 23:Paper No. 8, 29, 2018.

\bibitem{Kal}
O.~Kallenberg.
\newblock {\em Foundations of modern probability}.
\newblock Probability and its Applications (New York). Springer-Verlag, New
  York, second edition, 2002.

\bibitem{KaKr}
T.~Kamae and U.~Krengel.
\newblock Stochastic partial ordering.
\newblock {\em Ann. Probab.}, 6(6):1044--1049 (1979), 1978.

\bibitem{KLN}
K.~S. Kaminsky, E.~M. Luks, and P.~I. Nelson.
\newblock Strategy, nontransitive dominance and the exponential distribution.
\newblock {\em Austral. J. Statist.}, 26(2):111--118, 1984.

\bibitem{Ke72}
H.~G. Kellerer.
\newblock Markov-{K}omposition und eine {A}nwendung auf {M}artingale.
\newblock {\em Math. Ann.}, 198:99--122, 1972.

\bibitem{Ke73}
H.~G. Kellerer.
\newblock Integraldarstellung von {D}ilationen.
\newblock In {\em Transactions of the {S}ixth {P}rague {C}onference on
  {I}nformation {T}heory, {S}tatistical {D}ecision {F}unctions, {R}andom
  {P}rocesses ({T}ech. {U}niv., {P}rague, 1971; dedicated to the memory of
  {A}nton\'\i n \v {S}pa\v cek)}, pages 341--374. Academia, Prague, 1973.

\bibitem{Ke86}
H.~G. {Kellerer}.
\newblock {Order conditioned independence of real random variables.}
\newblock {\em {Math. Ann.}}, 273:507--528, 1986.

\bibitem{Ke87}
H.~G. {Kellerer}.
\newblock {Markov property of point processes.}
\newblock {\em {Probab. Theory Relat. Fields}}, 76:71--80, 1987.

\bibitem{Leo16}
C.~L\'eonard.
\newblock Lazy random walks and optimal transport on graphs.
\newblock {\em Ann. Probab.}, 44(3):1864--1915, 2016.

\bibitem{Lis}
S.~Lisini.
\newblock Characterization of absolutely continuous curves in {W}asserstein
  spaces.
\newblock {\em Calc. Var. Partial Differential Equations}, 28(1):85--120, 2007.

\bibitem{Low2}
G.~{Lowther}.
\newblock {Fitting Martingales To Given Marginals}.
\newblock {\em ArXiv e-prints}, Aug. 2008.

\bibitem{Low}
G.~Lowther.
\newblock Limits of one-dimensional diffusions.
\newblock {\em Ann. Probab.}, 37(1):78--106, 2009.

\bibitem{MaYo02}
D.~B. Madan and M.~Yor.
\newblock Making {M}arkov martingales meet marginals: with explicit
  constructions.
\newblock {\em Bernoulli}, 8(4):509--536, 2002.

\bibitem{Naga}
M.~Nagasawa.
\newblock {\em Schr\"odinger equations and diffusion theory}, volume~86 of {\em
  Monographs in Mathematics}.
\newblock Birkh\"auser Verlag, Basel, 1993.

\bibitem{Ol}
Y.~Ollivier.
\newblock Ricci curvature of {M}arkov chains on metric spaces.
\newblock {\em J. Funct. Anal.}, 256(3):810--864, 2009.

\bibitem{Ott}
F.~Otto.
\newblock The geometry of dissipative evolution equations: the porous medium
  equation.
\newblock {\em Comm. Partial Differential Equations}, 26(1-2):101--174, 2001.

\bibitem{pages}
G.~Pag\`es.
\newblock Functional co-monotony of processes with applications to peacocks and
  barrier options.
\newblock In {\em S\'eminaire de {P}robabilit\'es {XLV}}, volume 2078 of {\em
  Lecture Notes in Math.}, pages 365--400. Springer, Cham, 2013.

\bibitem{Pass}
B.~Pass.
\newblock On a class of optimal transportation problems with infinitely many
  marginals.
\newblock {\em SIAM J. Math. Anal.}, 45(4):2557--2575, 2013.

\bibitem{RR1}
S.~T. Rachev and L.~R{\"u}schendorf.
\newblock {\em Mass transportation problems. {V}ol. {I}}.
\newblock Probability and its Applications (New York). Springer-Verlag, New
  York, 1998.
\newblock Theory.

\bibitem{RR2}
S.~T. Rachev and L.~R{\"u}schendorf.
\newblock {\em Mass transportation problems. {V}ol. {II}}.
\newblock Probability and its Applications (New York). Springer-Verlag, New
  York, 1998.
\newblock Applications.

\bibitem{RSY}
Y.~Rinott, M.~Scarsini, and Y.~Yu.
\newblock A {C}olonel {B}lotto gladiator game.
\newblock {\em Math. Oper. Res.}, 37(4):574--590, 2012.

\bibitem{SS}
M.~Shaked and J.~G. Shanthikumar.
\newblock {\em Stochastic orders}.
\newblock Springer Series in Statistics. Springer, New York, 2007.

\bibitem{ST}
E.~Stepanov and D.~Trevisan.
\newblock Three superposition principles: currents, continuity equations and
  curves of measures.
\newblock {\em J. Funct. Anal.}, 272(3):1044--1103, 2017.

\bibitem{St65}
V.~Strassen.
\newblock The existence of probability measures with given marginals.
\newblock {\em Ann. Math. Statist.}, 36:423--439, 1965.

\bibitem{Vi1}
C.~Villani.
\newblock {\em Topics in optimal transportation}, volume~58 of {\em Graduate
  Studies in Mathematics}.
\newblock American Mathematical Society, Providence, RI, 2003.

\bibitem{Vi2}
C.~Villani.
\newblock {\em Optimal transport}, volume 338 of {\em Grundlehren der
  Mathematischen Wissenschaften}.
\newblock Springer-Verlag, 2009.

\end{thebibliography}
\def\cprime{$'$}

\end{document}